\documentclass[reqno]{amsart}

\usepackage{amsfonts,amsthm,amsmath,amssymb,amscd,amsmath,epsf,latexsym,mathrsfs,tkz-euclide,float}
\usepackage{graphicx}
\newtheorem{theorem}{Theorem}
\newtheorem{proposition}[theorem]{Proposition}
\newtheorem{lemma}[theorem]{Lemma}
\newtheorem{definition}[theorem]{Definition}
\newtheorem{corollary}[theorem]{Corollary}
\newtheorem{example}[theorem]{Example}
\newtheorem{remark}[theorem]{Remark}

\newtheorem{conjecture}[theorem]{Conjecture}

\newcommand{\bR}{\mathbb{R}}
\newcommand{\bC}{\mathbb{C}}
\newcommand{\ee}{\end{equation}}
\newcommand {\al}{\alpha}

\newcommand {\ga}{\gamma}
\newcommand {\Ga}{\Gamma}
\newcommand{\A}{\mathcal A}

\newcommand{\La} {\Lambda}

\newcommand{\bN}{\mathbb N}

\newcommand{\bD}{\mathbb D}
\newcommand{\D}{\mathcal D}
\newcommand{\C}{\mathcal C}

\newcommand \E {\mathcal E}

\newcommand \prt {\partial}

\newcommand \supp {\rm{\text{supp}}}

\newcommand \diff {\mathop{}\!\mathrm{d}}

\newcommand \q {\mathfrak q}
\newcommand{\B}{\mathcal {BR}}
\newcommand{\R}{\mathcal R}

\newcommand\W{\mathcal W}

\newcommand \mC{\mathfrak C}

\newcommand \Sh{\Upsilon}
\newcommand{\lline}[1]{\overline{#1}}
\newcommand \minispace {\hspace{.05em}}

\makeatletter
\@namedef{subjclassname@2020}{\textup{2020} Mathematics Subject Classification}
\makeatother

\newtheorem{PROP}{Proposition}

\newtheorem{LEMM}{Lemma}

\newtheorem{CORR}{Corollary}

\begin{document}
\numberwithin{equation}{section}
\numberwithin{theorem}{section}

\title[Rodrigues' descendants of a polynomial and Boutroux curves]{Rodrigues' descendants of a polynomial and Boutroux curves}

\author[R.~B\o gvad]{Rikard B\o gvad}
\address{Department of Mathematics, Stockholm University, SE-106 91
Stockholm, Sweden}
\email {rikard@math.su.se }

\author[Ch.~H\"agg]{Christian H\"agg}
\address{Department of Mathematics, Stockholm University, SE-106 91
Stockholm, Sweden}
\email{hagg@math.su.se}

\author[B.~Shapiro]{Boris Shapiro}
\address{Department of Mathematics, Stockholm University, SE-106 91
Stockholm, Sweden}
\email {shapiro@math.su.se }

\date{\today}

\begin{abstract} %Consider a probability measure $\mu$ with compact support in $\mathbb C$, sample the measure, producing $n$ complex numbers $\xi_i,\ i=1,2,...,n$ and a polynomial $S_n=\prod_{i=1}^n(z-\xi_i)$. Then differentiate this polynomial $[\al n]$ times producing 
Motivated by the classical Rodrigues' formula, we study below the root asymptotic of the polynomial sequence 
$$\R_{[\al n],n,P}(z)=\frac{\diff^{[\al n]}P^n(z)}{\diff z^{[\al n]}}, n= 0,1,\dots$$
where ${P(z)}$ is a fixed univariate polynomial, $\al $ is a fixed positive number smaller than $\deg P$, and $[\al n]$ stands for the integer part of $\al n$. 

Our description of this asymptotic is expressed in terms of an explicit harmonic function uniquely determined by the plane rational curve  emerging from the application of  the saddle point method  to the integral representation  of the latter polynomials using Cauchy's formula for higher derivatives.   As a consequence of our method, we conclude that this curve is birationally equivalent to the zero locus of the bivariate algebraic equation satisfied by the Cauchy transform of the asymptotic root-counting measure for the latter polynomial sequence.  We show that this harmonic function is also associated with an abelian differential having only  purely imaginary periods and the latter plane curve belongs to the class of Boutroux curves initially introduced in \cite{Be,BM}. 
   As an additional relevant piece of information, we derive a linear ordinary differential equation satisfied by $\{\R_{[\al n],n,P}(z)\}$ as well as  higher derivatives of powers of more general functions. \end{abstract}

\maketitle

\hskip5cm{\vbox{\noindent platt och avintetgjord\\
sl\"apar jag nollan min\\
vid h{\aa}ret\\
in i o\"andlighet. \\
(Ur ``I grund och botten",\\ Majken Johansson, 195$6^1$)}}\footnote{To our chagrin, we were not able to find a professional English  translation of these  highly relevant for the present article  four lines written by  the well-known Swedish poet Majken Johansson.
% \cite{MaJo}. 
Therefore we include here our homemade intepretation: ``Flattened and downtrodden / I drag  my zero / by its hair / all the way to infinity.''}

\tableofcontents
\section{Introduction}\label{sec:introduction} 
Around 1816 (Benjamin) Olinde Rodrigues\footnote{Born in a Jewish family of sephardic origin in Bordeaux on October 6, 1795, O.~Rodrigues, thanks to Napoleon's measures ensuring equality of rights for different religious minorities, was able to attend Lyce\'e Imperial which he joined in 1808 at the age of 14. Besides his mathematical interests, he had another passion: banking and its usage for social purposes. He was a close friend and supporter of Saint-Simon and a very peculiar philanthropic figure with strong socialist undertones, see more details  in \cite{Alt} and \cite{AltOr}.}  discovered his famous formula 
\begin{equation}\label{eq:Leg}
 P_n(z)=\frac{1}{2^nn!} \frac{\diff ^n}{\diff z^n} \left((z^2-1)^n\right)
 \end{equation}
  for the Legendre polynomials which undoubtedly became a standard tool in the toolbox of classical orthogonal polynomials and special functions, see e.g. \cite{AbSt}. (Later this formula was also rediscovered by Sir J.~Ivory and C.~G.~Jacobi, see \cite{As}.)
  
%   \begin{figure}
%\begin{center}
%\includegraphics[scale=0.7]{Figures/OlindeRodrigues.jpeg}
%\end{center}

%\vskip 0.5cm

%\caption{ (Benjamin) Olinde Rodrigues. }
%\label{Rodrigues}
%\end{figure} 

 Among other properties, the $n$-th Legendre polynomial $P_n(z)$ satisfies the linear ordinary differential equation
\begin{equation}\label{eq:GeneralizedLegendreEq}
(1-z^2)y'' - 2zy'+n(n+1)y = 0,
\end{equation}
and the asymptotic of the zeros as $n\to \infty$ is described by classical results.
 
\subsection{Main Problem}
Imitating Rodrigues' approach, given a polynomial $P$ of degree $d\ge 1$, let us consider a double-indexed family of polynomials determined by the Rodrigues-like expression
 \begin{equation*}\label{eq:pnmPolynomial}
\R_{m,n,P}(z) :=\ \frac{\diff ^m}{\diff z^m}\left(P^n(z)\right),\; n=0,1,\dots \text{ and } m=0,1,\dots, nd. 
\end{equation*}
%$$left\{\frac{d ^{m}}{d  z^{m}}\left(P^n(z)\right)\right\}_{n=1}^\infty.$$
These polynomials which we below call \textit{Rodrigues' descendants} of $P$ were apparently for the first time considered by N. Cior\^{a}nescu in 1933 (see \cite{Ci}) where he, in particular, derived linear differential equations satisfied by them. In 1965, and, to the best of our knowledge, independently of N. Cior\^{a}nescu's work  a linear differential equation satisfied by $\R_{n,n,P}(z)$ has been (re)discovered by J.~M.~Horner, (see \cite{Ho}).

\smallskip
If $P=z^2-1$ and $m=n$, we get the above classical case of the Legendre polynomials up to a scalar factor. In Fig. ~\ref{fig:polShadowExamples} we display the zeros of $\R_{m,n,P}(z)$ for some choices of $P$.
% Analogously, for a meromorphic function $f$ given in some open domain $\Omega\subseteq \bC$, one can define its \emph{ $(m,n)$-th Rodrigues' descendant} in $\Omega$ as 
 %$$\R_{m,n,f}(z) :=\ \frac{\diff ^m f^n(z)}{\diff z^m}.$$
 
In the present paper, we study the asymptotic root distribution for natural sequences of Rodrigues' descendants of $P$. (There is a straightforward generalization of our set-up to the case of rational/meromorphic $P$ which we plan to adress in a future publication.) %, see \cite{BHS}.
\medskip

\subsection{Main results} In what follows, we will always assume that a polynomial $P(z)$ under consideration satisfies the condition $d:=\deg P \ge 2$. The remaining case $d\le 1$ is trivial. 

\smallskip
For any polynomial $P$ 
and its Rodrigues' descendant $\R_{m,n,P}(z),$ denote by $\mu_{m,n,P}$ the {\it root-counting measure} of $\R_{m,n,P}(z)$ and by $$\C_{m,n,P}(z):= \frac{\R^\prime_{m,n,P}(z)}{(dn-m)\cdot \R_{m,n,P}(z) }$$ the {\it Cauchy transform of $\mu_{m,n,P}$}, see \eqref{eq:CauchyTr}. Note that $dn-m=\deg \R_{m,n,P}$. (For the basic notions of the logarithmic potential theory such as the Cauchy transform $\C_\mu$ and the logarithmic potential $L_\mu$ of a measure $\mu$ supported in $\bC$ consult \S~2.2 and \cite{Ra}.)

\smallskip
We say that a polynomial $P$ is \emph{strongly generic} if both $P$ and $P^\prime$ have simple roots.

\begin{theorem}\label{th:Cauchy}
For any strongly generic polynomial $P$ 
 and a given positive number $\al<\deg P$, there exists a weak limit 
$$\mu_{\al,P}:=\lim_{n\to \infty} \mu_{[\al n] ,n,P}.$$ Moreover, its Cauchy transform $\C_{\al,P}$ defined as the pointwise limit 
$$\C:=\C_{\al,P}(z):=\lim_{n\to \infty}\C_{[\al n],n,P}(z)  $$
 exists almost everywhere (a.e.)  in $\bC$ and satisfies %almost everywhere in $\bC$ 
 the algebraic equation
\begin{equation}
\label{eq:symbolcurve1}
(d-\al) \C=\frac{\diff }{\diff z}\log{P}\left(z+\frac{\al}{(d-\al)\C}\right).\end{equation}

\end{theorem}

%\bigskip
%\noindent
% {\rm (i)} For $\al=1$, $R=P/Q$ with $\deg P=3$, and $\deg Q=2$, we get the following algebraic equation for the Cauchy trasform $\C$
%$$2^{10}PQ\C^5+2^9 PQ^\prime \C^4+\left({3\cdot 2^5}PQ^{\prime\prime} + 2^6P^\prime Q^{\prime}-2^5P^{\prime\prime}Q\right)\C^3+ \left(\frac{2^5}{3}P^\prime Q^{\prime\prime}+2^4P^{\prime\prime}Q^\prime\right)\C^2$$ $$+\left(P^{\prime\prime}Q^{\prime\prime}-\frac{2}{3}P^{\prime\prime\prime}Q^\prime\right)\C+\frac{1}{6}P^{\prime\prime\prime}Q^{\prime\prime}=0.$$

\smallskip 

\begin{remark}
{\rm Observe that, by the Gauss-Lucas theorem, for any $0<\al <d,$ the support $S_{\al, P}$ of $\mu_{\al,P}$ is contained in the convex hull  of the zero locus  of  $P$.}
\end{remark}

\begin{remark}
{\rm The condition of strong genericity is apparently redundant and  is an  artefact of our particular proofs.}% We will suppress }
\end{remark}

Reinterpretation of formula \eqref{eq:symbolcurve1} in Theorem~\ref{th:Cauchy} implies the following result. 

\begin{corollary} {\rm  The Cauchy transform $\C:=\C_{\al,P}(z)$ of the limiting measure $\mu_{\al,P}$ satisfies the equation
\begin{equation}\label{eq:algebraicDiffEq}
\sum_{k=0}^{d} \frac{\alpha^{k-1}\,(\alpha-k)(d-\alpha)^{d-k}}{k!}\,P^{(k)}\C^{d-k} = 0.
\end{equation}}
\end{corollary}
\begin{example}

\medskip
\noindent
{\rm (i)} For $P = z^2 +az+b$, equation (\ref{eq:algebraicDiffEq}) reduces to
\begin{equation}\label{eq:LegendreAlgebraicEq}
(2-\al)(z^2+az+b)\C^2 + (\alpha - 1)(2z+a)\C - \al = 0. 
\end{equation}

\smallskip
\noindent
{\rm (ii)} For $P = z^3 +az^2+bz+c$, it reduces to
\begin{equation}\label{eq:LegendreAlgebraicEq2}
(3-\al)^2(z^3+az^2+bz+c)\C^3 + (\alpha - 1)(3-\al)(3z^2+2az+b)\C^2 + \al(\al-2)(3z+a)\C -\al^2= 0.
\end{equation}

\smallskip
\begin{remark} {\rm Observe that equations \eqref{eq:symbolcurve1}  and \eqref{eq:algebraicDiffEq} will substantially simplify if instead of the above Cauchy transform $\C$ one uses its scaled version $\W$ introduced in   \eqref{eq:scaled}, see \S~3.}  
\end{remark} 

%\smallskip
%\noindent 
%{\rm (ii)} For $d=2$, we get
%$$P\C^2-1=0.$$

%{\rm (iii)} For  $d=3$, we get
%$$P\C^3-\frac{P^{\prime\prime}}{2!\, 2^2}\C-\frac{2}{2^3}=0.$$
%{\rm (iv)} For  $d=4$, we get

%$$P\C^4-\frac{P^{\prime\prime}}{2!\,3^2}\C^2-\frac{2P^{\prime\prime\prime}}{3!\, 3^3}\C-\frac{3 P^{(\rm{iv})}}{4!\, 3^4}=0.$$
%{\rm (v)} For  $d=5$, we get 
%$$P\C^5-\frac{P^{\prime\prime}}{2!\, 4^2}\C^3-\frac{2P^{\prime\prime\prime}}{3!\, 4^3}\C^2-\frac{3P^{(\rm{iv)}}}{4!\, 4^4}\C-\frac{4 P^{(\rm{v})}}{5!\, 4^5}=0.$$

\end{example} 
% \begin{figure}[htp]
% \begin{center}
% \includegraphics[width=.4\textwidth]{Figures/Tri-n60-3deriv.pdf}\hfill
% \includegraphics[width=.4\textwidth]{Figures/Tri-n60-18deriv.pdf}
% \includegraphics[width=.4\textwidth]{Figures/Tri-n60-60deriv.pdf}\hfill
% \includegraphics[width=.4\textwidth]{Figures/Tri-n60-120deriv.pdf}
% \end{center}
% \caption{The zeros of $\R_{m,n,P}(z)$ (shown by small dots), where $P = (z-1)(z-6)(z-3i),\,n=60.$ Here, $m=3$ (top left), $m=18$ (top right), $m=60$ (bottom left), and $m=120$ (bottom right). The larger dots are the zeros of $P$, and the triangle is the center of mass of the zero locus  of $P$.}% The squares are zeros of $\frac{R(h,h')}{\alpha P}$, where $\alpha = m/n$,}
% \label{fig:frozenDerivative}
% \end{figure}

\begin{figure}[htp]
\begin{center}
\includegraphics[width=.21\textwidth]{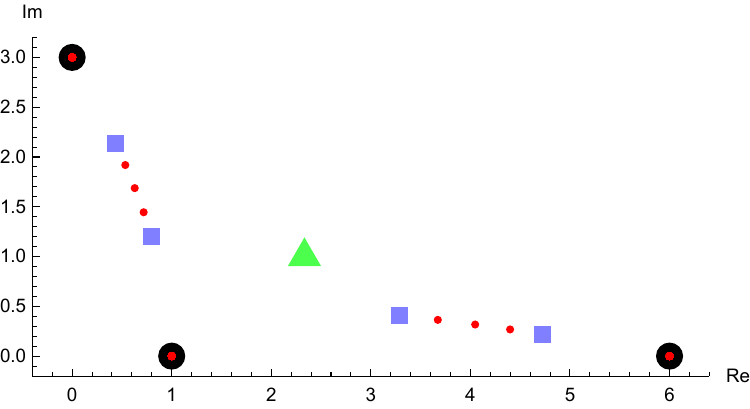}\hfill
\includegraphics[width=.21\textwidth]{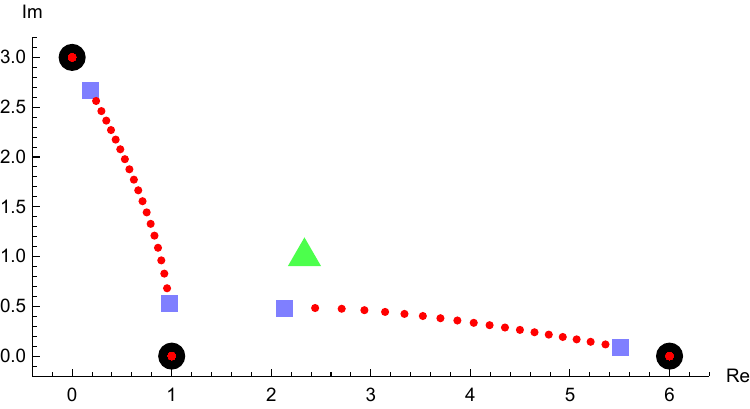}
\includegraphics[width=.11\textwidth]{filler.jpg}
\includegraphics[width=.21\textwidth]{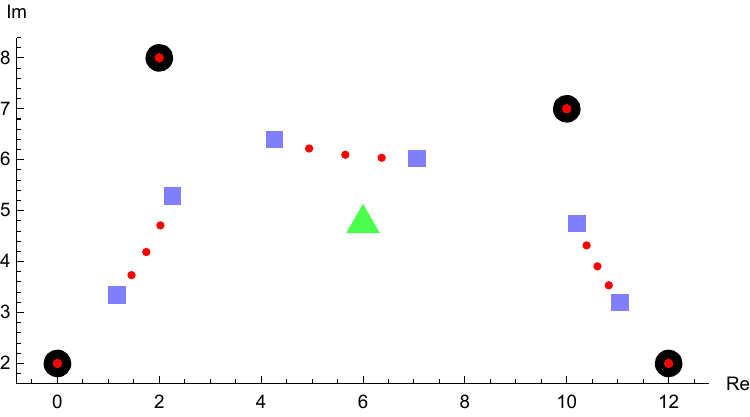}
\includegraphics[width=.21\textwidth]{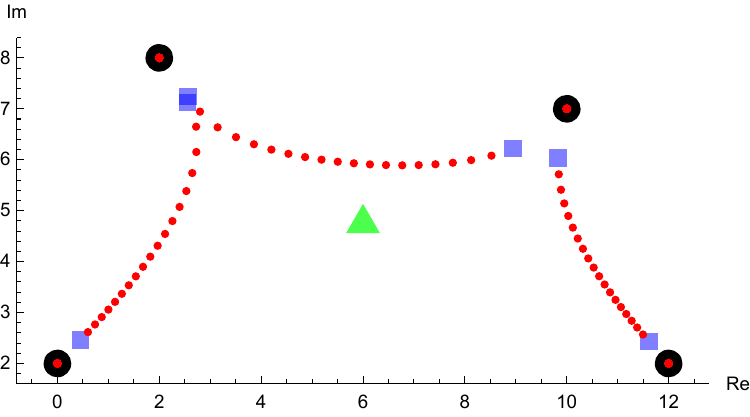}
\includegraphics[width=.21\textwidth]{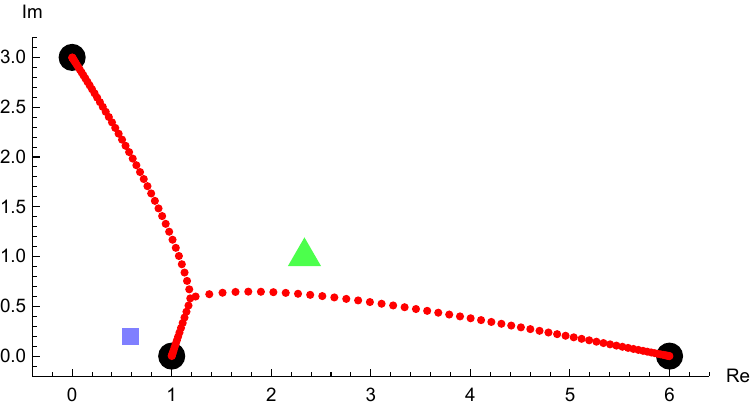}\hfill
\includegraphics[width=.21\textwidth]{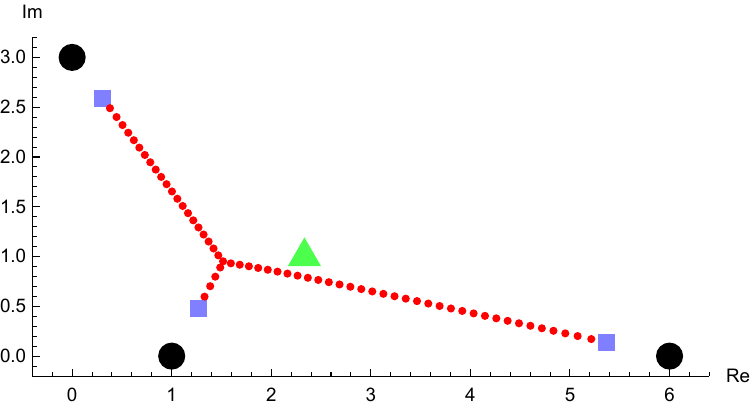}
\includegraphics[width=.12\textwidth]{filler.jpg}
\includegraphics[width=.21\textwidth]{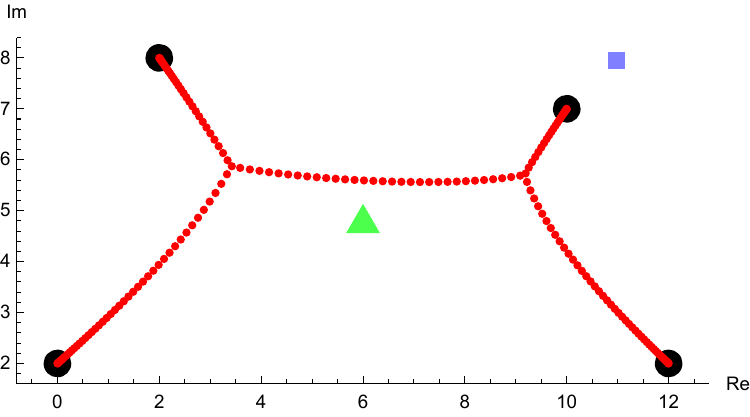}
\includegraphics[width=.21\textwidth]{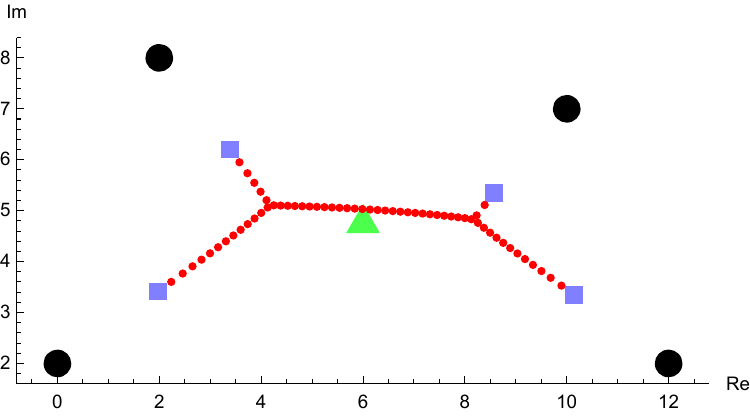}
\end{center}
\caption{The zeros of $\R_{m,60,P}(z)$ shown by the small red dots. (The larger dots are the zeros of $P$, the triangle is the center of mass of the zero locus of $P$, and the squares are branch points of \eqref{eq:algebraicDiffEq} and \eqref{eq:algebraicDiffEq2} in the $z$-plane when $\alpha = m/60$.) Both in the left and in the right subfigures, $m=3$ (top left), $m=18$ (top right), $m=60$ (bottom left), and $m=60(\deg{P}-1)$ (bottom right).}
\label{fig:frozenDerivative}
\end{figure}

The next theorem is our main technical result on the asymptotic limit of the above root-counting measures. (Although several notions in its formulation are explicated only later in the text we want to give a reader the flavour of our results.)

\smallskip
Set    
\begin{equation}
\label{eq:H}
H(z,u):=%\beta^{-1}G(z,u):= 
\frac{1}{d- \alpha }(\log\vert P(u)\vert-\alpha \log \vert u-z\vert).
\end{equation}

\smallskip
Let $\pi: \bC_z\times \bC_u \to \bC_z$ be the standard projection, and $\mathcal D\subset \bC_z\times \bC_u$ be the \emph{ saddle point curve} of $H$, see details in \S~\ref{sec:contour}. ($\mathcal D$ is a rational curve defined by an explicit algebraic equation \eqref{eq:saddle-pointsu} whose coefficients depend on $P$ and $\alpha$.) Further, let $U_{rel}\subset \mathcal D$ be the open set of relevant saddle points, see  \S~\ref{sec:contour}. Denote by $\tilde \pi: U_{rel}\to O\subset \bC_z$ the restriction of $\pi$ to $U_{rel}$ and  define the {\it tropical trace} $\tilde \pi_*H(z):O \to \bR\cup \pm\infty $ as a piecewise-harmonic function obtained by taking the fiberwise maximum of $H(z,u)$, see \S~\ref{sec:trace}. (We will show that $\tilde \pi_*H(z) $ is an $L^1_{loc}$-function defined on the dense and open subset $ O\subset \bC_z$, see Prop.~\ref{prop:cont}. Since the complement of $\bC_z\setminus O$ is a zero set, $\tilde \pi_*H(z) $ extends to a subharmonic $L^1_{loc}$-function on the whole $\bC_z$.)

\begin{theorem}
\label{prop:primitive} In the above notation, for any strongly generic polynomial $P$ of degree $d\ge 2$,  there exists a real number $B$ (explicitly calculated in Lemma~\ref{lem:constantas}) such that 
$$\lim_{n\to\infty}L_{\mu_{[\al n], n,P}}(z )= B+\tilde\pi_*H(z), $$
where the above relation is understood as the equality of $L^1_{loc}$-functions. Here $L_\mu$ stands for the logarithmic potential  of a measure $\mu$, see \eqref{eq:logpot}. %Furthermore, $\tilde\pi_*H(z)$ is a subharmonic $L^1_{loc}$-function. 
Consequently,  
$$\lim_{n\to\infty} \C_{\mu_{[\al n], n,P}}(z) = 2\frac{\partial}{\partial z}\tilde \pi_*H(z), 
$$
and  
$$\lim_{n\to\infty} \mu_{[\al n], n,P} = \mu_{\al ,P}:=\frac{2}{\pi}\frac{\partial^2}{\partial z\partial \bar z}\tilde\pi_*H(z), 
$$
where the latter two limits are understood in the sense of distributions and  $\mu_{\al ,P}$ is a positive measure.
\end{theorem}

Theorem~\ref{prop:primitive} combined with the definition of the tropical trace provide important general information about the support of the measure $\mu_{\alpha,P}$, e.g., that it  consists of certain level curves of explicit harmonic functions associated with the above function $H(z,u)$ together with the algebraic function defined by the saddle point curve, see Theorems \ref{lm:tropicaltrace} -- \ref{th:main}.
Fig.~\ref{fig:frozenDerivative} and ~\ref{fig:polShadowExamples} show several illustrations of (approximations to) the supports of the measures $\mu_{\al,P}$ and their union taken  over  $0< \al <\deg P$. 

\begin{remark} {\rm We want to mention that the previous theorems remain true for any sequence of Rodrigues descendants $\{\R_{a_n,n,P}(z)\}$ such that $\lim_{n\to \infty} a_n/n=\alpha$. (The only changes in our proofs needed to cover this case are setting $m:=a_n$ and $s_n:=n-a_n/\alpha$ in the proof in \S~\ref{sec:MainThs}, and observing that $s_n/n\to 0$ as $n\to\infty$, so that Corollary \ref{cor:uniform} applies, and finally checking that Lemma \ref{lem:constantas} still is valid.)  We want to thank an anonymous referee for suggesting this generalization.}

\end{remark}

%\begin{remark}{\rm 
%Note that the Maclaurin series for $e^z$ and Fa\`a di Bruno's formula can be used to derive the alternate expressions
%\begin{equation}\label{eq:Pnn-alternateForms}
%\begin{split}
%(P^n)^{(n)} & = \sum_{k=0}^{\infty}\frac{n^k}{k!}(\log^k P)^{(n)} \\
%& = \sum_{k=0}^{n}\frac{n!}{(n-k)!}P^{n-k}B_{n,k}(P',P'',\dotsc,P^{(n-k+1)}),
%\end{split}
%\end{equation}
%respectively, where $B_{n,k}(x_1,\dotsc,x_{n-k+1})$ are Bell polynomials.}
%\end{remark}

% \begin{figure}[htp]
% \begin{center}
% \includegraphics[width=.48\textwidth]{Figures/Kvad-SingleZero-n20.pdf}\hfill
% \includegraphics[width=.48\textwidth]{Figures/Kvad-DoubleZero-n20.pdf}
% \includegraphics[width=.48\textwidth]{Figures/Kvad-TripleZero-n20.pdf}\hfill
% \includegraphics[width=.48\textwidth]{Figures/Kvad-QuadZero-n20.pdf}
% \end{center}
% \caption{The zeros of $(P^n)^{(m)}$ for  $m\in\{0,1,\dotsc,(3+j)n\}$, (shown by small dots), where $P = z^j(z-2-8i)(z-8-7i)(z-12)$, $n=20$, and $j=1,2,3,4$. In each case, the larger dots are the zeros of $P$ and $P'$, and the triangle is the center of mass of the zero locus $Z(P)$.}% Furthermore, for $j=1,2$, the squares are zeros of $\frac{R(h,h')}{\alpha P}$ (which vanish for $j\ge 3$), for $\alpha = 1/100,2/100,\dotsc,\frac{299}{100}+j$.}
% \label{fig:quadAmoebas}
% \end{figure}

\begin{figure}[htp]
\begin{center}
\includegraphics[width=.48\textwidth]{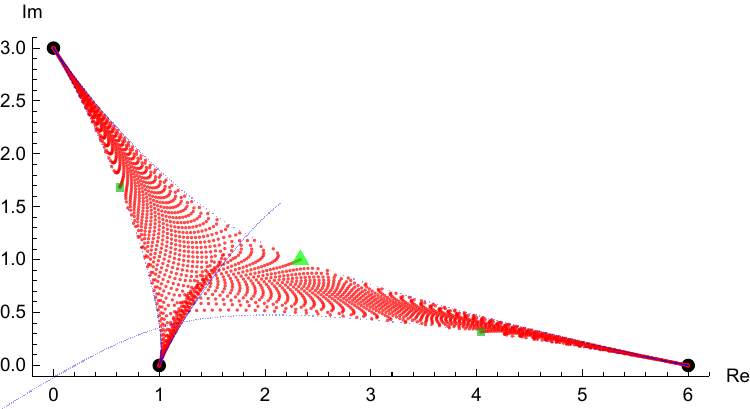}\hfill
\includegraphics[width=.48\textwidth]{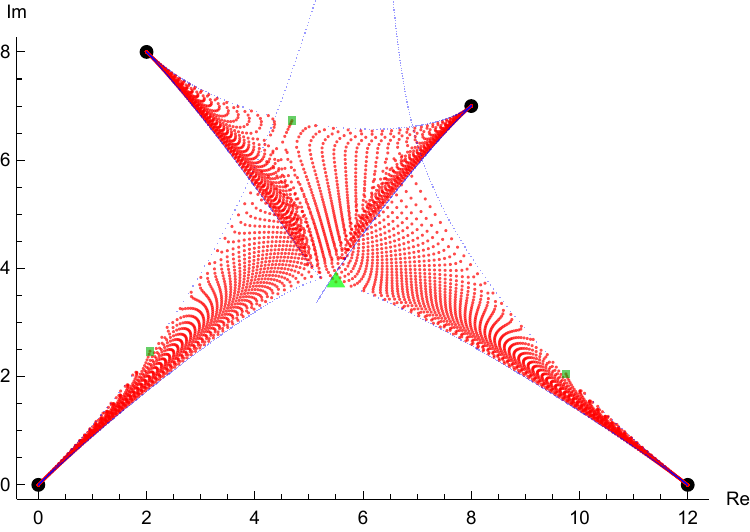}
\includegraphics[width=.48\textwidth]{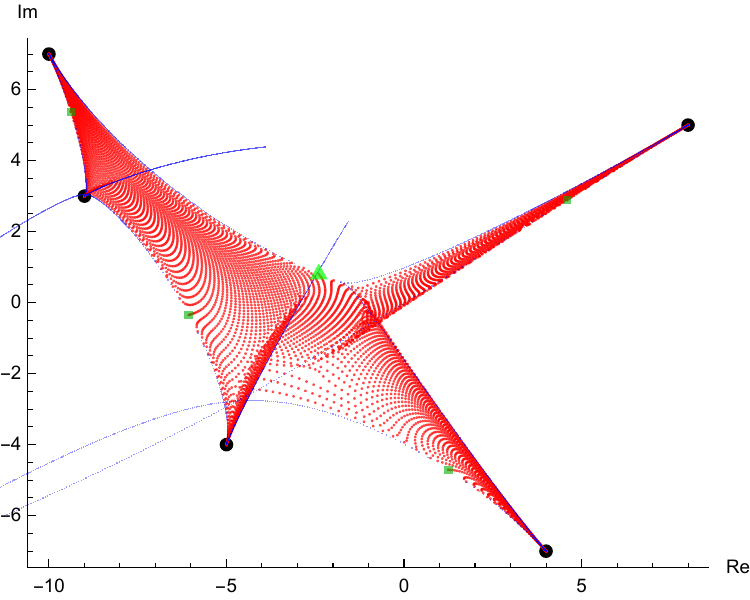}\hfill
\includegraphics[width=.48\textwidth]{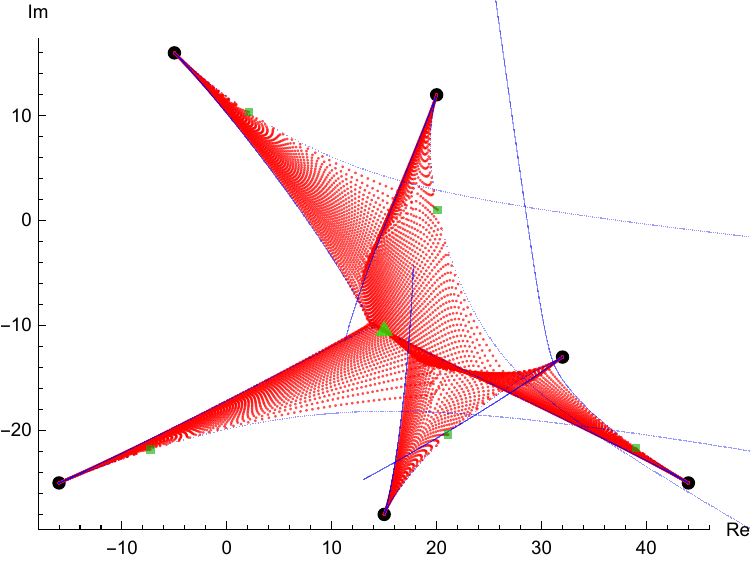}
\end{center}
\caption{The union of all zeros of $\R_{m,30,P}(z)$ for $m= 0,1,$ $\dotsc,30\deg{P}-1$ shown by small red dots. (Other symbols have the same meaning as in Fig.~\ref{fig:frozenDerivative}.)}%The large dots are the zeros of $P$, the large squares are the critical points of $P$, the triangle is the center of mass of the zero locus of $P$, and the small squares are the branch points of \eqref{eq:algebraicDiffEq} and \eqref{eq:algebraicDiffEq2} in the $z$-plane) \textcolor{red}{(*** OnÃ¶dig upprepning? ***)}.}
\label{fig:polShadowExamples}
\end{figure}

\subsection{Methods}
%\textcolor{red}{
Let us first sketch the proof of the main technical result Theorem \ref{prop:primitive}, from which many other results are formal consequences.
Cauchy's formula for higher order derivatives gives 
\begin{equation}
%\label{eq:cauchy}
\q_n(z)=\frac{([\alpha n] - 1)!}{2\pi i}\int_{c} \frac{P^n(u)\,\diff u}{(u-z)^{[\alpha n]}}=\frac{([\alpha n]-1)!}{2\pi i}\int_{c}\mathrm{exp}\bigg(n\log P(u)-([\alpha n]\log(u-z)\bigg)\diff z.
\end{equation}
Here $c$ is any simple closed curve in $\bC$ encircling $z$   once in the counterclockwise direction. 
The saddle point method  heuristically implies  that 
$$
\frac{1}{n} \log \q_n(z)\approx \log P(u)-\alpha \log(u-z),
$$
where $u(z) $ is some solution of the saddle point equation
$$
\frac{P'(u)}{P(u)}-\frac{\alpha}{u-z}=0, 
$$
determining the critical points of the integrand.
 The degree of the polynomial $\q_n(z)$ equals $d_n:=dn-[\alpha n]+1$. Therefore, up to scaling, $\frac{1}{n} Re\, (\log \q_n(z))$  equals the logarithmic potential $\{L_{\mu_n}(z)\}$ of the root-counting measure $\mu_n$, which we hence  understand asymptotically.
 %}

The main difficulty in making the above sketch rigorous is to describe which particular branch $u(z)$ of solution to the latter saddle point equation  to choose. We address this issue  by applying  to our specific situation a general framework developed in \S~\ref{sec:sub}.  We hope that this framework can be useful in other asymptotic questions involving sequences of polynomials originating from  families of linear ordinary differential equations.  

\medskip
Namely,  given $\bC^2\simeq \bC_z\times \bC_u$ with  coordinates $(z,u)$, we define  below a special class of plane algebraic curves which we call \emph{affine Boutroux curves} (shorthand aBc). Such a  curve $\Upsilon\subset \bC^2$ is  characterized by the fact that the standard $1$-form $u\minispace\diff z$ has only imaginary periods on the normalization of the compactification of $\Upsilon$ in $\bC P^1_z\times \bC P^1_u$.   

A version of Boutroux curves has been earlier introduced in \cite{BM} where also the term  ``Boutroux curves" was coined. This notion was  further elaborated in \cite {Be}  and later used by a number of authors. 

Given an  affine Boutroux curve, we define on it a natural harmonic function which is  essentially the real part of a primitive function of $u\minispace\diff z$, as well as the push-forward of this function to $\bC P^1_z$. This push-forward ---which we call the \emph{tropical trace}, and which is our crucial tool---is  piecewise-harmonic and its Laplacian (considered as a $2$-current on $\bC P^1_z$) is a signed measure supported on a finite union of segments of analytic curves and isolated points. The most essential property of this measure is that its Cauchy transform satisfies almost everywhere (a.e.)  in $\bC P^1_z$ the same algebraic equation which defines the initial affine Boutroux curve. We will also apply this construction to certain open (in the usual topology) subsets  of a Boutroux curve.

\medskip
The structure of the paper is as follows. After recalling some basic notions in \S~\ref{sec:basics}, we introduce  in \S\S~\ref{sec:sub} -- \ref{sec:BoutrouxCauchy} affine Boutroux curves  (aBc) as well as related harmonic functions, tropical traces   and their measures.  In particular, we give a simple general construction of Boutroux curves, of which the one used in this paper is a special case.   In \S~\ref{sec:perfect}, we prove  that the algebraic curves given by \eqref{eq:algebraicDiffEq}  and \eqref{eq:algebraicDiffEq2} are affine Boutroux curves.  In \S~\ref{sec:MainThs}, we settle   Theorems~\ref{th:Cauchy} and \ref{prop:primitive} and related results by applying the saddle point method to Cauchy's integral, in a very classical way. In \S~\ref{sec:proofs}, we derive linear differential equations satisfied by Rodrigues' descendants.  In \S~\ref{sec:quadratic}, we discuss in detail the case of a quadratic polynomial $P$. Finally, in \S~\ref{sec:final}, we suggest a generalization of our set-up to non-discrete measures and pose a number of open problems related to the asymptotic of Rodrigues descendants.

\begin{remark} {\rm This text has been mainly written, but for various reasons not completely finished already in Spring 2018; its content has been presented during a workshop ``Hausdorff geometry of polynomials and polynomial sequences"  at the Mittag-Leffler institute in Stockholm.  Since then several relevant papers discussing  similar questions about the behavior of roots of polynomials under consecutive differentiations appeared, see e.g., \cite {St1, St2, HoKa, KiTa}. In particular, paper \cite{St2} contains a heuristic deduction of an intriguing partial differential equation satisfied (under several additional assumptions) by the density of roots under differentiation. This equation has been further studied  in \cite{KiTa}. In addition, a recent contribution \cite{HoKa} contains a number of results in the case of polynomials of degree $2$ which are quite close to those in our \S~\ref{sec:quadratic}. }  
\end{remark} 

\medskip
\noindent{Acknowledgements.} The authors are sincerely grateful  to the Mittag-Leffler institute for the hospitality  in June 2018.  The third author wants to thank Professors Arno Kuijlaars, Maurice Duits, Zakhar Kabluchko, and Andrei Mart\'inez-Finkelshtein for their interest in this subject and discussions. The first author wants to thank Helga Lundholm for her support and interest in this project. Finally, we are very grateful to the anonymous referees for their high quality comments which enabled us  to improve  the content as well as the exposition of the paper. 

%  where some of these results were presented. %t tThe  third author wants to acknowledge the financial support of his research provided by the Swedish Research Council grant  2016-04416.  %The second author would like to thank Per Alexandersson for providing interesting numerical results about the zeros of $((P/Q)^n)^{(n)}$.
%The first author would on the other hand not even give the Swedish Research Council directions to the nearest cold beer if he passed them on a thirsty day in Sahara...

  \section{Various preliminaries} 
  
   \subsection{Basics of logarithmic potential theory}\label{sec:basics}
     For the convenience of our readers, let us briefly recall some notions and facts used throughout the text. Let $\mu$ be a finite compactly supported positive Borel  measure in the complex plane $\bC$.  Define the {\it logarithmic potential} of $\mu$ as
 \begin{equation}\label{eq:logpot}    
 L_\mu(z):=\int_\bC \ln|z-\xi|\,\diff \mu(\xi)
 \end{equation}
  and the {\it Cauchy transform} of $\mu$ as 
\begin{equation}\label{eq:CauchyTr}
\C_\mu(z):=\int_\bC \frac{\diff \mu(\xi)}{z-\xi}.
\end{equation}
%=\frac{\mathfrak \prt \mathfrak u_\mu(z)}{\prt z}$$ 

Standard facts about the logarithmic potential and the Cauchy transform include the following. 
\begin{itemize}

%\noindent
\item  $\C_\mu$ and $L_\mu$ are locally integrable; in particular they define  distributions on $\bC$ and therefore can be acted upon by $\frac{\partial}{\partial z}$ and $\frac{\partial}{\partial \bar z}.$

%\noindent
\item $\C_\mu$ is analytic in the complement of the support of $\mu$ considered in $\bC P^1\simeq \bC\cup \{\infty\}$. For example,  if $\mu$ is supported on the unit circle, then $\C_\mu$ is analytic both inside the open unit disc and outside the closed unit disc.

%\noindent
\item The main relations between $\mu$, $\C_\mu$ and $L_\mu$ are as follows: 
$$\C_\mu=2 \frac{\prt L_\mu}{\prt z} \quad \text{and} \quad  
\mu=\frac{1}{\pi}\frac{\prt \C_\mu}{\prt \bar z}=\frac{2}{\pi}\frac{\prt^2 L_\mu}{\prt z \prt \bar z}=\frac{1}{2\pi}\left(\frac{\partial^2L_\mu}{\partial x^2}+\frac{\partial^2L_\mu}{\partial y^2}\right).$$
(They should be understood as equalities of distributions.)
%(In the latter equality $\mu$ should be understood as the density of the measure; the actual measure equals $\mu dx dy$.) 

%\noindent
\item The Laurent series of $\C_\mu$ in a neighborhood of $\infty$ is given by
$$\C_\mu(z)=\frac{m_0(\mu)}{z}+\frac{m_1(\mu)}{z^2}+\frac{m_2(\mu)}{z^3}+\ldots ,$$
where 
$$m_k(\mu)=\int_\bC z^k\,\diff \mu(z),\; k=0,1,\dots$$
are the harmonic moments of the measure $\mu$. 
\end{itemize}

\medskip Given a polynomial $p$, we associate to $p$ its standard \emph{root-counting measure} $$\mu_p=\frac{1}{\deg p} \sum_i m_i\delta (z_i),$$ where the sum is taken over all distinct roots $z_i$ of $p$ and $m_i$ is the multiplicity of $z_i$. Here $\delta(a)$ stands for the standard Dirac measure supported at $a$. 
%For a rational function $p/q$ with coprime $p$ and $q$, we associate its \emph{ divisor-counting (signed) measure}$$\mu_{p/q}=\frac{1}{\deg p + \deg q}   \left(\sum_i m_i\delta (z-z_i)-\sum_j n_j\delta (z-\zeta_j)\right),$$
%where $\{z_i\}$ is the set of all distinct roots of $p$ with $m_i$ being the multiplicity  of $z_i$ and  $\{\zeta_j\}$ is the set of all distinct roots of $q$ with $n_j$ being the multiplicity of $\zeta_j$. (In many situations, it makes sense to add to $\mu_p$ (resp. $\mu_{p/q}$) the additional mass at $\infty$ such that the total mass of the resulting signed measure $\bar \mu_p$ (resp. $\bar \mu_{p/q}$) vanishes, see below. This operation corresponds to considering $p$ (resp. $p/q$) as a rational function on $\bC P^1$ instead of $\bC$.)   THIS IS NOT WHAT WE USE!!!

\smallskip
One can easily check that the Cauchy transform of $\mu_p$  is given by 
$$\C_{\mu_p}=\frac{1} {\deg p}\cdot\frac{p^\prime}{p}.\quad
% \text{resp.}\quad \C_{\mu_{p/q}}=\frac{1}{(\deg p +\deg q) }\cdot\frac {(p/q)^\prime}{p/q}=\frac{1}%{(\deg p +\deg q) }\left(\frac{p^\prime}{p}-\frac {q^\prime}{q}\right).
$$

For more relevant information on the Cauchy transform we  recommend the short and well-written treatise \cite{Ga}. 

\medskip
The above notions of a Borel measure $\mu$ compactly supported in $\bC$, its logarithmic potential $L_\mu$, and its Cauchy transform $\C_\mu$ have natural extensions to $\bC P^1\supset \bC$; we denote these extensions as   $\bar \mu, \bar L_{\bar \mu}, \bar \C_{\bar \mu}$ respectively. (The main relations between them will be preserved under such extension.) These notions are constructed as follows. 

(i) For a  finite positive measure $\mu$ compactly supported in $\bC$, we introduce the signed measure $\bar \mu$ of total mass $0$ defined on $\bC P^1$ by adding to $\mu$  the point measure $-\mathfrak{m} \cdot \delta(\infty)$ placed at $\infty$, where $\mathfrak m=\int_\bC \diff\mu$. (It is natural to think of $\bar \mu$ as an exact $2$-current on $\bC P^1$.)

(ii) The logarithmic potential $L_\mu$ is originally defined as a function on $\bC\subset \bC P^1$ with a logarithmic singularity at $\infty$. In terms of a local coordinate $w=1/z$ at $\infty$ the logarithmic potential is a $L^1_{loc}$-function near $\infty$ which implies that we can define its derivatives (at least in the sense of distributions).  We denote  by  $\bar L_{\bar \mu}$ the function $L_\mu$ considered as a $L^1_{loc}$-function on the whole $\bC P^1$. 

%DO WE USE THESE NOTIONS ON $\bC P^1$?

%of $\mu$, we define its logarithmic potential extension $\bar L_{\bar \mu}$ %as  the distribution on $\bC P^1$
 %obtained by considering $L_\mu$ as the distribution to $\bC P^1$. (Indeed, $ L_{\bar \mu}$ at $\infty$ has asymptotics 
 %$\mathfrak{m}\log \vert z \vert$ which implies that it belongs to $L^1_{loc}$ near $\infty$ and therefore has a natural extension to $\bC P^1$ as a distribution. (It is natural to think of $L_\mu$ as a $0$-current 
 %on $\bC P^1$.) 
 
 Recall that on any complex manifold, the exterior differential $\diff$ (acting on differential forms and currents) is standardly decomposed as $\diff=\diff^\prime+\diff^{\prime\prime}$, where $\diff^\prime$ is its holomorphic and $\diff^{\prime\prime}$ is its anti-holomorphic parts.  %(*** ``where $\diff^\prime$ and $\diff^{\prime\prime}$ are its holomorphic and anti-holomorphic parts, respectively.''?}
 For a function $f$ on a Riemann surface with  a local holomorphic coordinate $z$, we get  $$\diff'f=\frac{\prt f}{\prt z}\diff z\quad \text{and}\quad 
 \diff''f=\frac{\prt f}{\prt \bar z}\diff\bar z.$$
In the above notation, the quantities $\bar \mu$ and $ \bar L_{\bar \mu}$ satisfy  the relation 
$$
\bar \mu \,\diff x \wedge \diff y = \frac{i}{\pi} \diff'\diff'' \bar L_{\bar \mu}.
$$
More explicitly, we have that  
  
$$\bar \mu\minispace\diff x \wedge\diff y=\frac{1}{2\pi} \left( \frac{\prt^2 \bar L_{\bar \mu}}{\prt x^2}+ \frac{\prt^2 \bar L_{\bar \mu}}{\prt y^2} \right)\,\diff x\wedge\diff y=\frac{2}{\pi}\frac{\prt^2 \bar L_{\bar \mu}}{\prt z \prt \bar z}\,\diff x\wedge\diff y=\frac{i}{\pi}\frac{\prt^2 \bar L_{\bar \mu}}{\prt z \prt \bar z}\,\diff z \wedge\diff \bar z,$$
where $\frac{\prt^2 \bar L_{\bar \mu}}{\prt z \prt \bar z}$ is understood as a distribution on $\bC P^1$.

(iii) Finally, the Cauchy transform  $\bar \C_{\bar \mu}$ is naturally interpreted  as a  %exact
$1$-current defined by the relation 
$$\bar \C_{\bar \mu}=2\diff' \bar L_{\bar \mu}=  2\frac{\prt \bar L_{\bar \mu}}{\prt z}\,\diff z.$$
With this convention  we get 
$$\bar \mu\minispace\diff x\wedge\diff y=\frac{i}{\pi}\diff'\diff'' \bar L_{\bar \mu} = -\frac{i}{2\pi}\diff''\bar \C_{\bar \mu}=\frac{i}{2\pi}\frac{\prt \bar C_{\bar \mu}}{ \prt \bar z}\,\diff z\wedge\diff\bar z.$$

\medskip
\subsection{Differentials with imaginary periods}%, affine Boutroux curves, subharmonic functions, and signed measures on $\bC P^1$} 
\label{sec:imaginaryperiods}

\label{sec:sub}
To settle Theorem~\ref{th:Cauchy} and other related results, we  need to introduce a special class of plane algebraic curves and show how they give rise to measures on (open subsets of) $\bC P^1$. This is an instance of a more  general construction which can be carried out for  Riemann surfaces  endowed with an   abelian differential %(i.e., a meromorphic $1$-form)  
  and a meromorphic function.

\smallskip
Typically multi-valued (harmonic and subharmonic) functions on Riemann surfaces originate from the integration of meromorphic $1$-forms. For some special types of differentials however, one can get functions that are uni-valued instead of multi-valued which is exactly the situation which we want to capture. 

As usual, by a {\it period} of a meromorphic $1$-form on a Riemann surface $Y$ we mean the integral of this form over a $1$-cycle in $H_1(Y\setminus Pol,\bR)$, where $Pol$ is the set of poles of the form under consideration.  

\begin{definition}\label{def:1} {\rm A meromorphic $1$-form $\omega$ defined on a compact orientable Riemann surface $Y$ is said to have \emph { purely imaginary periods} if all of its periods are  purely imaginary complex numbers. %We say that a germ $w=f(z)$ of an algebraic function  near $\infty\in \bC P^1$ \emph {has purely imaginary periods} if the globally defined meromorphic $1$-form $\Omega_f$ on $\widetilde \Gamma_f$   has all imaginary periods.
 }
\end{definition}

\begin{remark}{\rm  
Observe that the periods of $\omega$ can be roughly subdivided into two different types: a) periods related to the poles of $\omega,$ i.e. integrals of $\omega$  over small loops surrounding the poles, and b) periods related to the non-trivial $1$-dimensional homology classes of $Y$, i.e., integrals of $\omega$ over the global cycles in $H_1(Y, \bR)$.  (Observe however that these two types of periods are, in general, dependent.)

Note that the first type of periods are purely imaginary if and only if all residues of $\omega$  are real and that the second type of periods do not occur if $Y$ has genus $0$. }
\end{remark}

\begin{remark}
\label{rmk:notneccompact}{\rm  
In some situations  Definition~\ref{def:1}  makes sense even if $Y$ is  a non-compact Riemann surface. For our purposes, it will be suffficient  to consider the case when $Y$ is an open subset of a compact Riemann surface $\widetilde Y$ such that $\widetilde Y\setminus Y$ consists of a finite number of points.
 %\textcolor{red}{ST\"AMMER DETTA?} 
This will always be the case for, e.g., smooth quasi-affine plane algebraic curves. Note that a meromorphic $1$-form $\omega$ on $\widetilde Y$ has purely imaginary periods if and only if the restriction of $\omega$ to $Y$ has purely imaginary periods.}
\end{remark}

For a meromorphic $1$-form $\omega$ with purely imaginary periods defined on a compact Riemann surface $Y$, denote by $Pol_\omega^- \subset Y$ (resp. $Pol_\omega^+\subset Y$) the set of all poles of $\omega$ with negative (resp. positive) residues. Set $Pol_\omega :=Pol_\omega^+\cup Pol_\omega^-$. 

\medskip
Meromorphic $1$-forms, i.e.,  abelian differentials with purely real periods were introduced by I.~Krichever in the 1980's in connection with the theory of integrable systems and have been discussed since then in a number of his papers. In particular, they were considered in  \cite{GK} where they were used to study  the moduli spaces of  Riemann surfaces with marked points. (In the present article we consider purely imaginary periods, but the translation is trivial.)  One of the results of \cite{GK} is as follows; see Proposition 3.4 in loc. cit.

\begin{PROP}\label{prop:GrKr}
 For any compact Riemann surface $Y$, a set of marked points $ p_1,\dots ,p_n \in Y$, any set of positive
integers $h_1,\dots, h_n$,  any choice of $h_i$-jets of local coordinates $z_i$ in the
neighborhood of marked points $p_i$, of the singular parts (i.e., for $i = 1,\dots,n,$ the choice of Taylor coefficients $c^1_i,\dots ,c^{h_i}_i$, with all imaginary residues 
$c_i^1\in i\bR$ and the sum of the residues $\sum c_i^1 $ vanishing), there exists a unique differential $\Psi$ on $Y$ with purely real periods and prescribed singular parts. In other words,  in a neighborhood $U_i$ of each $p_i$ the differential $\Psi$ satisfies the condition
$$\Psi\vert_{U_i}=\sum_{j=1}^{h_i}c_i^j\frac{\diff z}{z_i^j}+O(1).$$
\end{PROP}

\medskip
Proposition~\ref{prop:GrKr} implies that  on an arbitrary compact Riemann surface $Y$ there exists a large class of meromorphic $1$-forms with purely imaginary periods.

 Furthermore, we can associate to each meromorphic differential with imaginary periods on $Y$ a real-valued function $Y\setminus Pol_\omega \to \mathbb R$ as follows.  Fix a point $p_0\in Y\setminus Pol_\omega$ and
consider the multi-valued primitive function 
  $$\Psi(p):=\int_{p_0}^p \omega.$$
  $\Psi(p)$  is a well-defined uni-valued function on the universal covering of $ Y \setminus Pol_{\omega}$. The next statement is trivial.

\begin{lemma}\label{lm:trivial}
In the above notation,   $\omega$  has purely imaginary periods on $Y$ if and only if the multi-valued primitive function $\Psi(p)$ has a uni-valued real part $\mathrm{Re}\,\Psi(p)$. In other words,  
$$H(p):=\mathrm{Re}\,\Psi(p)$$
is  a well-defined uni-valued function on $Y \setminus Pol_\omega$. 
 \end{lemma}

Note that $H(p) $ is continuous and harmonic in a neighborhood of any point in $Y\setminus Pol_\omega$. The local behavior of $H$ near a pole $p$ is determined by the sign of the residue of $\omega$ at $p$. Namely,  let $\omega$ be a meromorphic $1$-form with purely imaginary periods and only simple poles. For $p\in Pol_\omega$, let $z$ be a local coordinate at $p$, and denote by $r$ the residue of $\omega $ at $p$.

\begin{lemma}\label{lm:trivial2}   In the above notation,  $H$ is a subharmonic $L^1_{loc}$-function on $Y\setminus Pol_\omega^-$ which is harmonic on $Y\setminus Pol_\omega$.  Locally, for the restriction of $H$ to a suitable neighborhood of $p$,  the following holds:
\begin{enumerate}
\item $H(z)=r\log\vert z\vert +\widetilde H(z)$, where $\widetilde H(z)$
 is a function harmonic in a neighborhood of $p$. Consequently, $\frac{\partial^2 H}{\partial z\partial \bar z}=\frac{r\pi}{2} \delta(p)$, where $\delta(p)$ is the Dirac measure at $p$, and the derivatives are taken in the sense of distributions.
\item If $p\in Pol_\omega^+$, then there is a neighborhood of $p$ in which $H$ is a well-defined  subharmonic function and $\lim_{z\to p}H(z)=-\infty$.

\item If $p\in Pol_\omega^-$, then $\lim_{z\to p}H(z)=+\infty$.

\item $\frac{\partial H(z)}{\partial z}=\frac{1}{2}\omega$.
\end{enumerate}
 \end{lemma}
 
 \smallskip
 \begin{proof} Item (1) is a consequence of the fact that $\omega$ has a simple pole at $p$ and hence locally it can be written  as $\omega=\frac{r}{z}\minispace\diff z+\widetilde \omega$, where $\widetilde \omega $ is holomorphic at $p$. Then (2) and (3) follow, while (4) follows from  the standard relation
 $$\omega=\frac{\partial\Psi}{\partial z}=2\frac{\partial(\text{Re }\Psi)}{\partial z}.$$ 
 \end{proof}

\subsection{Tropical trace}
\label{sec:trace}
%We will construct the logarithmic potential of the requred asymptotic measure  using a ``tropical'' analogue of the usual trace operation defined below. 
Given a branched covering $\nu: Y \to Y^\prime$ of Riemann surfaces, and a function $f: Y\to \bR$, we will define the induced function  on $Y^\prime$ by taking the maximum of the values of $f$ over each fiber. 
Notice that in the case of the usual trace one uses the summation/integration over the fiber. The basic idea of  tropical geometry is to substitute the operation of summation/integration by the operation of taking the maximum,  which provides a motivation for our terminology. It seems that this construction which regularly occurs in the study of the root asymptotic for polynomial sequences has not been given any special name yet.
 
 \begin{definition} \label{def:tr}{\rm Given a branched covering $\nu: Y \to Y^\prime$ and a real-valued function $f:Y\to \bR$, we define the  \emph {tropical trace  $\nu_*f: Y^\prime \to \bR$ } of this pair 
as
$$\nu_*f(z)=\max_{y_i \in \nu^{-1}(z)}f(y_i).$$
The same definition extends to real-valued functions $f$ defined on $Y\setminus S$, where $S$ is a discrete set such that for any $s\in S$, $\lim_{z\to s}f(z)$  exists either as a real number or $\pm \infty$. (In other words, we allow $f$ to attain values $\pm \infty$.)} 
\end{definition}

\begin{example}
\label{ex:trace} {\rm Let $H_i,\; i\in \{1,2....,n\}=:[n]$  be an $n$-tuple of real pair-wise different harmonic functions on $Y'$. They define  a harmonic function $H $ on the product $Y:=Y'\times [n]$ by setting $$H(z,i)=H_i(z).$$ For  the canonical projection  $\nu: Y \to Y'$  given by $\nu(z,i)=z$, we get 
$$\nu_*H(z)=\max_{i \in [n]}H_i(z).$$ 

%\textcolor{red}{
Let $C$ be the union of the segments of analytic curves given by $H_i(z)=H_j(z), \ i< j$. Notice that $\nu_*H(z)$ is a subharmonic function which is harmonic and coinciding with the unique $H_i$ on each connected component of the complement $Y\setminus C$.   Hence $\nu_*H$ is a piecewise-harmonic function 
and  its Laplacian  is supported on $C$.  In fact,  this Laplacian (considered as a measure on $C$) can be  explicitly given by the Plemelj-Sokhotski formula, in terms of the analytic functions $\frac{\partial H_i(z)}{\partial z}$ and the curve $C$, see e.g. \cite{BB}. Additionally,  the derivative $E(z):=\frac{\partial \nu_*H(z) }{\partial z}$ exists for $z$ in $Y'\setminus C$ and satisfies a.e. on $Y^\prime$ the equation 
$$\prod_{i=1}^n\left(E(z)-\frac{\partial H_i(z)}{\partial z}\right)=0$$ which is then an instance of an algebraic equation with analytic coefficients satisfied a.e. by  the derivative of a subharmonic function. Such equations often appear in the study of the asymptotic Cauchy transform of the  root-counting measures for polynomial sequences and (under some extra conditions) they imply that this asymptotic  Cauchy transform  is locally given as a  maximum of a finite number of harmonic functions,  i.e. is  their tropical trace as happens in the above example, see e.g. \cite{BBB}.}
%}
\end{example}

\begin{remark} {\rm 
Definition~\ref{def:tr} is applicable to an arbitrary finite map which is a branched cover of complex manifolds. The elementary fact that the maximum of a finite number of subharmonic functions is subharmonic implies certain restrictions on the support of the Laplacian of $\nu_*f$, which makes Definition~\ref{def:tr}  useful. In particular, the tropical trace of a subharmonic function is subharmonic (except possibly at its poles). We describe the situation in more detail in   Theorem~\ref{lm:tropicaltrace} below. }
\end{remark} %, for functions of the type in the preceding section.

\smallskip
Further, let $Y$ and $Y^\prime$ be  Riemann surfaces and  let $\nu: Y \to Y^\prime$ be a branched covering. Take  a real-valued function $f:Y\to \bR$ which is harmonic except at a finite set where it has logarithmic singularities. (In other words, in a neighborhood of a singular point $p\in Pol$, $f(z)=r\log\vert z\vert+\widetilde f(z)$, where $z$ is a local coordinate and $\widetilde f(z)$ is harmonic.) Let as above $Pol^-_f$ (resp.  $Pol_f^+$) be the set of those points $p\in Pol$ at which  the residue $r$ is negative (resp. $r$ is positive). Then $f$ is subharmonic in $Y\setminus Pol_f^-$. Note that  $Pol=Pol_f^-\cup Pol_f^+$  supports all  the point masses of the Laplacian of $f$ considered as a measure.

\begin{theorem}
\label{lm:tropicaltrace}
{\rm 
Under the above assumptions, the tropical trace  $\nu_*f$ is continuous and piecewise-harmonic  in the open set $ Y^\prime\setminus \nu(Pol)$,  subharmonic in $U=Y^\prime\setminus \nu(Pol_f^-)$, and  has at most logarithmic singularities. The Laplacian of $\nu_*f$ in $U$ is supported on a finite union of  segments of real analytic curves 
 and points; the latter set is contained in the set of the images  of all poles of $f$ under the map $\nu$.  }  
 \end{theorem}
 \begin{proof}Note first that the maximum $h(z)=\max f_i(z)$ of a finite number of harmonic functions $f_i, \ i=1,...,\ell$, defined on an open set $V^\prime\subset Y^\prime$ is subharmonic and continuous;  its Laplacian is supported on (some parts of) the level curves $f_i=f_j$, $i< j$. Furthermore these level curves are real analytic.  
 In each connected component $C$ of the complement to the union of  all level curves $f_i=f_j$, $i< j$, there exists an $i$ such that $h(z)=f_i(z)$ for all $z\in C$. Hence $h(z)$ is also piecewise-harmonic.
% Assume  that all $f_i$s are subharmonic and harmonic except for the certain points $p$, where they might have  logarithmic singularities. In other words,  $f_i(z)=c_i\log\vert z\vert+ \tilde f(z)$, where $c_i\geq 0$ and $\tilde f_i$ is harmonic. 
 
 %In our situation  $h(z)$ is  subharmonic.
 One gets $h(p)=-\infty$ only in the case where $f_i(p)=-\infty$ for all $i=1,2,...,\ell$. In addition to a measure supported on a union of segments of real analytic curves,  the Laplacian of $h$ will contain the point mass  $\min_i{r_i} \cdot \delta(p)$  at $p$ where the $r_i$'s are the respective residues. 
 If some $r_i<0$, then $h$ will not be subharmonic at $p$, but it still has a logarithmic singularity at $p$ which implies that its Laplacian contains a negative point mass at $p$.
 
 \smallskip
 Let $Cr\subset Y$ denote the set of all  critical points of the map $\nu$ and denote by  $Cv:=\nu(Cr)\subset Y^\prime$ its set of critical values.  
 If $V^\prime\subset Y^\prime\setminus Cv$ is a simply connected  open set, then the inverse image $\nu^{-1}(V^\prime)$ is  a disjoint union $\cup V_i$ of open sets $V_i\subset Y$, such that the restriction $\nu:V_i\to V^\prime$ is a local biholomorphism  which we denote by $\nu_i$. 
 
 In the above notation, $\nu_*f(z):=\max_i f(\nu_i^{-1}(z)),\ z\in V^\prime$, is a subharmonic function in $V^\prime\setminus \nu((Pol_f^-))$, and as shown above, its Laplacian in $V^\prime$ is supported on a union of segments of real analytic curves  and (possibly) at some points lying in $\nu(Pol_f)$. A similar argument works for  a critical value $p\in Cv$. Namely, in suitable local coordinates $w$ on $Y$ and $z$ on $Y^\prime $ respectively, where the point $p\in Cv$ is given by $z=0$, the map $\nu$ can  be written as $w\mapsto z=\nu(w)=w^k$ for some positive integer $k\ge 2$. The rest of Theorem~\ref{lm:tropicaltrace} follows since it is always possible to cover $Y'\setminus Cv$ by a finite number of open simply-connected sets such that the above argument holds.
  \end{proof}

Given a branched covering $\nu: Y \to Y^\prime$ and a real-valued harmonic and continuous function $f:Y\to \bR$, we can consider the restriction $\nu_U$ of $\nu$ to an open subset $U\subset Y$. Generally, the corresponding tropical trace $\nu_{U_*}$ can have a rather wild behavior on the boundary $\nu(\partial U$). We will now provide conditions on $U$ that ensure that  $\nu_{U_*}f$  is still a continuous function and inherits the nice properties of $\nu_{*}$ formulated in the above theorem.

\smallskip
In the notation from the proof of the above theorem, 
for a simply-connected open domain $V^\prime \subset Y'\setminus Cv$  and $i\neq j$, set 
$$\Delta^{ij}_{V^\prime}:=\{ z\in V^\prime: f(v_i(z))=f(v_j(z))\}.$$
 Then either $\Delta^{ij}_{V^\prime}$ coincides with  $V^\prime$ or it will be an analytic curve.
Now define the \emph{non-simple locus} of the pair $(\nu,f)$ as $$\Delta:=\cup\Delta^{ij}_{V^\prime}\subset Y^\prime,$$ where the union is taken over all $V^\prime$ and all $i\neq j$.

Note that if $\Delta\subset Y^\prime$ is a segment of a real analytic curve,  then 
the ordering  of the  values of $f$ on the different branches $v_i(z), \ i=1,\dots ,d,$ will be the same for all points in any chosen connected component $C$ of $\mathcal{O}:=Y^\prime\setminus \Delta$; that is, there is a permutation $(i_1, i_2, \dots, i_d)$ of $(1, 2, \dots, d)$ such that, for all $z\in C$, 
\begin{equation}
\label{eq:orderedfibre}
f(v_{i_1}(z))> f(v_{i_2}(z))>\dots >f(v_{i_d}(z)).
\end{equation} 
This implies that $\nu^{-1}(C)= \cup_{i=1}^dV_i$ is a union of disjoint open sets $V_i$ each of which is biholomorphic to $C$. Hence choosing a connected component $C$ we may speak about the uni-valued branches $v_1,v_2,\dots, v_d$ (where $v_i(z)\in V_i, \ z\in C$).% without worrying about possible monodromy.

  \begin{proposition} 
  \label{prop:tropicaltraceopen} In the above notation, 
  assume that $\Delta$  is a real analytic curve in $Y^\prime$ with  a locally  finite number of self-intersections. Given an open subset $U\subset Y$, assume that $\nu(U)$ is dense in $Y'$. 
  % \textcolor{red}{(*** Double-check notation ***)} 
   Then the following facts hold: 
  %\label{lm:continoustropicaltrace}
 % \begin{enumerate} 

 \smallskip
 \noindent
  {\rm (i)} $\nu_{*}f(z)\geq\nu_{U_*}f(z)$.
  
  \noindent
 {\rm (ii)} The trace $\nu_{U_*}f(z)$ is continuous on $Y'\setminus Cv$ if and only if 
   
   {\rm {a)}} in each connected component $C$ of $\mathcal{O}$, $\nu_{U_*}f(z)$  is equal to $f(v_{i}(z))$ for some $i=i(C)$, and 
   
   {\rm b)} if $C_1$ and $C_2$ share a boundary $\Delta^{ij}_{\mathfrak O} $, and $i(C_1)\neq i(C_2)$, then $\{ i,j\}=\{i(C_1),i(C_2)\}$.
  
  \noindent
  {\rm (iii)}  Up to $L^1_{loc}$-equivalence, the set  $\{\nu_{U*}f(z),\; U\subseteq Y\}$ of functions such that  $\nu_{U*}f(z)$ is continuous  is locally finite. Each of these functions is subharmonic and piecewise-harmonic, except possibly at its poles.

 % \end{enumerate}
  \end{proposition} %\textcolor{red} {SHOULD THIS BE A PART OF FORMULATION?}
 \begin{proof}  Observe first that  $\mathcal{O}=Y^\prime\setminus \Delta$ is open and dense and the boundary $\partial C\cap \mathfrak O$ in a neighbourhood $\mathfrak O $ of each connected component $C$ of $\mathcal{O}$ consists of a finite number of curves $\Delta^{ij}_{\mathfrak O}$. 
 
  Item (i) is trivial. Finiteness in item (iii) follows from item (ii) since the number of components $C$ is locally finite because $f$ and $\nu$ are real-analytic functions.  To settle  (ii), it suffices to prove that if $\nu_*f$ is continuous, then a) and b) hold. Let $A_i, i=1,\dots, d,$ be the subset of a connected component $C$ such that $\nu_{U_*}f(z)=f(v_i(z))\iff z\in A_i$. Then these sets are disjoint (since $C\subset \mathcal{O}$),  closed in $C$, and $C$ is their union. Since $C$ is connected, only one $A_i$ can be non-empty, which is exactly condition a). Furthermore, the continuity implies that the boundary $\Delta^\prime  $ between two components $C_1$ and $C_2$ with $i(C_1)\neq i(C_2)$ must be given by $f(v_{i(C_1)}(z))=f(v_{i(C_2)}(z))$. Finally, the subharmonicity in (iii) follows from the fact that  in a sufficiently small neighbourhood $M$ of any point in $\Delta^\prime  $ we have that
$$\nu_{U *}f(z)=\max \{ f(v_{i(C_1)}(z)),f(v_{i(C_2)}(z))\}\; \text{ for }\; z\in M.$$ 
Here we have analytically continued $v_{i(C_1)}$ and $v_{i(C_2)}$ across the boundary to all of $M$.
 \end{proof}

 Proposition~\ref{prop:tropicaltraceopen}
  implies that given that conditions a)  and b) are satisfied, most of the properties of the tropical trace described in 
 Theorem \ref{lm:tropicaltrace} remain true for  $\nu_{U_*}f(z)$ as well.

 \smallskip

\begin{remark}{\rm 
 Conditions (ii) a) and (ii) b) in  Proposition~\ref{prop:tropicaltraceopen} are requirements on the open set $U\subset Y$ that can be interpreted as follows. The branched covering $\nu$ together with the function $f$ induce a presheaf 
$F$ on $\mathcal{O}$ whose stalks are finite ordered (sub)sets of fibers $F(z):=(v_{i_1}(z),\dots ,v_{i_d}(z))$,  with the ordering of the indices given by \eqref{eq:orderedfibre}. For a connected component $C\subset \mathcal{O}$, the section is $F(C)=(V_{i_1}>V_{i_2}>...>V_{i_d})$, where $\nu^{-1}(C)=\cup_{i=1}^d V_i$ is the disjoint union of the different sheets $V_i=\{v_i(z),\ z\in C\}$ over $C$.
There is a sub-presheaf $F\cap U$ induced by the map $z\mapsto F(z)\cap U$, and the above condition  a) then says that, for a connected component $C\subset \mathcal{O}$, there is a maximal sheet $V_{i(C)}\subset U$ such that $U$ contains no elements of any larger sheets $V_j$.
   Condition  b) says that
 for two neighboring connected components in $\mathcal O$, either the maximal sheets in $U$ over each of these components are analytic continuations of each other or the boundary between these two components in $\mathcal O$ is determined by their equality after the fiberwise composition with $f$.}
 \end{remark}

   %\textcolor{red}{Kontrollera noga!}
   \begin{example} {\rm 
   Let $H_1(z):=\log\vert z\vert$ and $H_2(z):=\log\vert z-2\vert$ and let $Y$ and $Y'$ 
   be as in Example \ref{ex:trace}. Then $\Delta$ is the line given by $\mathrm{Re}(z)=1$, and there are three possibilities for a continuous $\nu_{U_*}f(z)$. Firstly, it can be equal to $\vert z\vert$ in the whole plane $\bC$ which occurs when e.g., $U=\mathbb C\times 1 $. Secondly, it can be equal to $\vert z-2\vert$ when e.g., $U=\mathbb C\times 2 $. Finally,  it can be equal to  $\max\{\log\vert z\vert,\log\vert z-2\vert \}$, when e.g., $U=Y$.}
   \end{example}
   Another concrete example of $\nu_*f$ and $\nu_{U_*}f(z)$ is given in Fig.~\ref{fig: trace} of \S~\ref{sec:quadratic}.
  
  \subsection{Branched push-forwards and  piecewise-analytic $1$-forms}
The derivative of a tropical trace is a piecewise-analytic differential $1$-form in the sense that we will now clarify.
\begin{definition}

{\rm Given a branched covering $\nu: Y \to Y^\prime$ of  compact Riemann surfaces,  by a \emph{uni-valued branch of $\nu$ } we mean an open subset $U\subset Y$ such that 
$\nu$ maps $U$ diffeomorphically onto its image $\nu(U)=Y^\prime\setminus \mC$, where $\mC$ is a finite union of smooth compact curves and points in $Y^\prime$.} 
\end{definition}

An easy way to  simultaneously construct several uni-valued branches for $\nu$ is to fix a cut $\mC\subset Y^\prime$ such that  

\noindent 
(i) $\mC$ contains all the branch points of $\nu$;

and

\noindent
(ii) $Y^\prime \setminus \mC$ consists of open contractible connected components. 

\smallskip
Then if $\deg \nu=d$ and $Y^\prime \setminus \mC$ is connected, the surface $Y\setminus \nu^{-1}(\mC)$ splits into $d$ disjoint  sheets such that  $\nu$ is a uni-valued function on each of these sheets. 
 \medskip
 \begin{definition} {\rm Given a meromorphic $1$-form $\omega$ on a compact Riemann surface $Y$ and a branched covering $\nu: Y \to Y^\prime$ of  degree $d$, where $Y^\prime$  is also a compact Riemann surface,  we define a \emph {branched  push-forward $\nu_*\omega$ } as a $d$-valued $1$-form on $Y^\prime$ obtained by assigning to a tangent vector $v$ at any point $p\in Y^\prime\setminus Cv$ one of the $d$ possible values $\omega(\nu^{-1}_j v),\;j=1,\dots, d$. Here $\nu^{-1}(v)_j$ is one of the $d$ possible pull-backs of $v$ to the tangent bundle of $Y$ and  $Cv$ is the set of all critical values of $\nu$. (Observe that $\nu$ is a local diffeomorphism near any point of $Y$ which is not its critical point.)}  
 \end{definition}

Using a somewhat fancier language, we can interpret the above definition as consideration of a set-theoretic section $\theta: Y'\to Y$ of the covering $\nu: Y\to Y'$, which  at each point of $Y'$ (with a finite number of exceptions) chooses one of the $d$ possible points in the fiber. This operation induces a branched  push-forward $\nu_*\omega$ as a set-theoretical section of the bundle of meromorphic 1-forms on $Y'$. (We can use set-theoretical sections, since we are not requiring any differentiability of $\theta$.)
 
 \smallskip
Now, in order to obtain a $d$-fold covering of an open subset of $Y^\prime$ by  disjoint sheets,  we want $\theta$ to satisfy certain conditions similar to those which we get by fixing an appropriate cut $\mC\subset Y'$. In other words,  we want to remove from $Y'$  a subset $E$ of Lebesgue measure $0$, and decompose the remaining surface into open domains on each of which $\theta$ is biholomorphic. More precisely, assume that 
$Y'\setminus Cv=\cup_{i=1}^n Y_i^\prime \cup E $, where all $Y_i^\prime$'s are disjoint open sets, %$E$ is a subset of Lebesgue measure $0$ and such that
and $\theta$ is a section of $\nu$ that is biholomorphic on each $Y_i^\prime$. %\textcolor{red}{(*** Formulera om den föregående meningen lite bättre. ***)} 
In this case we will say that the associated $1$-current $\nu_*\omega$ on $Y^\prime$ 
is a {\emph {piecewise-analytic} $1$-form}.
 
 \smallskip
The Cauchy transform of the asymptotic root-counting measure which we will construct later will be piecewise-analytic in the above sense. (Recall that we interpret the Cauchy transform as a $1$-current on $\bC P^1$, i.e., as a generalized 1-form.) The piecewise-analytic character of our construction stems from the fact that the  Cauchy transform will be associated to a section of a finite cover. 

\smallskip
The following relation to the tropical trace in the previous section is obvious. (We use a simple fact that a local variable $z$ on $Y^\prime$ is a local coordinate on $Y$ if $z\in Y^\prime\setminus Cv$.)

\begin{lemma}  In notation of Proposition \ref{prop:tropicaltraceopen}, assume  that $\nu_{U,*}f(z)$ is continuous in $Y^\prime\setminus Cv$. 
Then the 1-current  $\frac{\partial(\nu_{U,*}f(z))}{\partial z}$ is a branched push-forward $\nu_*\omega$ of $\omega:=\frac{\partial f}{\partial z}$ considered in the sense of distributions. 
\end{lemma}
\begin{proof} As the set $E$ of Lebesgue measure 0 in the above definition of a branched  push-forward take the set $\Delta$ defined in Proposition \ref{prop:tropicaltraceopen}. Furthermore, as the section used in the description of $\nu_*\omega$ take a connected component $C\subset Y^\prime\setminus \Delta$ biholomorphically equivalent to $V_{i(C)}$. Finally, use the fact that the distributional derivative of a piecewise-harmonic and continuous $L^1_{loc}$-function is equal to its usual derivative a.e., see e.g., \cite[Prop.2]{BB}.
\end{proof}

%In the above notation, if we have a branch \mC $\mC \subset Y^\prime$ which contains $Br_\nu$ and such that the complement $Y^\prime \setminus \mC$ consists of contractible connected components, then given $\nu_*\omega$ as above, we can determine its $d$ uni-valued branches which are defined everywhere on $Y^\prime\setminus \mC$. (If the complement $Y^\prime \setminus \mC$ is connected, then the set of branches is uniquely defined.) It might however happen that some uni-valued branches are well-defined in the complements to \mCs which do not have to satisfy all the previous requirements, i.e. they do not contain all the branch points or some connected components of the complement are not contractible, etc.

%If $\U$ is a uni-valued branch of $\nu: Y \to Y^\prime$ and $w$ is a meromorphic form on $Y$, we can define the push-forward $\nu_*\omega_{\vert_\U}$ which will be a meromorphic form defined in $\nu(\U)$.
%For any push-forward form $\nu_*\omega_{\vert_\U}$, 
%its poles might occur only as the image under $\nu$ of some pole $p\in Y$ of $\omega$. 
%where $p$ is not a critical point of $\nu$.

\medskip
\subsection{Defining affine Boutroux curves}
\label{sec:Boutroux}

Consider an irreducible affine algebraic curve $\Upsilon \subset \bC_z \times \bC_u,$ where the product $\bC_z \times \bC_u$ is equipped with coordinates $(z,u)$. 
%(Below $w$ will stand for the variable corresponding to the Cauchy transform of some measure in the $z$-plane.) 

\smallskip
Denote by $\widehat \Upsilon \subset \bC P^1_z\times \bC P^1_u$ the closure of $\Upsilon$. If $\Upsilon$ is the zero locus of the irreducible polynomial $f(z,u)$ then $\widehat \Upsilon$ is the image of the zero locus of the polynomial $q(z,t,u,s):=s^at^bf(z/t,u/s)$ given in $\bC^2\times \bC^2 $, under the product of two standard maps $\bC^2\to \bC P^1_z$ and $\bC^2\to \bC P^1_u$. Here $a,b$ are minimal (in the lexicographic order) positive integers such that $q$ is an irreducible polynomial. An alternative definition is that $\widehat \Upsilon\subset \bC P^1_z\times \bC P^1_u$ as a set coincides with the topological closure of $\Upsilon\subset \bC\times\bC $ in the ambient space $\bC P^1_z\times \bC P^1_u$.\\
%+++ok, Boris?+++ %We will use notation $\bC_w, \bC_z, \bC P^1_w,$ and $\bC P^1_z$ to indicate with which coordinate the corresponding coordinate (projective) line is associated. 

Let $\pi : \bC_z\times \bC_u \to \bC_z$
(resp. $\pi : \bC P^1_z\times \bC P^1_u \to \bC P^1_z$) be the standard projection onto the first %(second) 
coordinate. % and let $\pi_w$ be the similar projection onto the second coordinate.
Additionally, denote by $\mathfrak{n}: \widetilde \Upsilon\to \widehat \Upsilon$ 
the normalisation map. (Recall that the smooth compact Riemann surface $\widetilde \Upsilon$ is birationally equivalent to $\Upsilon$.)

\medskip 
 Now consider the standard meromorphic $1$-form
$$\Omega:=u\minispace\diff z$$
defined on $\bC_z\times \bC_u$ (resp. on $\bC P^1_z\times \bC P^1_u$).

\begin{remark}{\rm 
One can easily show that the zero divisor of $\Omega$ on $\bC P^1_z\times \bC P^1_u$ is a copy of $\bC P^1$ given by $u=0$; (the closure of) its pole divisor is the union of two intersecting copies of $\bC P^1$ given by $u=\infty$ and $z=\infty$}. %\textcolor{red}{(*** Dubbelkolla att min Ã€ndring efter semikolonet stÃ€mmer ***)}
\end{remark} 

Given a curve $\Upsilon \subset \bC_z \times \bC_u$ as above,
consider the meromorphic $1$-form $$\Omega_\Upsilon :=\Omega\vert_\Upsilon \quad (\text{resp.} \quad \Omega_{\widehat \Upsilon} :=\Omega\vert_{\widehat \Upsilon})$$ obtained by the restriction of $\Omega$ to $\Upsilon$ (resp. to $\widehat \Upsilon$). Denote by $\widetilde \Omega$ the pullback of $\Omega_{\widehat \Upsilon}$ to $\widetilde \Upsilon$ under the normalisation map $\mathfrak{n}: \widetilde \Upsilon\to \widehat \Upsilon$. This form will be the key 
ingredient below. %(By a slight abuse of notation, we denote the pullback of $\Omega_\Ga$ to $\widetilde \Gamma$ by the same letter.)

\smallskip 
\begin{remark}{\rm 
The zero divisor of $\Omega_{\widehat \Upsilon}$ consists of the intersection points $\widehat \Upsilon$ with the line $u=0$ and all the singularities of $\widehat \Upsilon\subset \bC P^1_z\times \bC P^1_u$. %CHECK 
The pole divisor of $\Omega_{\widehat \Upsilon}$ consists of all non-singular points of the intersection of $\widehat \Upsilon$ with the union of the lines $z=\infty$ and $u=\infty$.
}
\end{remark}

 %By our construction, the pullback $\pr^{-1}(f(z) d z)$ of the $1$-form $f(z)d z$ (which is well-defined near $\infty$ in $\bC P^1$) to the appropriate branch of $\widetilde\Gamma_f$ coincides in its domain of definition with the globally defined $\Omega_f$. 

%\begin{remark} {\rm Recall that the difference between the number of poles and the number of zeros (counted with multiplicities) of an arbitrary meromorphic $1$-form on an arbitrary compact Riemann surface $\Gamma$ equals the Euler characteristics $\chi(\Gamma)$. Another important circumstance is that the sum of all residues of an arbitrary meromorphic $1$-form on an arbitrary compact Riemann surface vanishes.}ÃÂ 
%s\end{remark}

   \medskip
  Further, given an irreducible affine curve $\Upsilon \subset \bC_z\times \bC_u$ as above and the corresponding meromorphic $1$-form $\widetilde \Omega$ on $\widetilde \Upsilon$, consider the multi-valued primitive function 
  $$\Psi(p)=\int_{p_0}^p \widetilde \Omega.$$
  $\Psi(p)$  is a well-defined uni-valued function on the universal covering of $\widetilde \Upsilon \setminus Pol$, where $Pol \subset \widetilde \Upsilon$ is the set of all poles of $\widetilde \Omega$ and $p_0\in \widetilde \Upsilon\setminus Pol$ is some fixed base point. The next statement is trivial, cf. Lemma~\ref{lm:trivial} in \ref{sec:imaginaryperiods}.

\begin{lemma}\label{lm:trivial3}
In the above notation, $\widetilde \Omega$  has purely imaginary periods if and only if the multi-valued primitive function $\Psi(p)$ has a uni-valued real part $\mathrm{Re}\,\Psi(p)$. In other words, $\mathrm{Re}\,\Psi(p)$ is a well-defined uni-valued function on $\widetilde \Upsilon \setminus Pol$. 
 \end{lemma}
The following class of curves has been introduced in \cite{Be,BM} and extensively studied there in the context of hyperelliptic curves and orthogonal polynomials.
 \begin{definition}
 \label{def:boutroux} {\rm A plane affine irreducible curve $\Upsilon\subset \bC_z\times \bC_u$  is called  an \emph{affine Boutroux curve } (aBc, for short) if the meromorphic $1$-form $\widetilde \Omega$ has purely imaginary periods on $\widetilde \Upsilon$. }
 \end{definition} 
 
 \begin{remark}{\rm 
 We can reformulate the latter definition as follows. 
  Let $\Upsilon_{sm}\subseteq \Upsilon$ be the smooth part of $\Upsilon$. Then $\Upsilon$ is an aBc if and only if the restriction of $\Omega=u\minispace\diff z$ to $\Upsilon_{sm}$  has on it purely imaginary periods. In fact,  this is equivalent to the requirement that $\Omega$ has purely imaginary periods on any smooth Riemann surface $\Upsilon_1\subseteq \widetilde \Upsilon $ such that $\widetilde\Upsilon\setminus  \Upsilon_1$ is a finite set.}
  \end{remark}

\subsection{How to construct affine Boutroux curves}
\label{sec:scheme} 
\bigskip
In this section we present an easy way to produce affine Boutroux curves.  A different combinatorial way to construct hyperelliptic  Boutroux curves can be found in \cite[App.~A-B]{BM}.  
After proving by brute force in \S~\ref{sec:perfect} that the curve \eqref{eq:algebraicDiffEq} is an aBc  we will later explain  that this statement is, in fact,  an instance of the construction in the present section, see \S~\ref{sec:Boutroux3}. 

\smallskip
 Let us first sketch the basic idea. 
We start with a real-valued harmonic function $H(z,v)$ on an open subset of $\bC_z\times \bC_v$, 
such that the holomorphic differential $$\diff'H=R_1(z,v)\,\diff z+R_2(z,v)\,\diff v$$ has bivariate rational coefficients. We assume that $H$ is harmonic on the set where both $R_1$ and $R_2$ are defined. Consider the curve $\E$  given by $R_2=0$, and change variables $(z,v)$ to $(z, u)$ where $u=R_1$ which implies that $\E$ embeds in  $(\bC_z\times \bC_u)$. Then (under certain additional genericity assumptions) this variable change will produce another curve  $\Upsilon$ which will be an aBc, since the real part of the integral of $udz$ will coincide with the restriction to $\Upsilon$  of the harmonic function $H/2$. 

\smallskip
Let us now explicate the details.
Expressing $R_1(z,v)=P_1(z,v)/Q_1(z,v)$ and $R_2(z,v)=P_2(z,v)/Q_2(z,v)$ with relatively prime polynomial numerators and denominators, let  $Q(z,v)$ be the least common multiple of the polynomials $Q_1(z,v)$ and $Q_2(z,v)$. Further assume that 

\smallskip \noindent
(*) $P_2(z,v)$ is an irreducible polynomial which is relatively prime with respect to $Q(z,v)$ and that $P_2(z,v),Q(z,v)  $ are not contained in the ring $\bC [z]$ (i.e. they do not depend only on the single variable $z$).

\medskip
Define the curve $\E\subset \bC_z\times \bC_v$ as the zero locus of $P_2(z,v)$, and let $\pi : \E\to \bC_z$ be the restriction to $\E$ of the standard projection  $\pi : (z,v)\to z$.  
Denote by $U\subset \bC_z\times \bC_v$ the complement to the zero locus of $Q$. Define by $\eta(z,v):=(z,R_1(z,v))$  the map 
$$\eta:\bC_z\times \bC_v\to \bC_z\times \bC_u,$$  and let $\Upsilon$ be the topological closure of $\eta(\E\cap U)$. We claim that $\Upsilon$ is an affine curve, and that the projection $ \Upsilon\to \bC_z$  has finite fibers and a dense image. 
\smallskip
The affine property of $\Upsilon$ is easiest to check using commutative algebra.
The affine ring of functions on $U$ coincides with the (localized) ring $\bC [z,v]_Q$ of rational functions whose denominators are powers of $Q$. The affine ring of functions on $\E\cap U$ is the domain given by 
$$\bC [z,v]_Q/(R_2)=\bC [z,v]_Q/(P_2).$$ 
The map $\eta$ restricted to $\E\cap U$ corresponds  to the map $\bC[z,u]\to \bC[z,v]_Q/(P_2)$,
 given by $z\mapsto z,\ u\mapsto R_1(z,u)$. The kernel $K$ of this map is an irreducible prime ideal  since  the irreducibility of  $P_2$ implies that $\bC[z,v]_Q/(P_2)$ is a domain. Finally, it is standard that
$$
 \Upsilon=
\{(z,u):\ p(z,u)=0\ \text{ for }p\in K \} \subset \bC_z\times \bC_u.
 $$ 
   To prove the finiteness,  notice first  that since $P_2 $ and $Q$ are relatively prime by (*), the set $\E\setminus \E\cap U$ is finite. Since $P_2$ is irreducible and not contained in $\bC [z]$ we get  that the map $$
\pi :\E\to \bC_z 
$$
has finite fibers and dense image.  Similarly, the projection $
\pi :\Upsilon\to \bC_z 
$
 has finite fibers. Indeed if the fiber over $z=a$ were infinite, this would imply that $K\subset (x-a)$, and since $K$ is prime,  that  $K=(z-a)$. This contradicts to the fact  that $\bC[z]\subset \bC[z,v]_Q/(P_2) $.
 %\textcolor{red}{Rikard, please, check the above and below since I changed the notation back and forth}

\medskip
\begin{proposition} In the above notation, the curve  $\Upsilon \subset \bC_z\times \bC_u$ is an aBc. 
 \label{prop:perfectconstruction}
\end{proposition}
\smallskip
\begin{proof}  
Except for a finite number of points $M^\prime\subset \Upsilon$,  $z$ is a local coordinate on $\Upsilon$ and the map $\pi:\Upsilon\to \bC_z$ is a branched cover.
By the remark following Definition \ref{def:boutroux} of affine Boutroux curves, it suffices to check whether the periods of $u\minispace\diff z$ on the open and Zariski dense subset $\Upsilon\setminus M^\prime$ are purely imaginary.  Indeed, let $H_\Upsilon$ be the restriction to $\Upsilon$ of the  harmonic function $H(z,v)$ that we started with. Observe that 
$\diff'H_\Upsilon=P\minispace\diff z=u\minispace\diff z$  on $\Upsilon$. The statement then trivially follows  from the fact that (up to an additive constant) the real part of a primitive function $\int u\minispace\diff z$ (which is well-defined on the universal cover) is actually given by the uni-valued function $\frac{1}{2}H_\Upsilon$ defined on $\Upsilon$. 
%\textcolor{red}{(*** ErsÃ€tt $H$ med $H_\E$ p\aa{} n\aa{}gra st\"allen? INTE s\aa{} bra notationer***)}
\end{proof}

\begin{example}{\rm 
Set $H(z,v)=2(\log\vert v^2+1\vert -\log\vert v-z\vert)$ which is pluriharmonic except for $v=\pm i$ and $v=z$. Its differential is given by  
$$
\diff'H=\frac{1}{v-z}\,\diff z+\left(\frac{2v}{v^2+1}-\frac{1}{v-z}\right)\diff v=R_1(z,v)\diff z+R_2(z,v)\diff v.
$$
In the above notation, $P_2(z,v)=2v(v-z)-(v^2+1)=v^2-2vz-1$ and $Q(z,v)=(v-z)(v^2+1)$, and hence $\E \cap U$ is the open subset of the curve $v(v-2z)=1$ in $\bC_z\times \bC_v$ where $v\neq \pm i$ and  $v\neq z$. Clearly the fibers of the map $\pi:\E \to \bC_z$ have cardinality at most 2. Since $R_1=\frac{1}{v-z}$ we set $u=\frac{1}{v-z}\iff v=z+\frac{1}{u}$. Substituting this in the relation $v(v-2z)=1$, we obtain $(z^2+1)u^2-1=0$. Hence $K$ is the prime ideal generated by the polynomial $(z^2+1)u^2-1$, and the irreducible curve $\Upsilon$ defined by $(z^2+1)u^2-1=0$ is an aBc. (This is the special case $\alpha=1$ and $P(u)=u^2+1$ of \S~\ref{sec:Boutroux3} which corresponds to the Legendre polynomials.) }
\end{example}

 \subsection{Affine Boutroux curves, induced signed measures on $\bC P^1$,   and their Cauchy transforms}
 \label{sec:BoutrouxCauchy}
\medskip
In this subsection we will combine the notions from the previous sections to show that given an affine Boutroux curve $\Upsilon\subset \bC_z\times \bC_u$, we can  under some additional assumptions construct a signed measure on $\bC P^1_z$ whose Cauchy transform %(thought of as a function???) 
satisfies the algebraic equation defining $\Upsilon$.

\medskip
Indeed, given a plane curve $\Upsilon\subset \bC_z\times \bC_u$ as in \S~\ref{sec:Boutroux}, we have a natural meromorphic function $\tau: \widetilde \Upsilon \to \bC P^1_z$  induced by the composition of the normalisation map $\mathfrak{n}: \widetilde \Upsilon\to \widehat \Upsilon$  with the  standard projection   $\pi : \widehat \Upsilon\subset \bC P^1_z\times \bC P^1_u\to \bC P^1_z$. The next result will be crucial later.

\begin{theorem}\label{th:main} In the above notation, given an affine Boutroux curve $\Upsilon \subset \bC_z \times \bC_u$ such that: 

{\rm a)} near the line $\{z=\infty\}$ the curve $\widehat \Upsilon \subset  \bC P^1_z \times \bC P^1_u$ consists of smooth branches transversally intersecting this line;

{\rm b)} the restriction of the canonical form $\Omega=u\minispace\diff z$ to the latter  branches of $\widehat \Upsilon$ has simple  poles among which there exists a unique one with minimal negative residue; 

\smallskip
\noindent
then there exists a  signed measure $\bar \mu_\Upsilon$ of total mass $0$ supported on $\bC P^1_z$ with the following properties: 
 
 \smallskip \noindent 
{\rm (i)} its support $S_{\bar \mu_\Upsilon}:=\mathrm{supp} (\bar\mu_\Upsilon)\subset \bC P^1_z$ consists of finitely many compact real analytic curves 
and isolated points; 
 
 \smallskip
 \noindent
{\rm (ii)} the support $S_{\bar \mu_\Upsilon}^-$ of the negative part of $\bar\mu_\Upsilon$ coincides with $\tau\left(Pol^-_{\Omega}\right)\subset \bC P^1_z$;

\smallskip \noindent
 {\rm (iii)} its Cauchy transform     $ \;\bar \C_{\bar \mu_\Upsilon}$   (considered as a  $1$-current; see  \S~2.1) coincides with a uni-valued piecewise-analytic branch  of $\tau_*\widetilde \Omega$  in $\bC P^1_z\setminus S_{\bar \mu_\Upsilon}$. In other words,  if we write $\bar \C_{\bar \mu_\Upsilon}=\C(z)\minispace\diff z$ in the affine chart $\bC_z\subset \bC P^1_z$, where $\C(z)$ is piecewise-analytic in $\bC_z\setminus S_{\bar \mu_\Upsilon}$, then $\C(z)$ satisfies there the algebraic equation defining the aBc $\Upsilon$.

\end{theorem}

\smallskip
\begin{remark} {\rm Observe that the points at which the branches in Condition a) intersect  the line $\{z=\infty\}$ do not have to be distinct. Furthermore,  the restriction of the canonical form $\Omega=u\minispace\diff z$ to all branches of $\widehat \Upsilon$ near the line $\{z=\infty\}$ in Condition b) must have  poles at all points of its intersection with this line since $\Omega$ has poles along this line in $\bC P^1_z\times \bC P^1_u$. The residues at all these poles must be real since $\Upsilon$ is an aBc. Thus the only essential requirement in  this condition is the existence of a unique minimal  negative residue.  We suspect that the requirement of  simplicity of poles in item b) can be  weakened with the same conclusions as in Theorem~\ref{th:main}.  

}\end{remark}

\smallskip
\begin{remark}
{\rm Prior to proving Theorem~\ref{th:main}, observe that, in general, its converse is  false, i.e., there exist curves  for which conditions (i), (ii) and (iii) hold, but  which are {\bf not}  necessarily affine Boutroux curves, see e.g., \S~4  of \cite{BoSh}. Thus being an aBc  provides a sufficient (but not necessary) condition for the validity of (i) -- (iii). Observe additionally  that if we remove condition (iii), then there exist situations in which $\bar \mu_\Ga$ is not unique,  see e.g., Theorem~4 of \cite{ShTaTa}.   %; though we ignore whether uniqueness is true if all the conditions of the theorem are enforced.  ??? 
  We   also want to point out a close connection of Theorem~\ref{th:main} with some results of \cite{BarSh} where condition ${\rm (iii)}$ is called the existence of  \emph{clean poles}. }
\end{remark}

\begin{proof}[Proof of Theorem~\ref{th:main}] Choose an arbitrary point $p_0\in \widetilde\Upsilon\setminus  Pol_{\widetilde\Omega}$ and,  as in Lemma \ref{lm:trivial2}, consider the function $$\text{Re } \Psi(p)=\text{Re }\left[ \int_{p_0}^p \widetilde \Omega\right],$$ 
where $\Omega=u\minispace\diff z$ and $ \widetilde \Omega$ is its pullback to the normalization $\widetilde\Upsilon$. Note that
$$\frac{\partial\text{Re } \Psi}{\partial z}\diff z= \frac{\Omega}{2}.$$
Since $\Upsilon$ is an aBc,  then  $\widetilde \Omega$ has purely imaginary periods on $\widetilde\Upsilon$ and $\text{Re } \Psi(p)$ is a  uni-valued harmonic function on $\widetilde \Upsilon\setminus Pol_{\widetilde\Omega}$. (One can consider $\text{Re } \Psi(p)$ as defined on all of $\widetilde \Upsilon$ if one allows it to attain values $\pm \infty$.) Let $Cr\subset \widetilde\Upsilon$ be the set of critical points of the meromorphic function $\tau: \widetilde \Upsilon \to \bC P^1_z$  obtained as the composition $\tau=\mathfrak{n}\circ \pi $ and let $\tau(Pol)\subset \bC P^1$ be the image of the set $Pol\subset \widetilde\Upsilon$ of all poles of $\widetilde \Omega$.   Recall that  the finite set $Cv\subset \bC P^1_z$ is defined as the set of all critical values of the meromorphic function $\tau: \widetilde\Upsilon \to \bC P^1_z$. %(Obviously, $Cv$ is a finite set.)

\smallskip
Now, for  any $z\in \bC P^1_z\setminus \left(Cv \cup \tau(Pol)\right)$, define the function $\Theta_\Upsilon$ on $\bC P^1_z$ given by 
$$\Theta_\Upsilon(z):=\max_{v_i(z)\in \tau^{-1}(z)} \{ \text{Re } \Psi(v)\}.$$   In other words,  $\Theta_\Upsilon(z)$ is the \emph{tropical trace} of the projection $\tau$ of the function $\text{Re } \Psi$ to $\bC P^1_z$. 
  
  \medskip
Observe that if $z$  lies in  $\bC P^1_z\setminus \left(CV\cup \tau(Pol)\right)$, then it is a local parameter on every  branch of  $\widetilde \Upsilon$ near each point belonging to  the fiber $\tau^{-1}(z)$, which implies that  each function $H_i(z):=\text{Re } \Psi(v_i(z))$ is a well-defined harmonic function near $z$. Moreover,  outside of its poles,   $\Theta_\Upsilon(z)$ is  a continuous subharmonic function.%, whose Laplacian is supported on the union   of the level sets $\Re \Psi(p_i)=\Re \Psi(p_j)$, for $i\neq j$.  
 
 \medskip
The above definition of $\Theta_\Upsilon(z)$ also makes sense  if $z$ is  a critical value or the image of a pole; in the latter case  $\Theta_\Upsilon(z)$  might attain infinite values. Namely, if  $v_i \in  \widetilde\Upsilon$ is a pole with residue $r$ and $z_v := \tau(v)$, then locally near $z_v$ the corresponding $H_i(z)$ has the  asymptotic  $r\log\vert z-z_p\vert$. Hence, if $r$ is positive, then $\lim_{z\to z_p}H_i(z)=-\infty$ in a sufficiently small  neighbourhood of $z_p$.  Analogously,  if $r$ is negative, then $\lim_{z\to z_p}H_i(z)=+\infty$. %Finally, the behavior of $H_i(z)$ at infinity in the $z$-plane is governed by $\Re s_i$.  ??? 
 %Taking the maximum over the fibers of $\tau$,  we get a globally defined on $\bC P^1_z$  function $\Theta$   which is subharmonic outside $\tau(Pol^-)$. It always tends to $+\infty$ at the image of a pole  with a negative residue. 
 Finally,   $\Theta_\Upsilon(z)=-\infty$ if and only if every point in   $\tau^{-1}(z)$ is a pole of 
 $\widetilde \Omega$ with a positive residue.   Near $\infty \in \bC P^1_z$,  the tropical trace $\Theta_\Upsilon(z)$  has the asymptotic $-r_{min}\log|z|$, where $r_{min}$ is the unique minimal negative residue guaranteed by  Condition b). 
 \medskip
 
Now let us define the $2$-current $\bar \mu_\Upsilon$ on $\bC P^1_z$ as given by 
\begin{equation}\label{Levy}
\bar \mu_\Upsilon:=\frac{1}{2\pi}\left(\frac{\partial^2 \Theta_\Upsilon}{\partial x^2}+ \frac{\partial^2 \Theta_\Upsilon}{\partial y^2} \right) \,  \diff x\diff y=\frac{i}{\pi}\frac{\partial^2\Theta_\Upsilon}{\partial z\partial \bar z}\,\diff z\diff \bar z,%=\Delta \Theta_\Ga d  x \wedge d  y,
\end{equation}
 where $(x,y)$ are the real and the imaginary parts of the affine coordinate $z$.  %(See more details in Appendix 2.)   
We will call the  function $\Theta_\Upsilon$   %which is subharmonic outside $\tau(Pol^-)$ 
the \emph{logarithmic prepotential} of the $2$-current $\bar \mu_\Upsilon$. 

\medskip
The $2$-current $\bar \mu_\Upsilon$ given by \eqref{Levy} satisfies conditions (i)-(ii) of Theorem~\ref{th:main} which immediately follow from Theorem \ref{lm:tropicaltrace} saying that $\bar \mu_\Upsilon$  is actually a signed measure on $\bC P^1$  supported on finitely many segments of analytic curves belonging to the level sets $\{\text{Re } \Psi(v_i(z))=\text{Re } \Psi(v_j(z))\}$, $i\neq j$, and finitely many isolated points including $\tau(Pol^-)$ and possibly some part of $\tau(Pol^+)$.  By the above asymptotic of  $\Theta_\Upsilon(z)$ at $\infty$, the $2$-current $\bar \mu_\Upsilon$ has a negative point mass $r_{min}$ at $\infty\in \bC P^1_z$. The rest of its support lies in a bounded domain in the affine chart $\bC_z=\bC P^1_z\setminus \infty$.  Observe that since $\bar \mu_\Upsilon$ has a (pre)potential, it must necessarily be exact which is equivalent to  
\begin{equation}
\label{eq:exact potential}
\int_{\bC P^1} \bar \mu_\Upsilon=0.
\end{equation}
(Since $\bar \mu_\Upsilon$ is a $2$-current on the $2$-dimensional manifold $\bC P^1$ it is automatically closed; in order to be exact its integral over $\bC P^1$ must vanish which is given by \eqref{eq:exact potential}. Observe that $H_2(\bC P^1, \mathbb Z)=\mathbb Z$.) 
 Finally notice that the negative part of $\bar \mu_\Upsilon$ is supported  on  $\tau(Pol^-)$, by construction. 

\smallskip
To settle (iii), assume that $V\subset \bC P^1\setminus\{\supp (\mu_\Ga) \cup \B\cup \tau(\textit{Pol\,})\}$  is a  simply connected subset. Then
$\Theta_\Upsilon(z)=\Psi(\nu_U(z))$ for a certain choice of a branch $\nu_V: V\to \Upsilon\subset \bC_z\times \bC_u$. Set $\nu_V(z)=(u(z),z)$,  
where $u$ is a branch of the algebraic function defined by $\Upsilon$. Clearly $z$ is a local coordinate both in $V$ and in $\nu_V(V)$ and therefore 
$$
d'(\Theta(z))=\frac{\partial\Theta(z)}{\partial z}\diff z=\frac{
u(z)
}{2}
\diff z.
$$
Thus $2\frac{\partial\Theta(z)}{\partial z}$ satisfies the equation defining $\Upsilon$. 

Next restrict $\bar \mu_\Upsilon$ to the affine chart $\bC_z=\bC P^1_z\setminus \infty$, denote this restriction by $ \mu_\Upsilon$, and define the (usual) logarithmic potential of $ \mu_\Upsilon$ as 
 $$L_{\mu_\Upsilon}(z):=\int_\bC \ln|z-\xi|\,\diff \mu_\Upsilon(\xi).$$ 
 As we explained in \S~2.1, $$\bar \mu_\Upsilon=\frac{1}{2\pi} \left( \frac{\prt^2 \bar L_{\bar \mu_\Upsilon}}{\prt x^2}+ \frac{\prt^2 \bar L_{\bar \mu_\Upsilon}}{\prt y^2} \right)\,\diff x\wedge\diff y=\frac{2}{\pi}\frac{\prt^2 \bar L_{\bar \mu_\Upsilon}}{\prt z \prt \bar z}\,\diff x\wedge\diff y=\frac{i}{\pi}\frac{\prt^2 \bar L_{\bar \mu_\Upsilon}}{\prt z \prt \bar z}\,\diff z \wedge\diff \bar z,$$
 where $\bar L_{\bar \mu_\Upsilon}$ is the the latter logarithmic potential $L_{\mu_\Upsilon}(z)$ considered as a $L^1_{loc}$-function on $\bC P^1$. 

Observe that  the application of Laplace operator to both $\Theta_\Upsilon$ and $\bar L_{\bar \mu_\Upsilon}$ defined in $\bC P^1$ gives exactly the same measure $\bar \mu_\Upsilon$.  (This is clear in the affine plane, and for the isolated point mass at $\infty$  follows from \eqref{eq:exact potential} together with the definition of $\bar \mu_\Upsilon$.)  

Hence the difference  $\Theta_\Upsilon - \bar L_{\bar \mu_\Upsilon}$ is a global harmonic function of the whole $\bC P^1$. Thus this difference has to be constant. Therefore, $$\C_{\mu_\Upsilon}=2\frac{\partial L_{\bar \mu_\Upsilon}(z)}{\partial z}=2\frac{\partial\Theta(z)}{\partial z}$$
has to satisfy the algebraic equation defining $\Upsilon$.
 \end{proof}

\section{Affine Boutroux curves related to Rodrigues descendants}\label{sec:perfect}

We start with the observation that, after a  scaling of the Cauchy transform,  the second part of Theorem~\ref{th:Cauchy} is equivalent to the following claim.

\begin{proposition}\label{lm:simplpol}
For the asymptotic root-counting measure $\mu_{\al,P}$ as in Theorem~\ref{th:Cauchy}, its scaled Cauchy transform  $\W$  defined by 
\begin{equation}\label{eq:scaled}
\W:=\W_{\al,P}:=\frac{d-\alpha}{\alpha}\C_{\al,P}
\end{equation}

 satisfies a.e. in $\bC$  the algebraic equation
 
\begin{equation}\label{eq:algebraicDiffEq2}
\sum_{k=0}^{d} \frac{\alpha-k}{k!}\,P^{(k)}\W^{d-k} = 0.
\end{equation}
\end{proposition} 
%(Equation~\eqref{eq:algebraicDiffEq2} is simpler than the original equation \eqref{eq:algebraicDiffEq}.)

\smallskip
In what follows we will denote

%\noindent
%{\rm (i)} by $\Ga_{\al,P/Q}\subset \bC_\C\times \bC_z$  the affine algebraic  curve given by \eqref{eq:qnmAlgebraicEq3};

\smallskip
\noindent
{\rm (i)} by $\Ga:=\Ga_{\al,P}\subset \bC_z\times \bC_\C$  the affine algebraic  curve given by \eqref{eq:symbolcurve1} or equivalently \eqref{eq:algebraicDiffEq}

\smallskip
and  

\smallskip
\noindent
{\rm (ii)}  by $\La:=\Lambda_{\al,P}\subset \bC_z\times \bC_\W$ the affine algebraic  curve given by \eqref{eq:algebraicDiffEq2}.

\smallskip 
We will refer to $\Ga_{\al,P}$ as the \emph{symbol curve} of the pair $(\al,P)$ and to $\Lambda_{\al,P}$ as the \emph {scaled symbol curve} of $(\al,P)$. %Observe that equation \eqref{eq:symbolcurve1} which we will use later is just another form of equation \eqref{eq:algebraicDiffEq}. Therefore the symbol curve $\Ga_{\al,P}$ can be equally well defined by \eqref{eq:symbolcurve1}. 
To prove the second part of Theorem~\ref{th:Cauchy}  (or, its equivalent Proposition~\ref{lm:simplpol}) using the above Theorem~\ref{th:main}, we need to study in detail the algebraic curves  $\Ga_{\al,P}$ and $\Lambda_{\al,P}$. 

\medskip
 Our goal is to show that for any strongly generic $P$ and $0<\al<\deg P$, the  irreducible curve   $\Ga_{\al,P}\subset \bC_z\times \bC_\C$  is an  aBc as defined above. In fact, we will prove this property for  the curve    $\Lambda_{\al,P}$. 
Since $\Ga_{\al,P}$ is obtained from $\La_{\al,P}$  by a real scaling of the first coordinate, the claim that  $\Ga_{\al,P}$ is an  aBc follows from that for $\La_{\al,P}$. 

\begin{remark}{\rm 
Observe that   \eqref{eq:algebraicDiffEq} defines the closure    $\widehat \Ga_{\al,P}$  in $\bC P^1_z\times \bC P^1_\C$ of the bidegree $(d,d)$. By the adjunction formula, a smooth curve in $\bC P^1_z\times \bC P^1_\C$ of the bidegree $(d,d)$ has genus $(d-1)^2$. However,  the curve given  by \eqref{eq:algebraicDiffEq} is rational and therefore highly singular.}
\end{remark}

The next  technical theorem explicates the algebraic geometric properties of  $\Lambda_{\al,P}$ and its canonical differential $\Omega=\W\minispace\diff z$ which are central for the application of the tropical trace to our problem.               
   %We have not strived for the maximal possible  generality in its formulation and, in order to  simplify the notation, we  formulated it only for $R=P/Q$ and $ \Gamma_{\al,P/Q}$. (The case of a polynomial $P$ and $ \Gamma_{\al,P}$ follows upon the substitution $Q=1$.) I
   In Theorem~\ref{pr:rational} below,  $\widehat \La:=\widehat \La_{\al,P}$ denotes the closure of $ \La_{\al,P}$ in $\bC P^1_z\times \bC P^1_\W$, and $\widetilde \La:=\widetilde \La_{\al,P}$ denotes the normalisation of 
$\widehat \La_{\al,P}$. Recall that $\W=\frac{d-\al}{\al}\C$ is a coordinate obtained by rescaling the coordinate $\C$. 

\begin{theorem}\label{pr:rational} Let $P$ be a strongly generic polynomial of degree $d\ge 2$ and let $0<\al< d$ be a positive number. %Assume that $P$ and $P'$   have only simple zeros.  %DO WE NEED $P^\prime$ also?
Then 
the algebraic curve  $ \La \subset \bC_z\times \bC_\W$   given by  \eqref{eq:algebraicDiffEq2}  %{eq:qnmAlgebraicEq3}    % (resp. $\Gamma_{\al,P}\subset \bC_\C\times \bC_z$  given by \eqref{eq:algebraicDiffEq}) 
is an aBc. %{\color{red} aBC nämns inte i beviset---kanske plocka bort det bevisas ju senare, eller lägga till ett argument. Dessutom ska $(0,\infty) \to (\infty, 0)$ in some places. Jag markerade några av dem.} 
More exactly, the following properties hold:  

\medskip
\noindent
\rm{(i)}  $\La$ is an irreducible rational curve.

\smallskip
\noindent
\rm{(ii)} The inverse image $\pi^{-1}(\infty)\subset \widehat \La\subset \bC P^1_z\times \bC P^1_\W$ %(resp.  $\pi_z^{-1}(\infty)\subset\widehat \Gamma_{\al,P}$)   
consists only of $(\infty,0)$; that is, $\infty\in \bC P^1_z$ is a complete ramification point of the  function  $\tau: \widetilde \La\to \bC P^1_z$.% (resp. $\tau: \widetilde \Gamma_{\al,P}\to \bC P^1_z$). 

\smallskip
\noindent
\rm{(iii)} %If  $R=P/Q$ has simple zeros and poles,  %(resp. $P$ has simple zeros),  
The equation defining the slopes $s$ of different branches of $\widehat \La$ %(resp. $\widehat \Ga_{\al,P}$) 
at $\infty\in \bC P^1_z$ is given by 
 \begin{equation}\label{eq:slopes}
 (s+1)^{d-1}(\al(s+1)-d)=0.
 \end{equation}
 It has only two distinct solutions which are  $s=\frac{d}{\al}-1$  and  $s=-1$ of multiplicity $d-1$. 

\smallskip
\noindent
\rm{(iv)}  %If  $R=P/Q$ has simple zeros and poles,  %(resp. $P$ has simple zeros),
 The only singularity of $\widehat \La\subset  \bC P^1_z\times \bC P^1_\W$ %(resp. $\widehat \Gamma_{\al,P}\subset  \bC P^1_\C\times \bC P^1_z$)
  is  $(\infty,0)$. As a consequence, the normalisation map $\mathfrak{n}:  \widetilde \La \to \widehat \La$ is one-to-one at all points except for $(\infty,0)\in \widehat \La$ whose preimage consists of $d$ points of $\widetilde \La$.

\smallskip
\noindent 
\rm{(v)} %If  $R=P/Q$ has simple zeros and poles  %(resp. $P$ has simple zeros) 
%and $R^\prime$ has simple roots, % (resp. $P^\prime$ has simple roots), 
All $d$ local branches of $\widehat \La$ %(resp. $\widehat \Gamma_{\al,P}$)  
at the point $(\infty,0)$  are smooth. 
%(This complete ramification point reduces the genus by $(d-1)^2$.)

\smallskip
Finally, the set of all poles of $\widetilde \Omega= \mathfrak{n}^{-1}(\W\minispace\diff z)$ on $\widetilde \La$ is described in (vi) -- (vii) below.

\smallskip
\noindent
\rm{(vi)}  $\widetilde \Omega$  has a simple pole at each of the points $p_i\in\widetilde \La,\; i=1,\dots,d$, 
whose images are given by $\mathfrak{n}(p_i)=(z_i,\infty)\in\widehat \La\subset \bC P^1_z\times \bC P^1_\W$, where $z_i$ runs over the set of zeros of $P$. At each such point $p_i$,  $\widetilde \Omega$ has the same residue equal to $\frac{1-\al}{\al}$. 

%\smallskip
%\noindent
%\rm{(vii)}
% $\widetilde \Omega$ has a simple pole at each of the points $q_j\in\widetilde \La,\; j=1,\dots , \deg Q$ 
%whose image $n(q_j)=(\infty, \zeta_j)\in\widehat \La\subset \bC P^1_\W\times \bC P^1_z$, where $\zeta_J$ runs over the set of zeros of $Q$. At each such point $q_j$ the residue of $\widetilde \Omega$ equals $-\frac{1+\al}{\al}$.

\smallskip
\noindent
\rm{(vii)}
 $\widetilde \Omega$ has a pole with real residue at each of  the $d$ preimages  of the singular point  $(\infty,0)\in \widehat \La$ %(resp. $(0,\infty)\in \widehat \Gamma_{\al,P}$) 
 under the normalization map $\mathfrak{n}:  \widetilde \La \to  \widehat \La$. This residue equals $1$ for each of the $d-1$ preimages coming from the branches with slope $-1$ at $(0,\infty)$ and the remaining residue equals  $\frac{\al-d}{d}$ for the preimage coming from the branch with the   slope  $\frac{d}{\al}-1$. %(resp. $n:  \widetilde \Gamma_{\al,P} \to  \widehat \Gamma_{\al,P}$). 
 \end{theorem}

\smallskip
In what follows we will refer to the solution $s=\frac{d}{\al}-1$ of the equation \eqref{eq:slopes} as the \emph{essential slope} since it defines the asymptotic at $\infty$ of the scaled Cauchy transform $\W(z)$, see \eqref{eq:scaled}.

\begin{remark} {\rm Condition (vii) implies that on the curve $\widetilde \Ga$ which is the normalization of $\widehat \Ga \subset \bC P^1_z\times \bC P^1_\C$, the canonical form has poles at all $d$ preimages of the point $(\infty,0)\in \bC P^1_z\times \bC P^1_\C$, $d-1$ of which have the same positive residue $\frac{\al}{d-\al}$ and the remaining point has residue $-1$. The latter value is related to the fact that the Cauchy transform of the asymptotic root-counting measure $\mu_{\al,P}$ (which is a compactly supported probability measure) has the standard asymptotic $\frac{1}{z}$ near $\infty$ in the $z$-plane.}
\end{remark} 

\begin{proof}[Proof of Theorem~\ref{pr:rational}]%Below instead of the curve $\Gamma$,   we will investigate  the curve $\Lambda$ satisfying the equivalent equation  \eqref{eq:qnmCauchyTransformLike}    for the scaled Cauchy transform $\W(z)$. 

To prove (i), observe that the global rational change of variables  $(\W=\W,\, \tilde z=z+\W^{-1})$ transforms \eqref{eq:algebraicDiffEq2}   %{eq:qnmCauchyTransformLike}   %(resp.  \eqref{eq:CauchyOfPLike}) 
into 
\begin{equation}
\label{eq:extra+}
\al \W=  \frac {P^\prime(\tilde z)}{P(\tilde z)}.
\end{equation} 
Since $\al\neq 0$, this equation allows us to consider $\W$ as the graph of a rational function in the variable $\tilde z$. The latter fact implies that $\Lambda$  is a rational curve, and in addition,   $\Lambda$ is irreducible since it is a graph.

\medskip
To prove (ii), we argue as follows. Assuming that all zeros of $P(z)=(z-z_1)\dotsm(z-z_{d})$ are simple, we obtain
\begin{equation}\label{eq:extra}
\al \W=\frac{P^\prime(z+\W^{-1})}{P(z+\W^{-1})}=\sum_{i=1}^{d}\frac{1}{z+\W^{-1}-z_i}.
\end{equation}

Substituting $z=\frac{1}{y}$ in  \eqref{eq:extra} and clearing the denominators, we get 
\begin{equation}\label{eq:mod}
\al \widetilde P =y\sum_{j=1}^{d}\widetilde P_j,
\end{equation}
where
$$\widetilde P:= \prod_{i=1}^{d}(y+\W-z_iy\W)\quad\mathrm{and}\quad\widetilde P_j:=\frac{\widetilde P}{y+\W-z_jy\W}.$$ 

\smallskip
 To obtain the fiber over $z=\infty\in \bC P^1_z$, i.e., over the point $y=0$, one should  substitute $y=0$ in \eqref{eq:mod}. One can easily check that the result of this substitution is 
$\al  \W^d=0,$ implying that the only point in the fiber $\tau^{(-1)}(\infty)$ is $\W=0$. (This argument works even if $P$  does not have simple zeros.) 

\smallskip
To settle (iii), we need to calculate the slopes of the branches of $\Lambda_{\al,P}$ %(resp. $\Lambda_{\al,P}$)
 at $(0,\infty)$, for which one should substitute $\W=s(y)y$ in \eqref{eq:mod}. These slopes coincide with $s:=s(0)$. After the substitution $\W=s(y)y$ in \eqref{eq:mod}, the factor $y^{\deg P}$ can be cancelled on both sides, which then, by letting $y=0$, yields
$$\al\prod_{i=1}^{d}(1+s)=d \prod_{i=1}^{ d-1}(1+s)
$$
or, equivalently,
$$\al (1+s)^d=d(1+s)^{d-1} \Leftrightarrow 
 (1+s)^{d-1}(\al(s+1) -d)=0,
$$ 
which is the required statement.

\medskip
To prove (iv), we need to show that there are no singularities of $\widehat \La$ %(resp. $\Lambda_{\al,P}$) 
above the affine part of $\bC P^1_z$, i.e., for all  $z\neq \infty$ and $\W\in \bC P^1$. Notice first 
that $\W=0$ is impossible for finite $z$. Further notice that $\W=0$ is equivalent to $\W^{-1}=\infty$ and rewrite  (\ref{eq:extra}) as
\begin{equation}
\label{eq:rewrite}
G(\W^{-1},z)=\alpha P(z+\W^{-1})-\W^{-1} P'(z+\W^{-1})=0.
\end{equation}
A simple calculation shows that the coefficient of the highest power of $\W^{-1}$ is $\al-d,$ which is negative since by our assumption $\al<d$. Hence, $z$ finite and $\W=0$ is impossible.
 In other words, the curve $\widehat \La$ intersects the coordinate line $\bC P^1_z$ only at $z=\infty$, and its part  $\La\subset \bC_z\times \bC_\W $ is contained in the Zariski open set $A\subset \bC_z\times \bC_\W$ given by $ \W\neq 0$. 
  
 \smallskip 
Secondly, observe that the rational change of coordinates $(\W,z)\mapsto (\W,\tilde z)$  given by  $\W=\W,\; \tilde z=z+\W^{-1}$ is a diffeomorphism between the above open set $A: \subset \bC_z\times \bC_\W$ and the open set  $B\subset \bC_{\tilde z}\times \bC_\W$ given by $\W\neq 0$.  
The curve given by (\ref{eq:extra+}) in the coordinates $(\tilde z,\W) $  is clearly smooth when $\tilde z$ is not a root of $P$. Additionally,  $\W=\infty$ at any root of $P$ implying that our curve is smooth in all of $B$. Any diffeomorphism preserves the smoothness property, and hence $\widehat \La$ is smooth in $A$. By the first observation, it is  therefore smooth at all points in $\bC_z\times \bC_\W$.

It remains to check the points of $\widehat \La$ with $\W=\infty$,
which occurs exactly at the roots of $P$. We can do this  by setting $\W^{-1}=0$ in (\ref{eq:extra}). Assume that $p\in \bC P^1_z\times \bC P^1_\W$ is of the form $(z_i,\infty)$ where   $P(z_i)=0$ and that $p$ a singular point of $\widehat \La$. Then at this point $p$, the partial derivatives of  $G(z,\W^{-1})$ in \eqref{eq:rewrite}  with respect to the variables $ \W^{-1}$ and $z$  must vanish. A short calculation  shows that 
$$\frac{\partial G(z,\W^{-1})}{\partial \W^{-1}}-\frac{\partial G(z,\W^{-1})}{\partial z}=-P'(z,\W^{-1}).$$ 
Since at $p$ one has $\W^{-1}=0$, $P(z_i)=0$ and we have assumed that $P(z)$ has only simple roots, we get that $P^\prime(z_i)\neq 0$  which implies that the latter difference between the partial derivatives cannot vanish at $p$, a contradiction.

\medskip
To prove (v), we first consider the essential branch at $\infty$, i.e., the branch whose slope is given by $s=\frac{d}{\al}-1$. By our assumption, $0<\al<d$, which, in particular, implies that  this slope differs from $ -1$ which is the slope for all other branches. By the implicit function theorem, the essential branch is smooth at $(0,\infty)$.  

\smallskip
Let us now consider the remaining cases for which 
\begin{equation}
\label{eq:infbranch}
\W=ys(y)=y(-1+s_1 y+s_2 y^2+\dots)=-y+y^2u(y).
\end{equation} We will first show that if $P^\prime$ has simple roots, then there exist $d-1$ distinct solutions for the variable $s_1$. Rewriting \eqref{eq:infbranch} in terms of $s(y)$ corresponds to the blow-up of the curve at the origin, and then rewriting it in terms of $u=u(y)$ corresponds to  still another blow-up.

\smallskip
Note that $$y^{-1}+\W^{-1}=\frac{u}{yu-1,}$$
and  substitution of  (\ref{eq:infbranch}) in equation (\ref{eq:extra}) results in
\begin{equation}
\label{eq:logder}
\alpha y=\sum_{i=1}^{d}\frac{1}{u-z_i(yu-1)}.
\end{equation}
If we now set $y=0$, the latter equation becomes
\begin{equation}
\label{eq:v1}
0=\sum_{i=1}^{d}\frac{1}{u+z_i}=-\frac{P'(-u)}{P(-u)}.
\end{equation}
Further $$\frac{P'(u)}{P(u)}=0\iff P'(u)=0$$ which is an equation in $u$ of degree $d-1$  and its solutions are exactly the zeros of $P^\prime(u)$. Thus there exist $d-1$ solutions $u(0)=s_1$ of \eqref{eq:v1}. Moreover they are all distinct by the assumption that $P^\prime$ has only simple roots. Additionally, we can observe that equation (\ref{eq:logder}) defines a curve $\al y=F(u,y)$ in $\bC_u\times \bC_y$ with coordinates $u$ and $y$. This curve will be smooth and transversal to $y=0$ at a point $(s,0)$ if 
$$
\frac{\partial F(u,y)}{\partial u}\big \vert_{(s_i,0)}=\sum_{i=1}^{d}\frac{1}{(s+z_i)^2}=\left(\frac {P'(-s)}{P(s)}\right )'\neq 0.
$$
On the other hand, if we assume that $s=s_1$ is one of the distinct roots of $P'(-s)$ we obtain
 $$\left(\frac {P'(-s)}{P(s)}\right)'=-\frac{P''(-s)P(-s)-(P'(-s))^2}{P^2(-s)}=-\frac{P''(-s)P(-s)}{P^2(-s)}\neq 0.$$

\smallskip 
This argument shows that in a neighbourhood of the line $y=0$  in $\bC_y\times \bC_u$ with coordinates $(y,u)$, there exist $d-1$ branches of the affine curve
 (\ref{eq:logder})  intersecting this line at the $d-1$ different smooth points $(s_i,0)$, where each $s_i$ is a root of  $P'(-s)$. If we now consider these
branches in the space $\bC_y\times \bC_\W$ with coordinates $(y,\W)$ using the coordinate change $\iota: (y,u)\mapsto (y,\W)=(y,-y+y^2u)$, then an easy calculation shows that they  will become  $d-1$ distinct branches each  having the slope $-1$. This argument proves that these branches are smooth at $y=0$. Note that $\iota$ is the composition of two blow-ups: $(y,u)\mapsto (y,-1+yu)=(y,\tilde u)$  which blows up the point $(0,-1)$ and $(y,\tilde u)\mapsto(y,\tilde u y)$ which blows up the origin $(0,0)$. We have deduced the desired results from the strict transform given by \eqref{eq:logder}.

\smallskip
Summarizing we get the following. At the complete ramification point besides the smooth essential branch, there are $d-1$ additional smooth branches  with the same slope of $-1$ and distinct coefficients of $y^2$.

%\textcolor{red} {Inte bra att anv\"anda $s$ som variabel f\"or $P^\prime$!!!}

\medskip
To prove (vi), observe that for $z\neq \infty$, the poles of the $1$-form $\widetilde \Omega$ (obtained as the pullback to $\widetilde \La$ under the normalisation map $\mathfrak{n}$ of the form $\W\minispace\diff z$ restricted to $\widehat \La$)  occur at the pullbacks of the non-singular points of $\widehat \La \cap H_\W^\infty$, where $H_\W^\infty\subset \bC P^1_z\times \bC P^1_\W$ is a projective line over the point $\W=\infty$.  Since $\W=\infty$ corresponds to $\W^{-1}=0$ and $\al\neq0$, then for $z \neq \infty$, we immediately observe from \eqref{eq:extra} that the poles of $\W\,\diff z$ restricted to $\widehat \La$ occur at the points  $(z_i,\infty)$,  where $z_i$ is a root of $P$.  

\smallskip
Using \eqref{eq:extra}, we can  calculate the residues of  $\W(z)\minispace\diff z$ restricted to $\widehat \La$ at each point $(z_i,\infty)$. Dividing   equation~\eqref{eq:extra}  
%$$\al \W(z)=\frac{R^\prime(z+\W^{-1}(z))}{R(z+\W^{-1}(z))}$$ 
by $\W(z)$ and introducing  the local coordinate $\xi_i=z-z_i$, we get 
$$\al=\frac{\W^{-1}(\xi_i)}{\W^{-1}(\xi_i)+\xi_i}+\sum_{j\neq i} \frac{\W^{-1}(\xi_i)}{\W^{-1}(\xi_i)+\xi_i-(z_j-z_i)}.$$ 

By expanding $\W^{-1}(\xi_i)$ as $\delta_i\xi_i+\dots$ and letting $\xi_i\to 0$ in the right-hand side of the above equation, we obtain
$$\al=\frac{\delta_i\xi_i}{\delta_i\xi_i+\xi_i}=\frac{\delta_i}{\delta_i+1}$$
which immediately implies that $\delta_i=\frac{\al}{1-\al}$. %\textcolor{red}{(*** Dubbelkolla att jag inte förstörde något i argumentet i den föregående meningen, där jag tog bort ett (onödigt?) antagande. ***)} 
Thus 
$$\W(\xi_i)=\frac{1-\al}{\al \xi_i} +\dots\; \Rightarrow \quad \text {Res}|_{(\infty, z_i)}\; \W\,\diff z=\frac{1-\al}{\al}.$$

%Analogously, to settle (vii) we carry out  the same procedure for $\xi_j=z-\zeta_j$, and obtain
%$$\al=-\frac{\kappa_j}{\kappa_j+1}$$
%which implies $\kappa_j=-\frac{\al}{1+\al}.$ Thus 
%$$\W(\xi_j)=-\frac{1+\al}{\al \xi_j} +\dots\; \Rightarrow \quad \text {Res}|_{(\infty, \zeta_j)}\; \W\minispace\diff z=-\frac{1+\al}{\al}.$$
%Therefore if all roots of $P$ and $Q$ are simple, all the residues of $\W\minispace\diff z$ for $z\neq \infty$ are real.  %Finally, since $\C=\frac{\al \W}{d-\al}$ and both $\al$ and $d$ are real,  the same holds for $\C\minispace\diff z$.

\medskip
%When several poles collide the residue of the resulting pole is the sum of %residues of colliding poles and we get the general formula 
%$$\kappa_i=\frac{\al \cdot\sharp_i}{1-\al}$$
%where $\sharp_i$ is the order of  the pole at $z_i$ or $\zeta_i$. %WHAT ABOUT WHEN POLES FROM DIFFERENT SIDES COLLIDE? CHECK! 

 To settle (vii) and  to study the behavior of $\widetilde \Omega=\W\minispace\diff z$ at the singular point $(0,\infty)$, we need to use the change of variable $z=\frac{1}{y}$. Then  $\widetilde \Omega=-\W \frac{\diff y}{y^2}$.  Observe that  under the assumptions of (v), each local branch of $\widehat \La$ at $(0,\infty)$ is smooth which implies that  the normalisation map is a local diffeomorphism of the corresponding small neighborhood of $\widetilde \La$ with this branch. Thus we have the following expansion of $\W(y)$ for each local smooth branch with slope $-1$ and the residue of $\W\,\diff z$ restricted to this branch: 
$$\W(y)=-y+\dots, \quad \Rightarrow -\W(y)\,\frac{\diff y}{y^2}=\frac{(1+\dots)\,\diff y}{y} \Rightarrow \text{Res\,}{\vert_{(\infty,0)}}\left(-\W(y)\,\frac{\diff y}{y^2}\right)=1.$$

Analogously, for the essential branch whose slope equals $\frac{d}{\al}-1$, we get 
$$\W(y)=\left(\frac{d}{\al}-1\right)y+\dots \Rightarrow \text{Res}{\vert_{(\infty,0)}}\,\left(-\W(y)\,\frac{\diff y}{y^2}\right) =\frac{\al-d}{\al}.$$

\smallskip
Finally, we can conclude that the  curve  $ \La \subset \bC_z\times \bC_\W$   given by  \eqref{eq:algebraicDiffEq2} is an aBc since it is rational, irreducible and the form  $\widetilde \Omega$ on $\widetilde \La$ has only simple poles  with real residues. 
\end{proof}

%(viii) The zeros of $\widetilde \Omega$  (which is the pullback to $\widetilde \La$ of the form $\W\minispace\diff z$ restricted to $\widehat \La$ under the normalisation map $n$) occur either at the preimages of the points of $\widehat \La$ where $\W=0$ or at the critical points of the projection $\pi_z: \widehat \La\to \bC P^1_z$ since at such points the pullback of the element $\diff z$ vanishes. As we already mentioned, there are no points on $\widehat \La$ with $\W=0$ unless $z=\infty$. (The latter case was carefully studied above.) So the only zeros  of $\W\minispace\diff z$ restricted to $\widehat \La$ occur at the above critical points. Since $\widehat \La$ is a rational curve of bidegree $(d,d)$, it can be presented as the image of $\bC P^1$ under the map given by a pair of degree $d$ univariate rational fucntions. The critical points of the projection $\pi_z$ correspond to the critical points of the rational function defining the $z$-coordinate. It is easy to check that there are typically $2d-2$ such points. 

\begin{remark} {\rm Under the  above assumptions of (vi) and (vii), the total number of poles of $\widetilde \Omega$ on $\widetilde \La$ equals $2d=2\deg P$, all of them  having real residues. Observe that if all zeros of $P$ are simple, the singular point $(0,\infty)$ on $\widehat \La$ reduces its genus  by $(d-1)^2$ which means that this point is a rather complicated singularity. Under the above assumptions, the number of critical points (values) of $\tau$ equals $2d-2$ which checks with the Riemann-Hurwitz formula saying that the Euler characteristic of $\bC P^1$ coincides with the number of poles of $\widetilde \Omega$ minus the number of  its zeros: If  poles and zeros  of $\widetilde \Omega$ are counted with multiplicities we get the correct value $2$ for the Euler characteristics of $\bC P^1$}. %  Although many of the above items have some genericity assumptions, Proposition~{pr:rational} holds for an arbitrary rational function $R=P/Q$ by continuity.  }
\end{remark}

\begin{remark} 
{\rm The sum of all residues of any meromorphic form on any compact Riemann surface must vanish. Our count gives the following sum:
$$\Sigma=\frac{1-\al}{\al} \cdot d +d -\frac{d}{\al} =0.$$
}
\end{remark}

%\begin{remark} 
%{\rm Although we used some non-degeneracy assumptions in Theorem~\ref{pr:rational}, by continuity, most of its statements hold   for an arbitrary $R=P/Q$.  
%}
%\end{remark}

\begin{remark}  {\rm Observe that the bivariate polynomial in the left-hand side of \eqref{eq:algebraicDiffEq} defining the curve $\Ga_{\al,P}$ belongs to the class of balanced algebraic functions  introduced in  \S~3 of \cite{BoSh}. For any balanced algebraic function,  it has been  conjectured in  loc. cit. that there always exists a probability measure whose Cauchy transform satisfies the respective equation a.e. in $\bC$. However not all balanced algebraic functions correspond to affine Boutroux  curves.   %Could it be that any balanced algebraic function determines a perfect algebraic curve????
}\end{remark} 

\medskip 
%{\bf Defining the saddle point curve} \label{section:algebraiccurve} 
Next we introduce yet another algebraic curve  which will naturally reappear  later in connection with the application of the saddle point method.

\begin{definition} Given a polynomial $P$ of degree $d$ and $0<\al<d$, we define its affine \emph{saddle point curve}  
$
\D:=\D_{\al,P}\subset \bC_z\times \bC_u
$ 
as the curve given by the equation: 
\begin{equation}
\label{eq:saddle-pointsu}
\frac{P'(u)}{\alpha P(u)}-\frac{1}{u-z}=0.
        \end{equation}
%Note that \eqref{eq:s
 Following our notational conventions, we denote by $\widehat \D$ the closure of $\D$ in $\bC P_z^1\times \bC P_u^1$.
\end{definition}

  It turns out that $\D:=\D_{\al,P}$  is closely related to the symbol curve $\Ga:=\Ga_{\al,P}$. Namely, consider the birational transformation $\chi: \bC P_z^1\times \bC P_u^1 \to \bC P_z^1\times\bC P_\C^1$  sending $(z,u)$ to $(z,\C)$ where $z\mapsto z$ and 
\begin{equation}
\label{eq:transf}
\C=\frac{\alpha}{d-\alpha}\cdot\frac{1}{u-z}\iff u=z+\frac{\alpha}{d-\alpha}\cdot \C^{-1}.
\end{equation}
Under this change of variables equation~\eqref{eq:saddle-pointsu} transforms into equation~\eqref{eq:symbolcurve1} which is equivalent to \eqref{eq:algebraicDiffEq}. Thus the restriction $\chi :\widehat\D\to \widehat\Ga$ provides a birational isomorphism. (Observe that $\chi$ sends the complement of the line $z=u$  isomorphically to the complement of the line $z=\infty$.)
%(associated to the Cauchy transform)     
Therefore we can apply the results of Theorem~\ref{pr:rational} to analyze the saddle point curve $\D$ and its closure $\widehat \D$.  Observe however that  it follows from \eqref{eq:saddle-pointsu}  that $\widehat \D\subset \bC P_z^1\times \bC P_u^1$ has bidegree $(1,d)$ while $\widehat\Ga\subset  \bC P_z^1\times \bC P_\C^1$ has bidegree $(d,d)$. Below we collect a number of properties of $\D$ and its closure $\widehat \D$ which we will need later.

\begin{corollary}
\label{prop:saddle-pointcurveprop} Assume that $P$ is strongly generic, i.e., that $P$ and $P^\prime$ have simple zeros. Then  the following statements hold:
 \begin{itemize}
\item[(i)] The affine curve $\D \subset \bC_z\times \bC_u$ and its closure $\widehat \D\subset \bC P_z^1\times \bC P_u^1$ are  irreducible, rational and smooth.  The birational equivalence $\chi$ gives the normalization map $\widehat \D\to \widehat \Ga$.  In other words,  $\widehat \D$ coincides with the normalization $\widetilde \Ga$ of $\widehat \Ga$ and $\chi$ is the normalization map. 

\item[(ii)]$\widehat\D\subset \bC P^1_z\times \bC P^1_u$ has $d$ branches over a neighborhood of the point $z=\infty$ in $\bC P^1_z$. One branch passes through  $(\infty,\infty)$. 
% with slope $d/(d-\al)$. 
The remaining $d-1$ branches pass through  $d-1$ points of the form $(\infty, q_i)$, where $q_1,\dots ,q_{d-1}$ are  the $d-1$ (distinct) critical points of $P$, i.e., the roots of $P'(u)=0$; these critical points are distinct since $P$ is strongly generic.

\item[(iii)] There are no points on the line  $u=\infty$ in $\bC P^1_z\times \bC P^1_u$ belonging to  $\widehat\D$ except for $(\infty,\infty)$.
The intersection of $\widehat\D$ with the diagonal line $u=z$ consists of  $(\infty,\infty)$    together with  the points $(z_j,z_j),\ j=1,\dots, d$, where $z_1,\dots,z_{d}$ are the roots of $P$. % \textcolor{red} %{KOmmer vi inte att f\aa{} aldeledes f\"or m\aa nga punkter???}\textcolor{green}{NEJ, vi \"ar inte i $P^2$, bigraden av det homogena polynomet som definierar D är (1,d).}

\item[(iv)]  The  pullback  to $\widehat\D$ of the meromorphic form $\C\,\diff z$ defined on $\widehat\Gamma$ under the normalization map $\chi$ is given by 
\begin{equation}
\label{eq:difform}
\chi^{*}(\C\,\diff z):=\frac{\alpha}{d-\al}\cdot \frac{\diff z}{u-z}. 
\end{equation} 
%\textcolor{red}{FEL att anvÃ€nda $\C dz$ Vilken form beh\"os???}

The poles of this pullback are all simple and located at the points:
\\
{\rm (a)}  $(z,u)=(z_j,z_j),\ j=1,\dots, d$ with the residue equal to $\frac{1-\al}{d-\al}$;
\\
{\rm (b)}  $(z,u)=(\infty, q_i)$, $i=1,\dots, d-1$ with the residue equal to $\frac{\al}{d-\al}$;
\\
 {\rm (c)}  $(\infty,\infty)$ with the residue equal to $-1$.
\end{itemize}
\end{corollary}

\section{Proofs of the main theorems}\label{sec:MainThs}

Our main tool in this section will be the classical saddle point method 
as presented in e.g., \cite{OS}, see also \cite{Bl}, \S~7.3.11, and \cite{Br}.
  Let as above $P$ be a monic polynomial of degree $d\ge 2$ and $\alpha \in (0,d)$. Slightly abusing our previous notation,  let  $\mu_n:=\mu_{[\al n]-1, n, P}$ be the  root-counting  measure of the Rodrigues' descendant
$$
\q_n(z):=\R_{[\alpha n]-1, n, P}(z)=(P^n)^{([\alpha n]-1)}(z).
$$ 
\smallskip
(Note that the order of derivative here is one less than in \S~1, but this will not effect the asymptotic result.)

\medskip
As already mentioned in the introduction, the proof of Theorem~\ref{th:Cauchy} is as follows.
For any $z\in \bC$, Cauchy's formula for higher order derivatives gives 
\begin{equation}
\label{eq:cauchy}
\q_n(z)=\frac{([\alpha n]-1)!}{2\pi i}\int_{c} \frac{P^n(u)\,\diff u}{(u-z)^{[\alpha n]}},
\end{equation}
%\textcolor{red}{(*** $q_i$ har anvÃ€nts tidigare bÃ¥de i tropical trace-avsnittet och som nollstÃ€llen till $P'$ AnvÃ€nd $\mathfrak{q_n}$? ***)}
where $c$ is any simple closed curve in $\bC$ encircling $z$   once in the counterclockwise direction. (Here we use the fact that $P$ has no poles.)

\smallskip
The saddle point method allows us to analyze the asymptotic of \eqref{eq:cauchy} when $n\to \infty$. The degree of the polynomial $\q_n(z)$ equals $d_n:=dn-[\alpha n]+1$. Below we will calculate the limit of the sequence  $\{L_{\mu_n}(z)\}$ of logarithmic potentials of $\mu_n$, where $L_{\mu_n}(z):=\frac{1}{d_n}\log\vert \q_n(z)/a_n\vert$ and $a_n$ is the leading coefficient of $\q_n(z)$. 

\smallskip
We will show that the critical points of the integrand in \eqref{eq:cauchy} belong to the above saddle point curve  $\mathcal D:=\D_{\al,P}$ given by \eqref{eq:saddle-pointsu}, which is birationally equivalent to the symbol curve $\Ga:=\Ga_{\al,P}$ given by  \eqref{eq:algebraicDiffEq}. Furthermore we will see that the critical points which will play an important role in our asymptotic calculation form an open subset $U \subset \mathcal D$.  These facts enable us to identify   
the limit $L_\mu(z):=\lim_{n\to\infty}L_{\mu_n}(z)$ with the tropical trace of a natural harmonic function defined on $ U$. Finally,  applying the  Laplace operator to $L_\mu$,  one obtains as an immediate consequence that the limiting asymptotic measure $\mu:=\lim_{n\to\infty}\mu_n$
exists and that its Cauchy transform satisfies the algebraic equation \eqref{eq:cauchy}. Let us now provide the relevant details dividing them into several  subsections.

\subsection{Root asymptotic via the saddle point method} Given $\alpha> 0$, define 
\begin{equation}
\label{eq:fractionalpart} s_n:=n-\frac{[\alpha n]}{\alpha},
\end{equation}
where $0\leq s_n<1/\alpha$ and set $m:=[\alpha n]$. Consider the integral  
$$I(z)=\int_\gamma \frac{P^n(u)\,\diff u}{(u-z)^{[\al n]}}$$ 
over a curve segment $\gamma$ that neither contains $z$  nor the zeros of $P$.  (For the moment we are suppressing the dependence of the integral on $P, n, \al, \ga$.)  On a sufficiently small neighboorhood $\mathcal O$ of any point in $\gamma$ there exists a single-valued branch of the logarithm $\log P$ which is well-defined in $\mathcal O$. Using this branch of the logarithm we can define real and complex powers of $P$  in $\mathcal O$ and ensure that they satisfy the relation $(P^{1/\alpha})^mP^{s_n}=P^n$. We may then analytically continue our choice of branch along $\gamma$. Then
\begin{equation}
\label{eq:cauchy2} 
I_{P, m,s_n,\gamma}(z):=\int_{\gamma} \bigg(\frac{P^{1/\alpha}(u)}{u-z}\bigg)^{m}P^{s_n}(u)\,\diff u =
\int_{\gamma} e^{k(z,u)m}(P(u))^{s_n}\,\diff u,
\end{equation} 
where \begin{equation}
\label{eq:defk}
k(z,u):=\frac{ 1}{\alpha}\log P(u)-\log(u-z).
\end{equation}
 Clearly, for fixed $z$, $k(z,u)$ is holomorphic w.r.t the second variable $u\in \mathcal O$ if $u$ avoids both $z$ and the zeros of $P$. 
 
 \begin{definition} For fixed $z$,
a \emph{saddle point} of $k(z,u)$ is  a zero of $\frac{\partial k}{\partial u}(z,u).$
 \end{definition}
  
 \smallskip
 The exact version of the saddle point method  which we will apply to the function 
 $$
 h(u)=k(z,u) \text { for  fixed } z 
 $$
  is formulated in Lemma \ref{lm11} below, compare Theorem 1.2 and Corollary 1.4 of \cite{OS}. Namely, assume that
 
 \smallskip 
 \noindent
  (i) $h(u)$ is any function holomorphic in a neighbourhood $\mathcal O$ of a simple curve $\gamma$;  
 
  \noindent 
  (ii)  $u^\ast\in \gamma$ is a saddle point of $h(u)$ and it is an inner point of $\ga$; %\textcolor{red}{(*** $\gamma$ ej nÃ¶dvÃ€ndigtvis sluten, och $u_0$ ligger bÃ¥de pÃ¥ $\gamma$ och inuti $\gamma$? ***)}

 \noindent  
  (iii) $\forall u\in \gamma\text{ such that } u\neq u^\ast$, $\text{Re } h(u)< \text{Re } h(u^\ast)$. 

\medskip  
  Finally, let $\ell\geq 2$ be the order of the saddle point $u^\ast$, i.e., 
 \begin{equation}
\label{eq:ordersaddle-point}
h(u)=h(u^\ast)-h_0(u-u^\ast)^\ell(1-\phi(u)),
 \end{equation}
where $h_0\neq 0$   and $\phi(u)$ is a  function which vanishes at $u^\ast$ and is holomorphic in a small neighborhood of $u^\ast$.  
 
 \begin{LEMM} \label{lm11} Using the above notation, for $m\in \bN$ and $0\leq s\leq A<\infty$, consider 
 $$I_{m,s,\gamma}:=\int_{\gamma} e^{h(u)m}P^{s}(u)\,\diff u.$$ 
Then,  under the above assumptions {\rm(i)--(iii)}, 
\begin{equation}
\label{eq:saddleas} 
I_{m,s,\gamma}= e^{h(u^\ast)m}\left(\Gamma(\ell^{-1})\frac{\beta_0(\epsilon_1-\epsilon_2)}{m^{\frac{1}{\ell}}}+O\left (\frac{K(P)}{m^{\frac{2}{\ell}}}\right)\right), 
\end{equation}
where $\epsilon_1$ and $\epsilon_2$ are two distinct $\ell$-th roots of unity  depending  only on $\gamma$ and $K(P)$ is an upper bound of $\vert P^{s}(u)\vert$ in $\mathcal O$. Here $\Ga$ stands for the gamma function (not to be confused with the curve $\Ga$ introduced in \S~\ref{sec:perfect}). The constant $\beta_0$ is given by 
$$
\beta_0=\frac{1}{\ell}\cdot {h_0^{-1/\ell}(P^s(u^\ast))}
$$
and the implicit constant in the  remainder term $O(\dots)$ of \eqref{eq:saddleas}  is independent of $P$, $s$, and $m$. 
\end{LEMM}

 \smallskip 
 We are going to apply Lemma~\ref{lm11} to the integral \eqref{eq:cauchy2} when setting $ h(u)=k(z,u)$ with  fixed $z$ and using a contour on which $k(z,u)$ is possibly multi-valued. Therefore condition (i) above is not necessarily valid.  
 However, it is enough to note that
  $$\text{Re }k(z,u)=\frac{ 1}{\alpha}\log \vert P(u)\vert -\log\vert (u-z)\vert $$ is defined independently of the above choice of a  branch of the logarithm.
% it is here one should start if one wants to understand monodromy action.....:) 
 
 \begin{CORR}  
\label{cor:uniform} For fixed $z$, assume that $u^\ast$ and the path $\gamma$ satisfy the  assumptions {\rm (ii)--(iii)} of Lemma~\ref{lm11}. Then, in  notation of (\ref{eq:cauchy2}), 
for any sequence   $\{s_n\}$, such that $\lim_{n\to \infty} s_n/n= 0 $,  one has
\begin{equation}
\label{eq:uniform2}
\lim_{n\to \infty} \vert I_{P, m,s_n,\gamma}(z)\vert ^{1/m} =  e^{\text{Re } k(z,u^\ast)}. 
\end{equation}  
%This convergence is uniform in $s$.                   
 \end{CORR}
\begin{proof} Let $\beta_0$ be the parameter used in the estimate of $I_{P, m,s_n,\gamma}(z)$ in \eqref{eq:saddleas}. The condition on $s_n$ implies that $\vert P(u^\ast)\vert^{s_n/m}\to 1$ as $n\to\infty$ (or, equivalently, as $m\to\infty$). Hence
$$
\lim_{m\to\infty } \left(\Gamma(\ell^{-1})\frac{\beta_0(\epsilon_1-\epsilon_2)}{m^{\frac{1}{\ell}}}\right)^{1/m}=1.
$$
A similar estimate of $\vert K(P)\vert^{^{1/m}}$ together with 
 Corollary  \ref{lm11} then clearly imply 
 \eqref{eq:uniform2},
for a $\gamma$ that satisfies (i)-(iii). If (i) is not satisfied, split $\gamma=\gamma_1+\gamma_2$ into two disjoint contours such that the saddle point $u^\ast$ is contained in $\gamma_1$ and to which Lemma~\ref{lm11} applies.  On the second contour $\gamma_2$, the integral $\vert I_{P,m,s_n,\gamma_2}\vert$ will be of  order $o(e^{m\text{Re } k(z, u^\ast)})$ when $m\to \infty$;  hence it will not contribute to the value of the limit.
\end{proof}

\label{subsection:cauchyasymptotics}

\subsection{Deformation of the contour}
\label{sec:contour} The next step in the proof of Theorem~\ref{th:Cauchy} is to find an appropriate integration contour to which Corollary \ref{cor:uniform} can be applied. Note that the only a priori condition imposed on the simple contour $c$ in the integral \eqref{eq:cauchy} is that it encircles the fixed complex number  $z\in \bC$ once counterclockwise. 
 
 \smallskip
 For all pairs $(z,u)\in \bC^2$ except for those for which either $P(u)=0$ or $u=z$,  define
\begin{equation}
\label{eq:guz}
G(z,u):=\frac{1}{\al} \log\vert P(u)\vert-\log \vert u-z\vert. 
\end{equation}

Observe that 
\begin{equation}
\label{eq:guzku}G(z,u)=\text{Re } k(z,u)
\end{equation}%{eq:saddle-pointsu}
and that $G(z,u)$ is a harmonic function of the variable $u$, except at a finite number of logarithmic singularities.

For fixed $z$, the saddle points $(z,u)$ of $k(z,u)$ 
are given by the values of $u$ for which the relation
\begin{equation}
\label{eq:saddle-points1}
2\frac{\partial G}{\partial u}(z,u)=\frac{\partial k}{\partial u}(z,u)=\frac{P'(u)}{\alpha P(u)}-\frac{1}{u-z}=0
        \end{equation}
        holds. 
Observe that for any fixed $z$, there are at most $d=\deg P$ such saddle points since for any $0<\alpha<d$, the polynomial $P'(u)(u-z)-\alpha P(u)$  has degree $d$ in $u$.  
 In particular, the projective closure of the set of these saddle points coincides with the algebraic curve 
 %NOTATION
$$
\widehat\D\subset \bC P_z^1\times \bC P_u^1
$$ 
determined by equation \eqref{eq:saddle-points1} in the affine $(z,u)$-plane. Under the projection $\pi : (z,u)\to z$, the smooth curve $\widehat \D$ becomes a branched covering of $\bC P_z^1$ with a finite branching locus $B$. (Its properties have been described in detail in the above Corollary \ref{prop:saddle-pointcurveprop}).
Over any simply connected domain $\bD\subset \bC_z\setminus B$, the curve $\widehat \D$ splits into $d$ distinct branches which we denote by $u_i^\D(z), \ i=1,\dots,d$. 

\smallskip 
Next let us restrict $z$ in such a way that we will see a  clear interaction between the function $G(z,u)$ and the saddle point curve $\D$. 
\begin{lemma}
\label{lemma:distinctharmonic}The set $\mathcal O=\{z\in \bC_z\setminus B: \forall i \neq  j, \; G(z,u_i^\D(z))\neq G(z,u_j^\D(z))\}$ is an open and dense subset of $\bC_z\setminus B$.  Locally the complement $\Delta:=\bC_z \setminus \mathcal O $  is  a finite union of segments of real analytic curves (and possibly isolated points). % \textcolor{red}{(*** ``for all'' krockar med notationen i antagandet (iii) innan Lemma A ***)}
 
\end{lemma}
\begin{proof}
Let $\bD\subset \bC_z\setminus B$  be a simply connected domain. The functions
$G(z,u_i^\D(z)), z\in \bD, \ i=1,\dots, d$, are harmonic and have at most a finite number of poles. As a consequence, for $i\neq j$, the equation $G(z,u_i^\D(z))=G(z,u_j^\D(z))$ is either satisfied identically for all $z\in \bD$ or  it holds only on a set whose complement $U_{ij}$ is open and dense in $\bD$. If we can exclude the former case for any simply connected $\bD$, then we can conclude that $\mathcal O$ being the union of the intersections 
$\cap_{i<j} U_{ij}$ taken over all possible simply connected $\bD$ 
is open and dense.

Indeed, suppose  that $G(z,u_1^\D(z))\equiv G(z,u_2^\D(z))$ for two distinct branches and all $z\in \bD$. 
Using the irreducibility of $\widehat \D$ (see Corollary \ref{prop:saddle-pointcurveprop} (i)), all other branches can  be obtained by the analytic continuation of the branch representing $u_1(z)$. In particular, $u_1^\D(z)$ can be analytically continued via a sequence of disks tending to $\infty$ to the unique branch $u^\D(z)$ for which $u^\D(z)\to \infty$ as $z\to\infty$; see Corollary~\ref{prop:saddle-pointcurveprop} (ii).

 \smallskip
  The corresponding analytic continuation of $u_2^\D(z)$ along the same sequence of disks must become a  branch $\tilde u^\D(z)$ near infinity which is different from $u^\D(z)$. Therefore by Corollary \ref{prop:saddle-pointcurveprop} (ii), we have that  as $z\to\infty$ then $\tilde u\to q_i\in \bC$  for some critical point  $q_i$, i.e., a root of $P'(z)=0$. On the other hand, by our assumptions, we have that 
  \begin{equation} 
  \label{eq:lemma ii one}
  G(z,u^\D(z))\equiv G(z,\tilde u^\D(z)) 
  \end{equation}
  in some neighborhood of $\infty$. 
But this cannot be the case. Namely, for the second branch  $\tilde u^\D(z)$, we have $G(z,\tilde u^\D(z))\sim -\log\vert z\vert$ since $ \tilde u^\D$ has a finite limit $q_i$ as $z\to \infty$. On the other hand, for the first branch   $u^\D(z)$, we have that $G(z,u^\D(z))\sim \frac{d-\al}{\al}\log\vert z\vert$. 

This can be checked by a calculation using Corollary~\ref{prop:saddle-pointcurveprop} which gives that $u^\D(z)\sim \frac{d}{d-\al} z$  and hence  $\log\vert u^\D(z)-z\vert\sim\log \vert u^\D(z)\vert \sim\log\vert z\vert$ as $z\to \infty$. 
Hence, \eqref{eq:lemma ii one}  fails in a neighborhood of infinity implying that $\mathcal O$ is an open and dense subset of $\bC_z\setminus B$.
\end{proof}
Next we prove that under the assumption of strong genericity of $P$ and for fixed $z$, $G(z,u)$ is a simple Morse function of the variable $u$.

 \begin{lemma} For any  strongly generic $P$, all saddle points of $G(z,u)$ are simple, i.e., have order 2.
 \end{lemma}
 \begin{proof} For a fixed $z$, a saddle point  $(z,u^\ast)$  is simple if and only if $\ell=2$ in formula \eqref{eq:ordersaddle-point} which is equivalent to $\frac{\partial^2k} {\partial ^2u}(z,u^\ast)\neq 0$. But $\frac{\partial k} {\partial u}(z,u)=\frac{P'(u)}{\alpha P(u)}-\frac{1}{u-z}=0$ which implies that   
 $$\frac{\partial^2k} {\partial ^2u}(z,u)=\frac{P''(u)P(u)-(P^\prime(u))^2}{\alpha (P(u))^2}+\frac{1}{(u-z)^2}.$$
 Assuming that $\frac{\partial k} {\partial u}(z,u^\ast)= \frac{\partial^2k} {\partial ^2u}(z,u^\ast)  =0$, we get 
 \begin{equation}
 \label{eq:simplesaddle}
 \alpha P''(u^\ast)P(u^\ast)+(1-\alpha)(P'(u^\ast))^2=0.
 \end{equation} 
 Since $\deg P=d$, by looking at the  leading term in  the variable $u^\ast$ in \eqref{eq:simplesaddle}, we derive that
 $$
 \alpha d(d-1)+(1-\alpha) d^2=0\iff \alpha=d.
 $$
But, since $0<\alpha<d$, we obtain that $\frac{\partial^2k}{\partial^2u}(z,u^\ast)\neq 0$, which implies that  $(z,u^\ast)$ is a simple saddle point. 
 \end{proof}
 
 Now observe that for fixed $z$, the level curve $G(z,u)=G(z,u^\ast)$ passing through a simple saddle point $(z,u^\ast)$ has  two local curve segments (branches) near $u^\ast$.  The analytic continuations of these branches must end  at some  saddle point, since $\lim_{\vert u \vert \to\infty} \vert G(z,u)\vert =\infty$. If, additionally, $z\in \mathcal O$, 
 then the analytic continuations of both  branches have to come back to the same saddle point. Again, since $z\in \mathcal O$,  these curves  will be  non-intersecting, and hence they form two closed ovals $C_i,\ i=1,2$, disjoint from each other everywhere except at the initial saddle point. There exist two possible topological configurations of such ovals in $\bC$. Namely, they either form a figure eight, see Fig.~\ref{fig:integrationcontour2} a), or one of the ovals contains the other, see Fig.~\ref{fig:integrationcontour2} b) and c).

 On the one side of each oval, the function $G(z,u)$ will increase, and on the other side it will decrease (which is marked by the $\pm$-signs in Fig.~\ref{fig:integrationcontour2}). Furthermore, by the maximum principle, each connected component of the complement of the level curve must contain a pole. 
 Now notice that the plane $\pi^{-1}(z)=\{ (z,u):u\in \mathbb CP^1\}$ contains two poles of $G(z,u) $ with positive residues, namely, $(z,z)$
 and $(z,\infty)$ and $d $ poles with negative residues, namely, $(z,z_j),\; j=1,\dots,d,$ where $P(z_j)=0$. Hence, there exist only three topological possibilies to place the pole $P^+:=(z,z)$ relative to the level curve under consideration which are 
 shown in Fig.~\ref{fig:integrationcontour2} a) -- c). 
% \begin{figure}[htp]
% \label{fig:integrationcontour2}
% \begin{center}
% \includegraphics[width=0.9\textwidth]{relevantsaddlepoint5.pdf}
% \end{center}
% \caption{The integration contour}
% \end{figure}

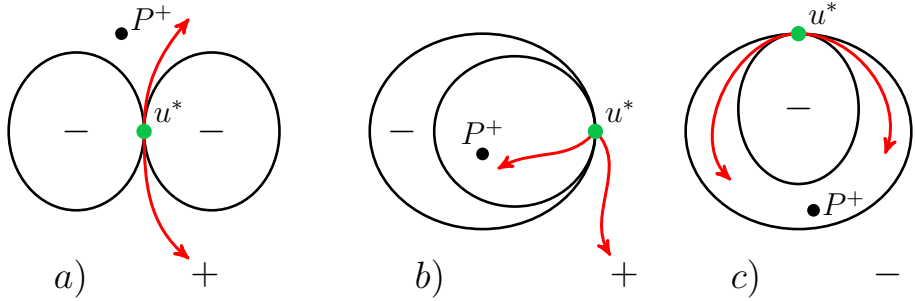
\begin{figure}[H]
\begin{center}
\usetikzlibrary {calc,intersections,through,arrows}
\begin{tikzpicture}[->,>=stealth',auto,node distance=3cm,
  thick,main node/.style={circle,draw,font=\sffamily\Large\bfseries}]
\definecolor{wColor}{RGB}{11,195,71}

\coordinate[label={[label distance=0.0cm,,font=\LARGE]0:$a)$}] (A) at (0.47,-1.956); % 1.8
\coordinate[label={[label distance=0.0cm,,font=\LARGE]0:$b)$}] (B) at (5.27,-1.956); % 6.3
\coordinate[label={[label distance=0.0cm,,font=\LARGE]0:$c)$}] (C) at (9.47,-1.956); % 10.5
\coordinate (LS1) at (0,0);
\coordinate (RS1) at (3.6,0);
\coordinate[label={[label distance=0.0cm,,font=\Large]20:$u^\ast$}] (MID1) at ($(LS1)!0.5!(RS1)$); % w1
\coordinate[label={[shift={(0.0,-0.299)},font=\LARGE]$-$}] (MIDLEFT1) at ($(LS1)!0.5!(MID1)$);
\coordinate[label={[shift={(0.0,-0.299)},font=\LARGE]$-$}] (MIDRIGHT1) at ($(MID1)!0.5!(RS1)$);
\coordinate (LS2) at (4.8,0);
\coordinate[label={[label distance=0.0cm,,font=\Large]20:$u^\ast$}] (RS2) at (7.8,0); % w2
\coordinate (MID2) at ($(LS2)!0.5!(RS2)$);
\coordinate (LS3) at (9,0);
\coordinate (RS3) at (12,0);
\coordinate (MID3) at ($(LS3)!0.5!(RS3)$);
\coordinate (IEMID2) at (6.7285714,0); % Center of second inner ellipse
\coordinate[label={[shift={(0.0,-0.299)},font=\LARGE]$-$}] (IEMID3) at (10.5,0.3); % Center of third inner ellipse
\coordinate[label={[shift={(0.0,-0.299)},font=\LARGE]$-$}] (MINUS3one) at (5.22857145,0.0); % Center of third inner ellipse
\coordinate[label={[label distance=0.0cm,,font=\Large]35:$u^\ast$}] (W3) at (10.5,1.3); % w3
\coordinate[label={[shift={(0.4,-0.08)},font=\Large]$P^+$}] (P1) at (1.5,1.3); % First P^+
\coordinate[label={[shift={(-0.0,0.0)},font=\Large]$P^+$}] (P2) at (6.3,-0.3); % Second P^+
\coordinate[label={[shift={(0.4,-0.2)},font=\Large]$P^+$}] (P3) at (10.7,-1.05); % Third P^+
\coordinate (TARGET1) at (2.4,1.5);
\coordinate (TARGET2) at (6.5,-0.5);
\coordinate (TARGET3A) at (9.6,-0.65);
\coordinate (TARGET3B) at (11.65,-0.3);
\coordinate[label={[label distance=-0.15cm,font=\LARGE]320:$+$}] (PLUS1) at (2.4,-1.698);
\coordinate[label={[label distance=-0.15cm,font=\LARGE]340:$+$}] (PLUS2) at (8,-1.655);
\coordinate[label={[label distance=-0.15cm,font=\LARGE]340:$-$}] (MINUS3two) at (11.5,-1.652);

\tkzDrawPoint[size=4,color=black](P1)
\tkzDrawPoint[size=4,color=black](P2)
\tkzDrawPoint[size=4,color=black](P3)
\draw[line width=0.35mm] (MIDLEFT1) ellipse (.9cm and 1.05cm);
\draw[line width=0.35mm] (MIDRIGHT1) ellipse (.9cm and 1.05cm);
\draw[line width=0.35mm] (MID2) ellipse (1.5cm and 1.3cm);
\draw[line width=0.35mm] (MID3) ellipse (1.5cm and 1.3cm);
\draw[line width=0.35mm] (IEMID2) ellipse (1.0714285cm and 1cm);
\draw[line width=0.35mm] (IEMID3) ellipse (.8cm and 1cm);

% Arrows in the first figure
  \path[every node/.style={font=\sffamily\small},color=red,line width=0.4mm]
	% (MID1) edge[out=90, in=300] node [left] {} (P1); % Reaches P^+
	(MID1) edge[out=90, in=230] node [left] {} (TARGET1); % Doesn't reach P^+
  \path[every node/.style={font=\sffamily\small},color=red,line width=0.4mm]
	(MID1) edge[out=270, in=140] node [left] {} (PLUS1);
	
% Arrows in the second figure
  \path[every node/.style={font=\sffamily\small},color=red,line width=0.4mm]
	% (RS2) edge[out=200, in=330] node [left] {} (P2); % Reaches P^+
	(RS2) edge[out=230, in=30] node [left] {} (TARGET2); % Doesn't reach P^+
  \path[every node/.style={font=\sffamily\small},color=red,line width=0.4mm]
	(RS2) edge[out=315, in=120] node [left] {} (PLUS2);
	
% Arrows in the third figure
  \path[every node/.style={font=\sffamily\small},color=red,line width=0.4mm]
	% (W3) edge[out=0, in=0, looseness=1.5] node [left] {} (P3); % Reaches P^+
	(W3) edge[out=180, in=135] node [left] {} (TARGET3A); % Doesn't reach P^+
  \path[every node/.style={font=\sffamily\small},color=red,line width=0.4mm]
	% (W3) edge[out=180, in=180, looseness=1.7] node [left] {} (P3); % Reaches P^+
	(W3) edge[out=0, in=70] node [left] {} (TARGET3B); % Doesn't reach P^+

\tkzDrawPoint[size=5,color=wColor](RS2)
\tkzDrawPoint[size=5,color=wColor](MID1)
\tkzDrawPoint[size=5,color=wColor](W3)

\end{tikzpicture}
\caption{Three possible shapes of the level curve passing through a saddle point. The (red) curve segments with arrowheads represent the paths of steepest ascent of $G(z,u)$.} 
\label{fig:integrationcontour2}
\end{center}
\end{figure}

The situation that will be of a special interest to us is presented in Fig.~\ref{fig:integrationcontour2} b), and we then say that such saddle point is {\it maximally relevant.}

\begin{lemma}
\label{lemma:situation} For each $z\in \mathcal O$,
% WHAT IS $D$ HERE? 
there exists a unique saddle point $(z,u^\ast_{max}(z))$, such that
\begin{enumerate}
 \item[i)] the connected component of the level curve $G(z,u)=G(z,u^\ast_{max}(z))$ passing through $u_{max}^\ast(z)$ is the union $C_1\cup C_2$ where $C_1$ and $C_2$ are closed ovals such that the interior of $C_1$ contains the pole $P^+=(z,z)$ and $C_1$ is contained in the interior of $C_2$, see Fig.~\ref{fig:integrationcontour}.
  \item[ii)] For all saddle points $(z,u^\ast)$ satisfying condition {\rm i)}, $G(z,u^\ast_{max}(z))>G(z,u^\ast)$.
 \end{enumerate} 
\end{lemma}
\begin{proof} 

For fixed $z$ and $t\gg 0$, the level set $G(z,u)=t$ consists of two enclosed ovals $C_i(t),\; i=1,2,$ and the set $\Sigma^+(t):=G(z,u)\geq t$ has two connected components both of which are topologically cylinders. The boundary of one of these components is the union of $C_2(t)$ and $(z,\infty)$ while  the  other one has the union of $C_1(t)$ and $P^+$ as its boundary. But
$\Sigma^+(t)$  is connected for $t\ll 0$, and thus there exists the minimal value $t_0$ of the parameter $t$ such that $\Sigma^+(t_0+\epsilon) $ is connected for $\epsilon\leq 0$ and 
  disconnected for $\epsilon>0$. This change of  topology occurs when the ovals $C_i:=C_i(t_0), i=1,2,$ in the formulation of Lemma~\ref{lemma:situation} touch each other, 
  which can only happen  at a saddle point $(z,u^\ast(z))$ of the type shown in Fig.~\ref{fig:integrationcontour2} b). Furthermore, this critical point is the unique maximally relevant saddle point.  Indeed, for any saddle point $u^\ast$ of $G$ with the critical value  $t>t_0$, the set $\Sigma^+(t) $ is disconnected. Therefore it is impossible to connect the saddle point $u^\ast$  both to the positive pole $P^+$ and to $\infty$ by using paths along which the function $G$ is increasing.  On the other hand, it is clearly possible to find such paths for a saddle point  shown in Fig.~\ref{fig:integrationcontour2} b).  
   
   Similarly, a saddle point of $G$ with a critical value $t<t_0$ cannot be maximally relevant, since %(as the relevant picture clearly shows \textcolor{red}{(*** a maximally relevant picture for understanding? ***)}) 
   $P^+$ cannot  be contained in the interior of any oval in the level set $G(z,u)=t<t_0$. In fact, the level curve passing through such $u^\ast$  has to look as  in Fig.~\ref{fig:integrationcontour2} a)
which finishes the proof.
\end{proof}

\begin{figure}[H]
\begin{center}
\usetikzlibrary {calc,intersections,through,arrows}
\begin{tikzpicture}[->,>=stealth',auto,node distance=3cm,
  thick,main node/.style={circle,draw,font=\sffamily\Large\bfseries}]
\definecolor{wColor}{RGB}{11,195,71}
\definecolor{shadeColor}{RGB}{206,228,251}

\coordinate (LS1) at (0,0);
\tikzset{
    ellipses/.pic = {
        \draw (0,0) ellipse (2.4cm and 2cm);
		\fill[color=shadeColor] (0.6,0) ellipse (1.8cm and 1.5cm);
        \draw (0.6,0) ellipse (1.8cm and 1.5cm);
		\fill[color=white] (0.95,0) ellipse (1.45cm and 1.3cm);
		\draw (0.95,0) ellipse (1.45cm and 1.3cm);
        \draw[-,color=red,line width=0.45mm] (0.6,0) -- (7,0);
		\draw[color=red,line width=0.45mm] (2.4,0) -- (3.8,0);
		\draw[color=red,line width=0.45mm] (2.4,0) -- (1,0);
		\fill[color=wColor] (2.4,0) circle (0.12);
    }
}
\path[rotate=7,transform shape]  (0,0) pic{ellipses};
\coordinate[label={[font=\Large]$C_1$}] (C1) at (1.1,-1.12);
\coordinate[label={[font=\Large]$C_2$}] (C2) at (1.6,-2.2);
\coordinate[label={[font=\Large]$\tilde C_1$}] (TILDE) at (-0.93,1);
\coordinate[label={[font=\Large]$\infty$}] (TILDE) at (7.35,0.66);
\coordinate[label={[font=\huge]$+$}] (PLUS) at (0.98,0.6);
\coordinate[label={[font=\huge]$+$}] (PLUS) at (7.2,-0.25);
\coordinate[label={[font=\huge]$-$}] (MINUS) at (-1.795,-0.325);
\coordinate[label={[font=\Large]$N$}] (N) at (-0.825,-0.25);
\coordinate[label={[shift={(0.4,-0.08)},font=\Large]$P^+$}] (P) at (0.2,-0.5);
\coordinate[label={[label distance=0.0cm,,font=\Large]35:$u^\ast_{max}$}] (W1) at (2.34,0.37);
\tkzDrawPoint[size=4,color=black](P)

\end{tikzpicture}
\caption{The integration contour.}
\label{fig:integrationcontour}
\end{center}
\end{figure}
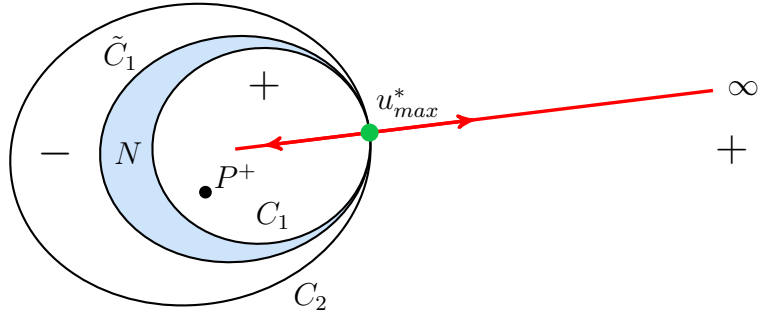

\begin{remark}
\label{remark:path} {\rm  Of the two paths of maximal ascent  starting at  a maximally relevant saddle point $u^\ast_{max}(z)$, one necessarily 
goes to $P^+$ and the other one to $(z,\infty)$, see Fig.~\ref{fig:integrationcontour}. To prove this fact notice that  there exist  paths going into each of the regions marked with the $+$-sign. Moreover they have to approach the pole with the negative residue contained in the respective region.}
\end{remark}

We say that a saddle point $(z,u^\ast)$ is  {\it relevant} if it is either maximally relevant or
 there exists a maximally relevant saddle point $(z,u^\ast_{max})$ such that 
 $G(z,u^\ast_{max})>G(z,u^\ast)$.
The next notion is very important for our story.

\begin{definition} In the above notation, we denote by $U_{rel}\subset \widehat \D$  the set of all relevant saddle points of the function $G(z,u)$, and by $U_{max}\subset \widehat \D$ the set of all maximally relevant saddle points.  
\end{definition}
Some  examples  of $U_{rel}$ and $U_{max}$ are given in \S~\ref{sec:quadratic}. Our main use of these sets will be to construct the tropical trace of $H$ and hence, in practice,  we  only need $U_{max}$  since it contains all the maximally relevant saddle points. We believe however that, conceptually,   $U_{rel}$ is more appropriate, as it encodes the ordering of branches by their height (given by $G(z,w)$) for different components of $\mathcal O$. It also is better suited for our sheaf-theoretical interpretation. 

\begin{lemma}
\label{lemma:andronov} 
{\rm i)} The set  $U_{rel}\cap \pi^{-1}(\mathcal O)\subset \widehat \D$ is open where the set   $\mathcal O$ has been defined in Lemma \ref{lemma:distinctharmonic}.

\noindent
{\rm ii)} % \textcolor{red}{ %(*** \"Annu ett $D$... ***)}
 Let $\bD\subset \mathcal O$ be an open simply-connected  subset. Then there exists a branch $u_i^\D(z)$ of $\widehat\D$, such that for each $z\in \bD$,  the maximally relevant saddle point of $G(z,u)$ is given by $(z,u^\ast_{max}(z))=(z,u_i^\D(z))$.
%\end{enumerate}
\end{lemma}
\begin{proof}It suffices to prove ii) which will follow the next claim.

\smallskip
\noindent 
(*) Suppose that $N\subset \mathcal O$ is a neighborhood of $z^\ast$, $u_i^\D$ is a branch of $\widehat \D$ defined in $N$ and $(z^\ast,u_i^\D(z^\ast))$ is a maximally relevant saddle point. Then there exists a neighborhood $\bD\subset N$ of $z^\ast$ such that $(z,u_i^\D(z))$ is maximally relevant for all $z\in \bD$.

Clearly (*) implies ii) as well as i).
Namely, if  $(z,u_j^\D(z))$ is relevant, but not maximally relevant, then,  by definition,  there is a maximally relevant $(z,u_i^\D(z))$ such that  $G(z, u_i^\D(z))>G(z,u_j^\D(z))$. Then (*) together with the continuity of $G$ and the assumption that $z\in \mathcal O$ implies that $(z,u_j^\D(z))$ is relevant in some neighborhood of $z$.
On the other hand,  in the case when $(z,u_j^\D(z))$ is maximally relevant, then $\{(z,u_j^\bD(z)), z\in \bD\})$ is an open neighborhood of the saddle point which implies that  i) is valid in this case as well.

\smallskip 
 In order to settle  (*), notice that by the definition of a maximally relevant saddle point,  the (connected component of the) level set of $G(z,u)$ passing through the saddle point has the following properties. Firstly, it consists of two enclosed ovals disjoint from each other except at  the saddle point and secondly,  $P^+$  is contained in the inner oval. In a neighbourhood of $z$,  the first property is obvious since, firstly, the compact level sets $G(z,u) =G(z,u_i^\D(z))=t$ vary continuously with $t$, and, secondly, since $z\in \mathcal O$ they only contain one saddle point. Therefore in a neighborhood of $z$,  these level sets cannot  change from being two enclosed ovals into a figure eight shape. The second property  is also obvious
since when $z$ varies the pole $P^+$ cannot escape from the inner oval  as long as this oval exists. Hence $(z,u_i^\D(z))$ is maximally relevant in some neighborhood $N$ of $z$, % \textcolor{red}{
 and (*) is proved. %This finishes the proof of the lemma.
% (*** Vi bevisar ett antagande formulerat i bÃ¶rjan av beviset av ett lemma? ***)}
\end{proof}

% \begin{figure}[htp]
% \label{fig:integrationcontour}
% \begin{center}
% \includegraphics[width=0.9\textwidth]{relevantsaddlepoint3.pdf}
% \end{center}
% \caption{The integration contour}
% \end{figure}

 Now let us consider the situation  as in Lemma \ref{lemma:situation} and Fig.~\ref{fig:integrationcontour}. Denote the region between $C_1$ and $C_2$ by $E$. Since the positive pole $P^+$ is contained in the interior of $C_1$, $(z, \infty)$ lies in the exterior of $C_2$, and since there are no other poles with positive residue, one has $G(z,u)<G(z,u^\ast_{max}(z))$ for $u\in E$. Hence there is a half-tubular neighborhood $N$  contained in $D$ with the boundary $C_1\cup \tilde C_1$, such that $\tilde C_1\cap C_1=u^\ast_{max}$, see Fig.~\ref{fig:integrationcontour}.  %\textcolor{red}{(*** OK implicit definition av $\tilde C_1$? ***)} 
 Clearly $\tilde C_1$ can be used as an integration contour in \eqref{eq:cauchy} and,  additionally,  it passes through $u_1$.  Further,   $G(z,u^\ast)< G(z,u^\ast_{max})$ for $u^\ast\in \tilde C_1$ and $u^\ast\neq u_{max}^\ast$.  Thus $\tilde C_1$ satisfies the condition of Corollary \ref{cor:uniform} for being a suitable integration contour.
Hence, we obtain the following key result. %\textcolor{red}{(*** Viktig fÃ¶ljdsats; pÃ¥minn lÃ€saren om vad $\D$ Ã€r i den? ***)}

\begin{corollary}\label{cor:imp} Assume that $z\in \mathcal O$, $\D$ is the saddle point curve \eqref{eq:saddle-points1}, and $(z,u^\ast_{max}(z))\in \D$ is the maximally relevant saddle point of $G(z,u)$.
Then, 
\begin{equation}
\label{eq:cauchy4} 
\lim_{m\to\infty}\vert I_{P,m,s_n,c}(z)\vert^{1/m}=e^{G(z, u^\ast_{max}(z))},
\end{equation}
where $c$ is any contour encircling $z$ once counterclockwise. 
\end{corollary}

In the following sections we will use Corollary~\ref{cor:imp} to prove that, up to  an additive constant, the logarithmic potential of the asymptotic  root-counting measure of the Rodrigues' descendants is the tropical trace  of $G$ taken on the set $U_{rel}\subset \widehat \D$ of relevant saddle points. (A similar fact can be found in the proof of Theorem~\ref{th:main}.) As before, let $\pi:(z,u)\mapsto z$ be the standard projection.

\begin{proposition}\label{prop:cont} The trace $\pi_{U_{rel *}}G(z), \ z\in \bC$, is a continuous and piecewise-harmonic function in the complement to the finite set of its poles. These poles are  logarithmic and have positive residues. Therefore the trace is a subharmonic $L^1_{loc}$-function. 
\end{proposition}

\begin{proof}First we will show that $\pi_{U_{rel *}}G(z)$ satisfies the conditions of Proposition \ref{prop:tropicaltraceopen} guaranteeing  continuity. The conditions  on $\Delta $ and $\pi(U_{rel})$ are true by Lemma \ref{lemma:distinctharmonic} and Lemma \ref{lemma:situation}, respectively.
By Lemma \ref{lemma:andronov} ii) the first condition (ii) a) is true. To settle (ii) b), assume that $C_1$ and $C_2$ are two adjacent connected components of $ \mathcal O$ 
and that $\pi_{U_{rel *}}G(z)=G(z,u_i^\D(z))$ if $z\in C_i,\ i=1, 2$.
 Let their common boundary be given by $G(z,u_j^\D(z))=G(z,u_k^\D(z)) $. We have to prove that either $\{j,k\}=\{1,2\}$, or else $u_1=u_2$. %
 %FATTAR EJ!!! 
 Assume first that $j\neq 1,2$. Then, as we move $z$ from $C_1$ across the boundary to $C_2$, the saddle point $(z,u_1(z))$ will not collide with any other saddle point. Hence if we are in the situation of   Fig.~4b) then nothing will happen. Indeed,  the continuously changing level curve passing through the saddle point can neither  change from the type of Fig.~4b) to the type of Fig.~4a)  nor can the pole $P^+$ escape from its inner oval  to create the shape shown in Fig.~4c). This means that $(z,u_1(z))$ will remain a maximally relevant saddle point, and  thus $u_1=u_2$ by Lemma \ref{lemma:andronov}. By symmetry, this proves the first part of  Proposition~\ref{prop:cont}.

Finally, we have to show that the tropical trace has no poles with negative residue in the finite plane. We argue by contradiction. Suppose that the tropical trace has such a pole. Then it must originate from a pole  of $G(z,u)$ on $\D$ with a negative residue. That is, this pole is of the form $(z,z)$. By Corollary \ref{prop:saddle-pointcurveprop} iii), the only possibilities for this pole are $(z_i,z_i),\; i=1,\dots ,d,$ where $P(z_i)=0$. In addition,  the negativity of the residue of a pole clearly implies that $\alpha>1$.

Without loss of generality,  assume that this pole of $G(z,u)$ coincides with $(z_1,z_1)$. Since it also induces a pole of the tropical trace, we get that  $\lim_{z\to z_1}(z,u^\ast_{max}(z))=(z_1,z_1)$.
 In a neighbourhood of $(z_1,z_1)\in \bC_z\times \bC_u$, we have 
  $$
  G(z,u)=\frac{1}{\alpha}\log\vert u-z_1\vert-\log\vert z-u\vert+O(1)=: B(z,u)+O(1).$$
  For fixed $z$, in a sufficiently small neighbourhood, the graph of $G(z,u)$ with respect to the variable $u$ will be close to the graph of $B(z,u)$. Making an affine change of coordinates, one can  assume that $z_1=0$, in which case  the only saddle point of $B(z,u)$ is $v^\ast=-\frac{z}{\alpha-1}$. Plotting the graph of $B(z,tz)$, for $t\in \bR$, we can find the positions of the poles with respect to the level curve passing through the saddle point $v^\ast$, see Fig.~\ref{fig:Bztz}.  One can easily conclude that this curve is of the type in Fig.~4 c) and hence the saddle point under consideration is not maximally relevant. This claim gives a contradiction and finishes the proof of the proposition.
  \end{proof}
  
  \begin{figure}
\begin{center}
\includegraphics[scale=.4]{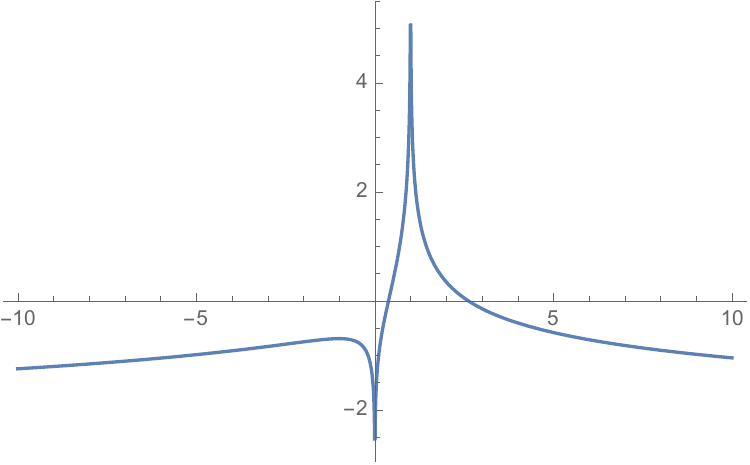}
\end{center}
\caption{Graph of $B(z,tz)$  with the local maximum $t=-\frac{1}{\alpha-1}$ corresponding to the saddle point $v^\ast=-\frac{z}{\alpha-1}$. The level curve passing through $v^\ast$ contains two other points which are visible in this  graph. They are obtained by intersecting the graph with the horizontal tangent at the latter local maximum. Their positions guarantee that the pole $z_1=0$ (which has a positive residue) is contained in the inner oval which implies that  this level curve is of the non-relevant type shown in Fig. 4 c). }
\label{fig:Bztz}
\end{figure}

\medskip 

 \subsection{Convergence of the logarithmic potentials almost everywhere}
 \label{sec:stirling}
 
By  Cauchy's integral formula  (\ref{eq:cauchy}),  the monic polynomial $\tilde \q_n$ which is proportional to the polynomial $\q_n$  is given by 
\begin{equation}
\label{eq:monicq}
\tilde \q_n(z):=\frac{(nd-([\alpha n]-1))!}{(nd)!} \q_n(z)=\frac{([\alpha n]-1)!(nd-([\alpha n]-1))!}{2\pi  i \cdot(nd)!}
\int_{c} \frac{P^n(u)\,\diff u}{(u-z)^{[\alpha n]}}.
 \end{equation}
 
 The degree of the polynomial $\q_n$ equals $d_n=nd-(m-1)$, where  $m=[\alpha n]$.  Recall that the logarithmic potential $L_{\mu_n}(z)$ of the root-counting measure $\mu_n$ of $\tilde \q_n$ can be expressed as 
 $$
 L_{\mu_n}(z)=\frac{1}{d_n}\log \vert \tilde \q_n(z)\vert.
 $$ %The logarithmic potential of the root-counting measure of  $\tilde q_n(z)$ equals $\vert \tilde q_n\vert^{1/d_n}$. 
 By (\ref{eq:fractionalpart}),  $$n=\frac{m}{\alpha}+s_n,$$ where $ 0\leq s_n< 1/\alpha$.
 Hence
 \begin{equation}
\label{eq:degree2}
d_n=\left(\frac{d-\alpha}{\alpha}\right)m+(s_nd+1)=\beta m+O(1),
 \end{equation}
where $\beta:=\frac{d-\alpha}{\alpha}$.
  \begin{lemma}
  \label{lem:constantas} In the above notation, 
$$  \lim_{n\to\infty}\frac{1}{d_n} \log \left(\frac{(m-1)!(nd-(m-1))!}{(nd)!}\right)=
\frac{\beta\log \beta-(\beta+1)\log(\beta+1)}{\beta} =: B. $$
 
  \end{lemma}

\begin{proof} Straight-forward calculation using Stirling's formula.
 \end{proof}

Now we can calculate the limit of the sequence $\left\{\frac{1}{d_n}\log \vert \tilde \q_n\vert\right\} $ of logarithmic potentials. Note that $d_n\sim \beta m$, 
take the logarithm of (\ref{eq:monicq}), and use Lemma \ref{lem:constantas} together with (\ref{eq:cauchy4}) in Corollary \ref{cor:imp}. 
 
 \begin{corollary}
  \label{cor:first limit} For any point $z\in \mathcal O$, %for which there exists a unique maximally relevant saddle point $(z,u^\ast_{max}(z))$, one has 
$$  \lim_{n\to\infty}L_{\mu_n}(z)=
B+ \frac{1}{\alpha\beta }\left(\log\vert P(u^\ast_{max}(z))\vert-\alpha \log \vert u^\ast_{max}(z)-z\vert\right).$$
  \end{corollary}

\subsection{Convergence  of $\{L_{\mu_n}\}_{n=1}^\infty$ in $L^1_{loc}$ and final steps of the proofs of Theorems~\ref{th:Cauchy} and  \ref{prop:primitive}}% of the main theorems
%Theorem~\ref{th:Cauchy}
%}
\label{section:finalsteps}

%The next theorem is our main result on the asymptotic limit of the sequence of the root-counting measures under consideration. Set
%\begin{equation}
%\label{eq:H}
%H_\al(z,u):=\beta^{-1}G(z,u):=\frac{1}{d- \alpha }(\log\vert P(u)\vert-\alpha \log \vert u-z\vert).
%\end{equation}

%Let $\pi: \bC_z\times \bC_u\ \to \bC_z$ be the usual projection, and let $\mathcal U\subset \mathcal D$ be the open set of maximally relevant saddle points, defined in \S~\ref{sec:contour}. Then $\pi$ induces a map $\tilde \pi: \mathcal U\to \bC$ and we can form the tropical trace $\tilde \pi_*H_\al(z,u) $ as in  \S~\ref{sec:trace}, which  by Proposition \ref{lm:continoustropicaltrace}, gives us a $L^1_{loc}$-function defined on %the dense and open
 %$ \bC$. 
%Since the complement of $0$ is a zero-set, $\tilde \pi_*H(z,u) $ is well-defined as a $L^1_{loc}$-function on $\bC$.

%\begin{theorem}
%\label{prop:primitiveNew} For any strongly generic polynomial $P$ of degree $d\ge 2$, 
%$$\lim_{n\to\infty}L_{\mu_n}(z)=B+\tilde\pi_*H_\al(u, z), $$
%as $L^1_{loc}$-functions.   
%Consequently,  
%$$\lim_{n\to\infty} \C_{\mu_n}(z)= 2\frac{\partial}{\partial z}\pi_*H_\al(u, z),
%$$
%and,  
%$$\lim_{n\to\infty} \mu_n= \mu:=\frac{2}{\pi}\frac{\partial^2}{\partial z\partial \bar z}\pi_*H_\al(z,u), 
%$$
%where the last two limits are understood in the sense of distributions.
%\end{theorem}

Corollary \ref{cor:first limit} provides the limit when $n\to \infty$ of the sequence $\{L_{\mu_n}\}$ a.e. in $\bC$, but to settle Theorem~\ref{prop:primitive} we need to prove that this limit also holds in $L^1_{loc}$. Vitali's convergence theorem (see e.g. \cite[Thm. 4.5.4 and Cor 4.5.5]{Bo}) gives an appropriate criterion for this to hold. In our situation it provides the following corollary.

\begin{lemma}
\label{lem:vitalis}Let $\{p_n\}$ be a sequence of monic polynomials of strictly increasing degrees  $d_n:=
\deg p_n\to\infty$ as 
$n\to\infty$. Denote by $\mu_n:=\frac{1}{d_n}\sum_{i=1}^{d_n}\delta(\zeta_i)$ the root-counting measure of $p_n$ and let $L_n(z):=\frac{1}{d_n}{\log\vert p_n(z)\vert}$ be the logarithmic potential of $\mu_n$.
%Let $Z:=\cup_nZ(p_n)$ be the total set of zeros of all the polynomials in the sequence.
Assume that
\begin{enumerate}
\item[(i)] there is a compact set $K\subset \bC$ containing all the zeros $\zeta_1,\dots,\zeta_{d_n}$ of $p_n$ for all $n=1,2,\dots $;
  \item[(ii)] the sequence $\{L_n(z)\}$ converges to some locally integrable function $L(z)$ pointwise a.e. in $\bC$.
%\item[iv)] $D$ consists of a finite number of smooth curve segments of finite length.
\end{enumerate}
Then, $L(z)$ is a $L^1_{loc}$-function and $\lim_{n\to \infty} L_n(z) = L(z)$ in the  $L^1_{loc}$-sense. %\textcolor{red}{REFERENS!} 

\end{lemma}
\begin{proof} By Vitali's convergence theorem, we only need to check the uniform integrability of our functions on an arbitrary fixed compact set $M\supset K$.
Let $E$ be a set with Lebesgue measure $\lambda(E)<\epsilon<1$. Introduce   
$$
\log_+(x):=\vert \log\vert x\vert\vert=f_{<\epsilon}(x)+f_{\geq\epsilon}(x),
$$ %\textcolor{red}{(*** $\vert\log\vert z\vert\vert$? ***)}
where $f_{<\epsilon}(x)=\log_+(x)=-\log \vert x\vert$ if $0<x\leq \epsilon$, and
$f_{<\epsilon}(x)=0$ if $x> \epsilon$. (Thus,  
$f_{\geq \epsilon}(x)=\log_+(x)$ if $x>\epsilon$, and $f_{\geq \epsilon}(x)=0$ %\textcolor{red}{(*** $f_{\geq \epsilon}(x) = 0$? ***)} 
if $0<x\leq \epsilon$).

\smallskip
We obtain
\begin{align}
\label{eq:vitali1}
&\int_E \left\vert L_{\mu_n}(z)\right\vert\,\diff\lambda(z)\leq \frac{1}{d_n}\sum_{i=1}^{d_n}\int_E \log_+(z-\zeta_i)\,\diff \lambda(z)\leq\\
&\leq \frac{1}{d_n}\sum_{i=1}^{d_n}\int_E f_{<\epsilon}(z-\zeta_i)\,\diff \lambda(z)+\frac{1}{d_n}\sum_{i=1}^{d_n}\int_E 
f_{\geq\epsilon}(z-\zeta_i)\,\diff \lambda(z):=I_1+I_2.
\label{eq:vitali2}
\end{align}
If $\bD_\epsilon(\zeta_i)$ is a disk of radius $\epsilon$ centered at $\zeta_i$, then
\begin{equation}
\label{eq:logR}
\int_{\bD_\epsilon(\zeta_i)}\vert \log \vert z-\zeta_i\vert \vert\,\diff z =-\pi \epsilon^2\left(\log \epsilon-\frac{1}{2}\right).
\end{equation} 
Hence 
$$
\int_E f_{<\epsilon}(z-\zeta_i)\,\diff\lambda(z)\leq \int_{\bD_\epsilon(\zeta_i)} f_{<\epsilon}(z-\zeta_i)\,\diff\lambda(z)=-\pi \epsilon^2\left(\log \epsilon-\frac{1}{2}\right)
$$
which implies that $$I_1\leq-\frac{1}{d_n}\left(d_n \pi \epsilon^2\left(\log \epsilon-\frac{1}{2}\right)\right)=O(\epsilon)$$ with a constant depending only on $\epsilon$. Let $\delta$ be the diameter of $M$. For the second sum  in \eqref{eq:vitali2}, let $m:=\max\{ -\log\epsilon, \log_+(\delta)\}$ be the upper bound of $f_{\geq\epsilon}(x-\zeta)$ for $x,\zeta\in M$. Then $I_2\leq m\lambda(E)\leq m\epsilon =o(1)$ as $\epsilon \to 0$. The estimates for $I_1$ and $I_2$ and \eqref{eq:vitali1}-\eqref{eq:logR} prove that 
$$\lim_{\lambda(E)\to 0 }\underset{n}{\sup} \int_E \vert L_n(z)\vert \diff\lambda(z)=0.$$ By Vitali's theorem  the desired
convergence in $L^1_{loc}$ then follows from the convergence a.e., see e.g., \cite[4.5.2-4.5.5]{Bo}. 
\end{proof}

  We now finalize our proof of Theorem \ref{prop:primitive}. Observe that  Corollary \ref{cor:first limit}, reformulated in terms of the tropical trace, says that we have pointwise convergence a.e. provided by the formula
  $$  \lim_{n\to\infty}L_{\mu_{[\al n],n,P}}(z)=
B+ \tilde\pi_* H(z),$$ where 
$\tilde\pi: {U}_{rel}\to \bC_z$ and $H(z,u)=\frac{1}{\beta}G(z,u)$.
Together with Lemma \ref{lem:vitalis} this fact implies that the sequence $\{L_{\mu_{[\al n],n,P}}(z)\}$    converges to the right-hand side of the latter formula in $L^1_{loc}$, and a fortiori is  convergent as a sequence of distributions. This is the first part of Theorem \ref{prop:primitive}. Since a measure $\mu_{\al,P}$ and its Cauchy transform are distributional derivatives of the logarithmic potential of $\mu_{\al,P}$, the other parts follow from the basic properties of distributions.

 Next we will settle Theorem~\ref{th:Cauchy}. The convergence $\mu_n\to \mu$ of Theorem \ref{prop:primitive} implies that $L_{\mu_{|\al n],n,P}}(z)\to L_{\mu_{\al,P}}(z)$ a.e.
%(Saff/Totik Theorem 1.6.9) 
and hence $$L_{\mu_{\al,P}}(z)=B+ \tilde\pi_* H(z)$$ as $L^1_{loc}$-functions. Taking  distributional derivatives gives for the Cauchy transform the relation
$$
\mathcal C_{\mu_{\al,P}}=2 \frac{\prt L_{\mu_{\al,P}}}{\prt z}=2\frac{\prt \tilde\pi_* H(z)}{\prt z}.$$

The distributional derivative of a continuous piecewise-harmonic subharmonic function is equal to  its usual derivative a.e., see e.g. \cite[Prop. 2]{BB}. By 
Proposition \ref{prop:cont}, the tropical trace $\tilde\pi_* H(z)$ is such a function. Let us now calculate its derivative a.e. using the statement of Lemma \ref{lemma:andronov} ii) saying that $\bC$ can be covered a.e. by open sets $O_i\subset \mathcal O,\ i\in I$, such that in each $O_i$ there is a
branch $u=u^\D(z)$ of the saddle point curve $\mathcal D$ for which the equality
\begin{equation}
\label{eq:theorem19}
H(u^\D(z),z)=\tilde\pi_* H(z)
\end{equation}
holds. 

\smallskip
In other words, in each $O_i$ we get 
\begin{equation}
\label{eq:cauchy19}
2\beta^{-1}\frac{\partial G(z,u^\D(z))}{\partial z}=2\beta^{-1}\left(\frac{\partial G}{\partial z}(z,u^\D(z))+\frac{\partial G}{\partial u}(z,u^\D(z))\frac{\partial u^\D(z)}{\partial z}\right)=\C_{\mu_{\al,P}}(z).
\end{equation}  The algebraic equation defining $\D$ (which is obviously satisfied by $u^\D$) says exactly that
$\frac{\partial G}{\partial u}=0$. Hence

\begin{equation*}
\label{eq:cauchy192}
\beta^{-1}\frac{1}{u^\D-z}
=\C_{\mu_{\al,P}}(z)\iff u^\D=z+(\beta \C_{\mu_{\al,P}}(z))^{-1}.
\end{equation*} 
On the other hand,  $u^\D$ satisfies equation \eqref{eq:saddle-pointsu}, and therefore 
the Cauchy transform $\C=\C_{\mu_{\al,P}}$ satisfies a.e. in $\bC$ the equation
\begin{equation}
\label{eq:symbolcurve2}(d-\al) \C=\frac{\diff\left(\log{P}\left(z+(\beta \C)^{-1}\right)\right)}{\diff z}.
\end{equation}

\smallskip
Formula~\eqref{eq:symbolcurve2} coincides with equation (\ref{eq:algebraicDiffEq}) which settles Theorem~\ref{th:Cauchy}, up to a small shift of the order of the derivative. We have actually proven that the sequence of root-counting measures for $\{(P^n)^{([\alpha n]-1)}\}$ converges, but using e.g., the main result of \cite{To}, we  also get that the sequence considered in 
Theorem~\ref{th:Cauchy} has the same limit as that of  $\{(P^n)^{([\alpha n]-1)}\}$. \qed

%\begin{proof}[\textcolor{red}{Proof of Corollary \ref{th:supp}}] Follows immediately from Theorem \ref{prop:primitive}, using the description of  the tropical trace in Theorem \ref{lm:tropicaltrace} as a piecewise harmonic function and taking into account the description of the poles of $H_\al(w,z)$ on the saddle point curve in Corollary \ref{prop:saddle-pointcurveprop} (iv).
%\end{proof}

 \subsection{The  symbol curve $\Ga$ is  an instance of our general construction of affine Boutroux curves}
 \label{sec:Boutroux3}   
 
 Recall the general construction of an aBc in \S~\ref{sec:scheme}. By following its steps  we will see now that the symbol curve $\Ga$ is a particular instance of this construction.  
 
\medskip 
The starting point is the function
$$H(z,u):=\frac{1}{\beta}G(z,u):=\frac{1}{d- \alpha}(\log\vert P(u)\vert-\alpha \log \vert u-z\vert).$$
% \textcolor{red}{(***  INTA SAMMA FORMEL SOM TIDIGARE!!! ***)}
It is well-defined and pluriharmonic 
for all $(z,u)\in \bC^2$ except at points where either $P(u)=0$ or $u=z$. Its differential is  the meromorphic $1$-form given by   
$$
d(H(z,u)):= \frac{1}{2(d-\alpha)}\left(   \frac{\al}{u-z}\diff z+\left(\frac{P'(u)}{ P(u) }-\frac{\alpha}{u-z}\right)\diff u\right).
$$ %\textcolor{red}{(*** $\diff(H_\al(z,u))$? ***)}
The saddle point curve  $\mathcal D\subset \bC_z\times \bC_u$ is the rational plane 
curve given by $$2(d-\alpha)\frac{\partial H}{\partial u}=\frac{P'(u)}{ P(u) }-\frac{\alpha}{u-z}=0.$$ Restricting $H$ to $\mathcal D$, we get a simplified expression for its differential given by
$$dH(z,u)=\frac{1}{2\beta}\cdot  
\frac{1}{u-z}\,\diff z,\quad (z,u)\in \mathcal D .$$

Consider the usual projection $\pi : \mathcal D\to \bC_z$ sending $(z,u)$ to $z$. Except for a finite number of branch points, $z$ is a local coordinate on $\mathcal D$.  % (reference???). 
Since $\D$ is smooth by Corollary \ref{prop:saddle-pointcurveprop}, in a neighborhood of every point $p=(z,u)\in \mathcal D$, the restriction of $H$ to $\mathcal D$ is a real-valued harmonic function satisfying
$$H(p)-H(p_0)=\text{Re } \int_{p_0}^{p}\frac{1}{\beta(u-z)}\,\diff z,$$  
where $p_0$ is another fixed point on $\D$. In particular, this implies that the form 
$$\omega = \frac{\diff z}{\beta(u-z)}\,$$ has imaginary periods on $\mathcal D$  which also follows from Theorem~\ref{pr:rational}. Notice that $\mathcal D$ is not an aBc, but if we change coordinates as explained below  the resulting curve will become an aBc.

\smallskip
Namely,  
the affine curve $\E $ introduced in  \S~\ref{sec:scheme} is constructed from the differential of $H$ as $\mathrm{Spec}\ \bC[z, \frac{1}{\beta(u-z)}]$. In our case this step just corresponds to the change of coordinates
$$
v=\frac{1}{\beta(u-z)},\ z=z\quad \iff\quad u=z+\frac{1}{\beta v},\; z=z.
$$
Hence, for the above pluriharmonic function $H$, the Boutroux curve $\E $ given in \S~\ref{sec:scheme} is precisely the symbol curve $\Ga$ defined by equation \eqref{eq:symbolcurve1} and  is satisfied by the asymptotic Cauchy transform $\C_{\mu_{\al,P}}$ according to Theorem \ref{th:Cauchy}.

%we see that $\mathcal D$ considered as a plane curve in  coordinates $(v,z)$ is an affine Boutroux curve, since $\omega=v\minispace\diff z$.
%\medskip
% The tropical trace of $H(p)$ under the projection $\pi_z: \mathcal D\to \bC P^1$ is  given by 
%$$\pi_*H(z,u):=\max_{p\in \D}\{  H(p) \vert \pi_z(p)=z\}. $$
%Note that this trace does not depend on the particular choice of coordinates in $\bC^2 $ used to describe $\mathcal D$ as a plane curve; it only depends on the projection $\pi_z$ and the function $H$. In terms of the coordinates $(v,z)$, $H(z,u)=\widetilde H(v,z)$, where
%$$ 
%\widetilde H(v,z)=
%\beta^{-1}G(z,u):=
%\frac{1}{\alpha\beta }(\log\vert P(z+(\beta v)^{-1})\vert+\alpha \log \vert (\beta v)\vert),
%$$
%and $\pi_*H(z)=\tilde\pi_*\widetilde H(z),$ where $\tilde\pi:\ (v,z)\mapsto z$.

%Since $d\widetilde H=\frac{1}{2}v\minispace\diff z$, this gives an alternative proof of Theorem~\ref{pr:rational} claiming  that $\mathcal C$ is an aBc.

\section{Differential equations satisfied by Rodrigues' descendants}\label{sec:proofs} 

\subsection {Deriving the differential equations}
In this section we obtain linear differential equations satisfied by the Rodrigues' descendants and use them to deduce equations ~\eqref{eq:algebraicDiffEq} and ~\eqref{eq:symbolcurve1} independently of most of the machinery in this paper, given a few additional assumptions. Having in mind future applications and generalisations, we derive a differential equation for the Rodrigues' descendants not just for a polynomial $P$, but for a more general  meromorphic function of the form $$f(z):=P(z)e^{T(z)}/Q(z),$$ where $P(z)\not\equiv 0,\,Q(z)\not\equiv 0$ and $T(z)$ are polynomials with $\gcd(P,Q) = 1$. In case $T\equiv 0$ and $Q\equiv 1$, we have $f(z)=P(z)$ considered in the present paper, see Corollary~\ref{cor:pnmDE}.

\begin{proposition}\label{prop:qnmDEMoreGeneral}
In the above notation  and for $d := \deg{P} + \deg{Q} + \deg{T}$, the Rodrigues' descendant $\R_{m,n, Pe^T/Q}(z)$ satisfies the linear homogeneous differential equation
\begin{equation}\label{eq:qnmMoreGeneral}
\begin{split}
& \sum_{i=0}^d\sum_{j=0}^i\sum_{k=0}^j \frac{(m+d-i+n(2j-i))\delta_{k,0} - nkT^{(k)}}{(m+d-i)!(i-j)!(j-k)!k!} P^{(i-j)}Q^{(j-k)}y^{(d-i)} = 0
\end{split}
\end{equation}
of order $d$. Here $\delta_{k,0}=\begin{cases} 1,  \text{if}\;  k=0\\ 0,\; \text{otherwise.}\end{cases}$ 
 
\end{proposition}

%(Observe that we do not consider the Rodrigues' descendants of the meromorphic functions $P(z)e^{T(z)}/Q(z)$ in the present paper. We plan to return to this topic in the future.)

\smallskip
As special cases of the latter statement we obtain the following three corollaries. 

\medskip
\begin{corollary}\label{cor:qnmDE}
The Rodrigues' descendant $\R_{m,n,P/Q}(z)$ of a rational function $P(z)/Q(z)$ satisfies the linear homogeneous  differential equation
\begin{equation}\label{eq:qnm}
\sum_{i=0}^{d}\sum_{j=0}^{i}\frac{m+d+(n-1)i-2nj}{(m+d-i)!\,(i-j)!\,j!}P^{(j)}Q^{(i-j)}y^{(d-i)} = 0
\end{equation}
of order $d = \deg{P} + \deg{Q}$.
\end{corollary}

\begin{corollary}\label{cor:pnmDE}
The Rodrigues' descendant $\R_{m,n,P}(z)$ of a polynomial $P(z)$ satisfies the  linear homogeneous differential equation
\begin{equation}\label{eq:pnm}
\sum_{i=0}^{d}\frac{(m-nd)-(i-d)(n+1)}{(d+m-i)!\,i!}P^{(i)}y^{(d-i)} = 0
\end{equation}
 of order $d = \deg{P}$.
\end{corollary}
\begin{remark}{\rm Differential equations satisfied by $\frac{\mathrm{d}^m}{\mathrm{d}z^m}\left(Q_1^{N_1}Q_2^{N_2}\dotsm Q_d^{N_d}\right)$, where $Q_1,\dotsc,Q_d$ are polynomials in $z$ and $N_1,\dots,N_d$ are nonnegative integers, were previously derived by Cior\^{a}nescu (see \cite{Ci}). As Cior\^{a}nescu remarks, one of these differential equations looks strikingly similar to Pochhammer's generalized Gaussian differential equation. A special case was later rediscovered by J.~M.~Horner (see \cite{Ho}).}
\end{remark}
 
The original Rodrigues' formula inspires the following consequence  of Corollary~\ref{cor:pnmDE}.

\begin{corollary}\label{cor:hornerDE}
The Rodrigues' descendant $y = \R_{n,n,P}(z) := \frac{\diff ^n}{\diff z^n}(P^n(z))$ satisfies the linear differential equation
\begin{equation}\label{eq:pnn}
\sum_{i=0}^{d}\frac{(d-1)-(i-1)(n+1)}{(d+n-i)!\,i!}P^{(i)}y^{(d-i)} = 0
\end{equation}
of order $d = \deg{P}$.
\end{corollary}

\begin{proof}[Proof of Proposition~\ref{prop:qnmDEMoreGeneral}]
Consider the first-order differential equation
\begin{equation}\label{eq:qnmFirstOrderDEMoreGeneral}
PQw' + n(PQ'-P'Q-PQT')w = 0,
\end{equation}

Clearly, if $f = Pe^T/Q$, then $w = f^n$ satisfies \eqref{eq:qnmFirstOrderDEMoreGeneral}. By differentiating both sides of \eqref{eq:qnmFirstOrderDEMoreGeneral} $\ell\ge d-1$ times (or $\ell > d-1$ times if $d=0$) and using Leibniz's rule for the derivative of a product, we get

\begin{equation}\label{eq:qnmLeibnitzMoreGeneral}
\sum_{i=0}^{\ell}\binom{\ell}{i} U^{(i)}w^{(\ell+1-i)} + n\cdot\sum_{i=0}^{\ell} \binom{\ell}{i} V^{(i)}w^{(\ell-i)} - n\cdot\sum_{i=0}^{\ell}\binom{\ell}{i} W^{(i)}w^{(\ell-i)}= 0,
\end{equation}
where $U := PQ,\,V := PQ' - P'Q$ and $W := PQT'$. In the first sum, remove the first term and replace $i$ by $r+1$ in the remaining sum. In the second and third sums, replace $i$ by $r$ and remove the last terms. By combining the three resulting sums and simplifying, equation (\ref{eq:qnmLeibnitzMoreGeneral}) becomes
\begin{equation}\label{eq:qnmSumMergeMoreGeneral001}
Uw^{(\ell+1)} + nV^{(\ell)}w - nW^{(\ell)}w + \sum_{r=0}^{\ell-1}  \binom{\ell}{r}\left(\frac{\ell-r}{r+1}U^{(r+1)} + nV^{(r)} - nW^{(r)}\right)w^{(\ell-r)} = 0.
\end{equation}
By changing the upper limit of summation in (\ref{eq:qnmSumMergeMoreGeneral001}) from $\ell-1$ to $\ell$, the terms $nV^{(\ell)}w$ and $-nW^{(\ell)}w$ are encompassed by the sum. Since $U,\,V$ and $W$ are polynomials of degrees at most $d,\,d-1$ and $d-1$, respectively %\textcolor{red}{(where we let $\deg{0} := 0$) (*** onödigt om vi antar $d\ge 2$, men jag tycker det kan vara kvar ***)}, 
and $\ell\ge d-1$, we can change the upper limit of summation further to $d-1$, since higher terms vanish. That is, we obtain the equation
\begin{equation*}\label{eq:qnmSumMergeMoreGeneral002}
Uw^{(\ell+1)} + \sum_{r=0}^{d-1}\binom{\ell}{r} \left(\frac{\ell-r}{r+1}U^{(r+1)} + nV^{(r)} - nW^{(r)}\right)w^{(\ell-r)} = 0,
\end{equation*}
or equivalently, if we replace $r$ by $i-1$, change the lower index of summation to $i=0$, and define $0\cdot V^{(-1)} = 0\cdot W^{(-1)} = 0$ as to not introduce any new terms,
\begin{equation}\label{eq:qnmSumMergeMoreGeneral003}
\sum_{i=0}^{d} \frac{\ell!}{(\ell-i+1)!\,i!}\left((\ell-i+1)U^{(i)} + niV^{(i-1)} - niW^{(i-1)}\right)w^{(\ell-i+1)} = 0.
\end{equation}

Since the terms in $U$ and $V$ contain two factors, while $W$ contains three factors, we expand their derivatives using Leibniz's rule as follows:
\begin{equation}\label{eq:LeibnizExpansion001}
U^{(i)} = (P\cdot Q\cdot 1)^{(i)} = PQ^{(i)} + \sum_{j=0}^{i-1}\sum_{k=0}^j \binom{i}{i-j,j-k,k} P^{(i-j)}Q^{(j-k)}\delta_{k,0},
\end{equation}
\begin{equation}\label{eq:LeibnizExpansion002}
\begin{split}
V^{(i-1)} & = (P\cdot Q'\cdot 1)^{(i-1)} - (P'\cdot Q\cdot 1)^{(i-1)} \\
& = \sum_{j=0}^{i-1}\sum_{k=0}^j\binom{i-1}{i-j-1, j-k,k}\left(P^{(i-j-1)}Q^{(j-k+1)} - P^{(i-j)}Q^{(j-k)}\right)\delta_{k,0},
\end{split}
\end{equation}
\begin{equation}\label{eq:LeibnizExpansion003}
W^{(i-1)} = (P\cdot Q\cdot T')^{(i-1)} = \sum_{j=0}^{i-1}\sum_{k=0}^j\binom{i-1}{i-j-1, j-k,k}P^{(i-j-1)}Q^{(j-k)}T^{(k+1)}.
\end{equation}
By inserting the expressions in \eqref{eq:LeibnizExpansion001}-\eqref{eq:LeibnizExpansion003} into \eqref{eq:qnmSumMergeMoreGeneral003} and simplifying, we see that
\begin{equation}\label{eq:tripleSumSimplify001}
\begin{split}
& \sum_{i=0}^d\binom{\ell}{i}PQ^{(i)}w^{(\ell-i+1)} + \sum_{i=0}^d\sum_{j=0}^{i-1}\sum_{k=0}^j \frac{\ell!}{(\ell-i+1)!(i-j)!(j-k)!k!}\left[(\ell-i+1)P^{(i-j)}Q^{(j-k)}\delta_{k,0}\right. \\
& + \left.n(i-j)\left(\left(P^{(i-j-1)}Q^{(j-k+1)} - P^{(i-j)}Q^{(j-k)}\right)\delta_{k,0} - P^{(i-j-1)}Q^{(j-k)}T^{(k+1)}\right)\right]w^{(\ell-i+1)}=0.
\end{split}
\end{equation}
By changing the upper index of summation from $i-1$ to $i$ in \eqref{eq:tripleSumSimplify001}, and using the convention that $0\cdot P^{(-1)} = 0$ as previously, the first sum is encompassed by the triple sum. Next, reverse the order of summation in the outer sum, and let $m := \ell-d+1$, which gives that
\begin{equation}\label{eq:tripleSumSimplify002}
\begin{split}
& \sum_{i=0}^d\sum_{j=0}^{d-i}\sum_{k=0}^j \frac{(m+d-1)!}{(m+i)!(d-i-j)!(j-k)!k!}\left[(m+i)P^{(d-i-j)}Q^{(j-k)}\delta_{k,0} + n(d-i-j)\times\right. \\
& \left.\left(\left(P^{(d-i-j-1)}Q^{(j-k+1)} - P^{(d-i-j)}Q^{(j-k)}\right)\delta_{k,0} - P^{(d-i-j-1)}Q^{(j-k)}T^{(k+1)}\right)\right]w^{(m+i)}=0,
\end{split}
\end{equation}
for all $\ell=m+d-1\ge d-1\iff m\ge 0$. Now let $(\ast )$ denote the equation obtained by replacing $w^{(m+i)}$ by $y^{(i)}$ in \eqref{eq:tripleSumSimplify002}. Clearly, $y = w^{(m)} = (f^n)^{(m)} = ((Pe^T/Q)^n)^{(m)}$ satisfies $(\ast )$. Thus, by reversing the order of summation in $(\ast )$ and simplifying, the proposition follows.
\end{proof}

Corollaries~\ref{cor:qnmDE},~\ref{cor:pnmDE}, and ~\ref{cor:hornerDE} are immediate consequences of Proposition~\ref{prop:qnmDEMoreGeneral}. 

\subsection{An algorithm for obtaining an algebraic equation satisfied by the asymptotic Cauchy transform $\C$}

In \S~\ref{section:finalsteps} we proved that the Cauchy transform of the asymptotic root-counting measure $\mu_{\al,P}$ satisfies the algebraic equations~\eqref{eq:algebraicDiffEq} and ~\eqref{eq:symbolcurve1}. We will now see that this also follows formally from equation ~\eqref{eq:pnm}, using a scheme suggested in \cite{BBS}. (It is observed that the formal derivation is validated under the assumption that hypotheses i) - iii) of Proposition 3 in loc. cit. hold.) 

\smallskip
In the notation  of \S~\ref{sec:introduction} our algorithm is as follows:

\medskip
\noindent 
Step 1: Multiply both sides of equation ~\eqref{eq:pnm} by the constant $(m+d-1)!$ (which we retained in the proof of Proposition ~\ref{prop:qnmDEMoreGeneral} until the final simplification).

\medskip
\noindent 
Step 2: Replace $m$ by $\alpha n$ and divide both sides by $y$.

\medskip
\noindent 
Step 3: Replace $y^{(d-i)}/y$ by $(n(d-\alpha)\C)^{d-i}$ in the resulting equation and divide both sides by $n^d$.

\medskip
\noindent 
Step 4: Let $n\to\infty$.

\medskip
By carrying out the above four steps, the resulting equation becomes 
\begin{equation}\label{eq:algebraicEqAgain}
\sum_{i=0}^{d}\frac{\al^{i-1}(\al-i)(d-\al)^{d-i}}{i!}P^{(i)}(z)\C^{d-i} = 0
\end{equation}
which is identical to equation ~\eqref{eq:algebraicDiffEq} up to the choice of the index of summation. As previously, using  the scaled Cauchy transform $\mathcal{W} = \frac{d-\al}{\al}\C$   we can  transform equation ~\eqref{eq:algebraicEqAgain} into
\begin{equation}\label{eq:simplerAlgebraicEqAgain}
\sum_{i=0}^{d}\frac{\al-i}{i!}P^{(i)}(z)\mathcal{W}^{d-i} = 0.
\end{equation}
Since $P(z+u) = \sum_{i=0}^{d}\frac{P^{(i)}(z)}{i!}u^i$ and $uP'(z+u) = \sum_{i=0}^{d}i\frac{P^{(i)}(z)}{i!}u^i$ (which can be seen from the Taylor expansions of the left-hand sides of these equations around $u=0$), we can use the change of variable $u = 1/\mathcal{W}$ along with these two sums to transform equation ~\eqref{eq:simplerAlgebraicEqAgain} into
\begin{equation}\label{eq:eqMultiplicity}
\al P(z+\mathcal{W}^{-1}) - \mathcal{W}^{-1} P'(z+\mathcal{W}^{-1}) = 0.
\end{equation}
Notice that if $z=b$ is a multiple zero of $P$, then $\mathcal{W} = (b-z)^{-1}$ solves equation ~\eqref{eq:eqMultiplicity} and consequently such $\mathcal{W}$ also solve equation ~\eqref{eq:simplerAlgebraicEqAgain}. Finally, equation ~\eqref{eq:eqMultiplicity} is easily rewritten in the form  
\begin{equation}\label{eq:CauchyOfPLike}
\al \W = \frac{P'\left(z+\W^{-1}\right)}{P\left(z+\W^{-1}\right)}=\frac{\diff \log{P}(z+\W^{-1})}{\diff z}.
\end{equation}
which can be transformed back into equation ~\eqref{eq:symbolcurve1}.

\smallskip
It should also be noted that the above algorithm  can be used with the differential equation ~\eqref{eq:qnmMoreGeneral} in Proposition ~\ref{prop:qnmDEMoreGeneral} as its starting point. From this procedure, it is possible to derive the algebraic equation
\begin{equation}\label{eq:moreGeneralCurve}
\al \W = \frac{\diff \log{f}(z+\W^{-1})}{\diff z}
\end{equation}
where $f(z) = P(z)e^{T(z)}/Q(z)$ is the meromorphic function defined in the beginning of this section and $\mathcal{W}$ is the scaled Cauchy transform of the asymptotic root-counting measure associated with $\R_{[\al n],n,Pe^T/Q}(z)$. This procedure (which involves slightly more elaborate Taylor expansions of functions such as $P(z+u)Q(z+u),\,u\cdot\frac{\partial}{\partial u}(P(z+u)Q(z+u))$ and $u\cdot Q(z+u)\cdot\frac{\partial}{\partial u}P(z+u)$) strongly suggests that the main results of this paper can be generalized to Rodrigues' descendants of such functions $f$.

\section{Case of a quadratic polynomial $P(z)$}\label{sec:quadratic}

The simplest instance of our %\textcolor{red}{problem (*** 
study
% ***)} 
occurs when $P$ is a quadratic polynomial. This case is closely related to the Legendre polynomials and the original Rodrigues' formula. For these polynomials, the asymptotic behavior of their zeros is known since long. In particular, for $\al=1$, the density of the asymptotic root distribution of the polynomials \eqref{eq:Leg} 
 equals
$$\eta(x)=\frac{1
}{\pi}\frac{1}{\sqrt{1-x^2}}\diff x,\; x\in [-1,1].$$  For general $\al$, the asymptotic measure has been recently calculated by Hoskins and Kabluchko \cite{HoKa}, by using methods quite different from ours. (Paper \cite{HoKa} refers to a draft version of the present text.)

\smallskip
In order to illustrate our methods, we will explicitly calculate the asymptotic measure  $\mu_{\al,P}$ for qudratic $P$. Without loss of generality, we can assume that $P(z)=z^2-1$ and $0<\al<2$, and 
 consider the polynomial sequence
$$P_n^{(\al)}(x):= \frac {\diff^{[\al n]} (x^2-1)^n}{\diff x^{[\al n]}}.
$$
By the Gauss-Lucas theorem, the zeros of these polynomials are contained in the interval $[-1,1]$, but we will not use this fact directly taking instead a more circuitous route via the saddle-point analysis suggested in the previous sections.

\smallskip
 For $P(z)=z^2-1$, the saddle point curve $\D\subset \bC_z\times \bC_u$ defined in the previous section gives
 $$
 P'(u)(u-z)=\al P(u)\iff (2-\al)u^2-2u z+\al=0. 
 $$

 The projection $\pi: \D \to \bC_z$ has  two branch points $b_\pm=\pm \sqrt{\al(2-\al)}$.  Since $0<\al<2$, both branch points are real.  The curve $\D$ has two branches: 
 $$
 u_\pm(z)=\frac{1}{2-\al}(z\pm\sqrt{z^2+\al^2-2 \al}).
 $$ 
  The monodromy along a contour that encircles both branch points is trivial, so the two branches are well defined in $V:=\bC\setminus [b_-,b_+]$.  Let $ u_\pm(z)$ be the branch that satisfies
 $u_\pm(z)=\frac{1}{2-\al}(z\pm\sqrt{z^2+\al^2-2 \al})
 $ on the interval $I_\al:=[b_+,\infty[$ of the positive real axis. 
 
 By Corollary \ref{prop:saddle-pointcurveprop}, we know that there is one branch which goes to $\infty$  as $\vert z\vert\to \infty$. This branch is clearly $u_+$.  By the same corollary,  the other branch $u_-$ goes to the root of $P'(u)=0$ which in our case is $u=0$. 
 
On the negative real axis the branch that goes to infinity  as $\vert z\vert\to\infty$ will be given by
$u_+(z)=\frac{1}{2-\al}(z-\sqrt{z^2+\al^2-2 \al})$, and hence  the identity
\begin{equation}
\label{eq;leftright}
u_+(-z)=-u_+(z)
\end{equation} holds everywhere, since it holds on a non-discrete set.
 
 \smallskip
The pluriharmonic function %\textcolor{red}{$H$ (*** dubbelkolla att det blir rätt funktion efter eventuella ändringar i \eqref{eq:H} ***)} 
given by \eqref{eq:H} equals 
\begin{equation}
\label{eq:Hkvadr}
H(z,u)=\frac{1}{2-\al}\left(\log\vert u^2-1\vert-\alpha\log\vert u-z\vert\right).
 \end{equation}

 For $u\in \mathcal D$, we get 
 \begin{equation}
 \label{eq:Honcurve}
 H(z,u)=\frac{1}{2-\al}\left(\log\left(\frac{2}{\al}\right)+\log\vert u\vert+(1-\al)\log\vert u-z\vert\right).
 \end{equation}

Set $H_\pm(z):=H(u_\pm(z),z)$ and notice that (except at the poles) $H_\pm(z)$ are harmonic functions well-defined in $V=\bC\setminus [b_-,b_+]$ and enjoy the following properties.
\begin{lemma}
\label{lemma:hplus} 
%\begin{enumerate} 

\noindent
{\rm i)} $H_+(-z)=H_+(z)$.

\noindent
{\rm ii)} $H_+$ can be extended to a continuous piecewise-harmonic function in $\bC$, possibly with singularities at $\pm 1$. 

\noindent
{\rm iii)} If $0<\al<1$, $H_+(z)$ has two poles at $z=\pm1$ with the   asymptotic near these poles given by $H_+(z)\sim \frac{1-\al}{2-\al}\cdot\mathrm{log}\vert z \mp 1\vert$. 

\noindent
{\rm iv)} If $1\leq\al<2$, $H_+(z)$ has no poles.
%\end{enumerate}
\end{lemma}
\begin{proof} Item i) follows from the relation $H(-u,-z)=H(z,u)$ together with \eqref{eq;leftright}.
To settle ii), notice that if $z\in [b_-,b_+]$, then $\sqrt{z^2-b_\pm^2}$ is purely imaginary providing that 
$\vert u_-(z)\vert=\vert u_+(z)\vert$ and $\vert u_-(z)-z\vert=\vert u_+(z)-z\vert$. This implies that $H_+(z)=H_-(z)$. Since the monodromy around the branch points will interchange $u_-(z)$ and $u_+(z)$, $H_+(z)$ is continuous for $z\in ]b_-,b_+[$.  Hence it is continuous except possibly at $\pm 1$.

To prove iii) and iv), observe that  for $\al=1$, equation \eqref{eq:Honcurve} shows that $H_+(z)$ has no pole at $z=\pm 1$.
 If $\al \neq 1$ the Taylor expansion of $u_+(z)$ at $z=1$ gives 
\begin{equation}
\label{eq:taylor}
u_+(z)=\frac{1+\vert \al-1\vert}{2-\al}+ \frac{\vert \al-1\vert +1}{\vert \al-1\vert(2-\al)}(z-1)+O(z-1)^2.
\end{equation}

For $1<\al< 2$,  \eqref{eq:taylor} implies that $u_+(1)=\frac{\al}{2-\al}$ and, in particular, both 
$u_+(1)$ and $u_+(1)-1$ are non-zero. Using \eqref{eq:Honcurve} and i), we obtain iv).
If $0<\al< 1$, then \eqref{eq:taylor} simplifies to 
\begin{equation}
u_+(z)=1+\frac{(z-1)}{1-\al}+O(z-1)^2.
\end{equation}
Thus $\log\vert u_+(z)-z\vert\sim \log\vert z-1\vert$ near $z=1$. Again using \eqref{eq:Honcurve} and i) we obtain iii). 
\end{proof}

\medskip
The level curve $\Delta: H_+(z)=H_-(z)$ and its complement $\mathcal O$  are shown in Fig.~\ref{fig: trace}. (Recall that by Lemma \ref{lemma:distinctharmonic}, we get that $\mathcal O$ is open and  dense in $\bC$ for an arbitrary strongly generic polynomial $P$.  For $P=z^2-1$, this fact is obvious directly. If $\al\neq 1$, then $\gamma$ consists of two ovals centered around $\pm 1$ together with the interval $[b_-,b_+]$ of the real axis. For $\al=1$, $\gamma$ is simply the interval $[-1,1]$. 

The additional circumstance which simplifies the application of  our results in the case of quadratic $P$ is that the set $U_{max}$ of maximally relevant saddle points has an easy description which we will now provide.  Set $W_+=\{ (z,u_+(z)),\quad z\in O\}$ and recall that $\pi: \bC^2\to \bC$ is the standard projection sending $(z,u)$ to $z$. 
\begin{figure}[H]
\begin{center}
\usetikzlibrary {calc,intersections,through,arrows}
\begin{tikzpicture}
    \node[anchor=south west,inner sep=0] at (0,0) {\includegraphics[width=\textwidth]{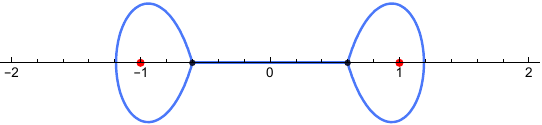}};
    \coordinate[label={[font=\Large]$H_-$}] (HM1) at (3.5,0.3);
	\coordinate[label={[font=\Large]$H_-$}] (HM2) at (9.2,0.3);
	\coordinate[label={[font=\Large]$H_+$}] (HP) at (6.4,0);
	\coordinate[label={[font=\Large]$\Gamma$}] (G) at (9.1,2.95);
		\coordinate[label={[font=\Large]$b_-$}] (BM) at (4.8,1.5);
	\coordinate[label={[font=\Large]$b_+$}] (BP) at (7.9,1.5);
\end{tikzpicture}
\caption{The curve $\Delta$ given  by $H_+(z)=H_-(z)$. The piecewise behavior of $\pi_*H(z)$ in the connected components of $\mathcal O=\bC\setminus\Delta$ is marked in each such component, while $\tilde\pi_*H(z)$ equals $H_+(z)$ in all the components. }
\label{fig: trace}
\end{center}
\end{figure}
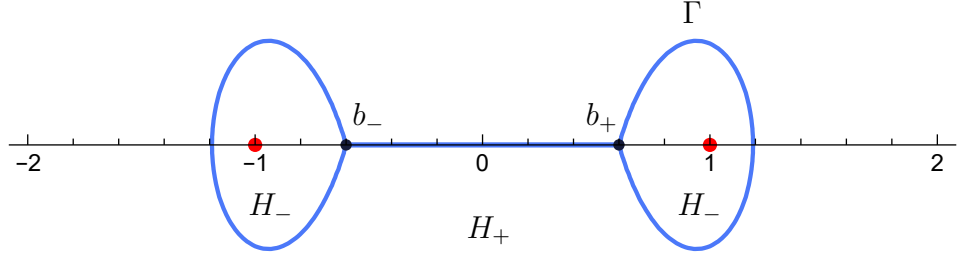
The above figure and the next proposition also describe the relevant saddle points, as a presheaf $F$ on $\mathcal O$, in the way that was explicated at the end of \S~2.3. Namely, over the complement of the two finite components of $\mathcal O$ each stalk consists of the two points in the fiber of $\pi$, while the stalk over the two components consists of the only point $(s, u_+(z))$.
\begin{proposition}
\label{prop:maxrelevant} In the above notation, we get 
$${U_{max}}\cap \pi^{-1}O=W_+.$$
\end{proposition}

%\textcolor{red} {Ngt om prestalk on Fig.~6??? Är notationen bra i Proposition \ref{prop:maxrelevant}?}
Note that the tropical trace $\tilde \pi_* H(z)$ for the map $\tilde \pi :U_{rel}\to \bC_z$ equals the tropical trace for the restriction $\tilde \pi:U_{max}\to \bC_z.$
Hence, our earlier results and Proposition~\ref{prop:maxrelevant} easily provide the asymptotic behavior of the zeros of the Rodrigues descendants of $P(z)=z^2-1$. For fixed $0<\al<2$, let $\mu_n:=\mu_{|\al n], n, P}$ be the root-counting measure of $\R_{[\alpha n],n,P}=(P^n)^{([\alpha n])}(z)$ with logarithmic potential $L_{\mu_n}(z)$ and Cauchy transform $\C_{\mu_n}(z)$. (Note the very explicit description of $\tilde \pi_* H(z)$ in item i) below.)

\begin{corollary}
\label{prop:kvadratresultat} 
\begin{enumerate} 
\item[i)] The  equality 
$$\lim_{n\to\infty} L_{\mu_n}(z)=\tilde \pi_* H(z)+B =H_+(z)+B ,$$ 
is valid in the $L^1_{loc}$-sense, where $H(z,u)$ is given by \eqref{eq:Hkvadr} and the constant $B$ is defined in Lemma \ref{lem:constantas} with $d=2$.

\smallskip
\item[ii)] For $0\leq\alpha < 1 $, one has 
$$\mu:=\mu_{\al,P}:= \lim_{n\to\infty} \mu_n= \frac{2}{\pi} \frac{\partial^2 H_+}{\partial z\partial \bar z}.$$ 

\smallskip
The measure $\mu$ is the probability measure explicitly given by 
$$\mu=\frac{1}{(2-\al)\pi}\frac{\sqrt {b_+^2-x^2}}{1-x^2}\diff s +\frac{1-\al}{2-\al}\delta(1)+\frac{1-\al}{2-\al}\delta(-1) , \quad x\in [b_-,b_+],$$ 
and its continuous part is supported on $[b_-,b_+]$.

\smallskip
\item[iii)] For $1\leq\alpha < 2 $, one has 
$$\mu:= \lim_{n\to\infty} \mu_n=  \frac{2}{\pi} \frac{\partial^2H_+}{\partial z\partial \bar z}.$$ 

\smallskip
Here $\mu$ is the probability measure explicitly given by 
$$\mu=\frac{1}{(2-\al)\pi}\frac{\sqrt {b_+^2-x^2}}{1-x^2}\diff s, \quad x\in [b_-,b_+].$$

\smallskip
\item[iv)] The Cauchy transform $\C_\mu(z)=\frac{2\partial L_\mu}{\partial z}$ is the analytic continuation of the function 
$$
\C_\mu(x)=\frac{2\al}{(\al-1)x+ \sqrt{x^2-b_+^2}},\quad x\in ]1,\infty]
$$
 to the domain $\bC\setminus [b_-,b_+]$.

It satisfies equation \eqref{eq:eqMultiplicity} which in our special case reduces to 
\begin{eqnarray*}
\al P\left(z+\frac{\al }{2-\al}\C^{-1}\right)-\frac{\al }{2-\al}\C^{-1}P'\left(z+\frac{\al }{2-\al}\C^{-1}\right)=0\iff\\
\frac{\al}{2-\alpha}+\frac{2(1-\al)}{2-\al}z \C+(1-z^2) \C^2=0.
\end{eqnarray*}
\end{enumerate}
\end{corollary}

%\textcolor{red}{(*** Byt eventuellt $x$ mot $z$ i ii) och iii) för att matcha iv) (och resten av artikeln) bättre. $\C_\mu(z)$ is iv) håller sig ju också initialt på reella axeln... ***)}

\begin{remark}{\rm 
The above items ii) and iii) have previously been derived by  Hoskins and Kabluchko in \cite[4.2]{HoKa}. As was pointed  to us by an anonymous  referee item iii) in particular is known in the literature as the Kesten-McKay law, see \cite{DuEd}. 
}
\end{remark}

\begin{proof} The previous proposition implies that the fiberwise maximum on $U_{max}$ in the definition of $\tilde \pi_* H(z)$ is taken just on the single sheet $W_+,$ and hence $\tilde \pi_* H(z) =H_+(z)$ as $L^1_{loc}$-functions. Theorem \ref{prop:primitive} then gives i).

Since by Lemma \ref{lemma:hplus}, $\tilde \pi_* H(z)$ is continuous and harmonic a.e. in $\bC$ its distributional derivative equals its derivative a.e., see e.g. \cite[Prop.2]{BB}. Hence 
$$\C_\mu(z)=
2\frac{\partial L_\mu}{\partial z}=\frac{2}{2-\al}\left(\frac{u_+'(z)}{u_+(z)} +\frac{u_+'(z)-1}{u_+(z)-z}   \right),
$$
which can be reduced to iv). 

By the Plemelj-Sokhotski formula, the distributional derivative $\frac{\partial \C_\mu}{\partial\bar z}$ along  $[b_-,b_+]$ equals one half of the absolute value of the jump  of the function across the line segment $[b_-,b_+]$, i.e., $\frac{\vert C_+-C_-\vert}{2}$, where $C_+$ och $C_-$ are the two values of the analytic continuation of $\C_\mu$ on both sides of $[b_-,b_+]$, see e.g. \cite[Lemma 2]{BB}. This fact explains the expressions for the continuous part  of the measure  in ii) and iii). For $0<\al< 1$, the point mass contributions are immediate from Lemma \ref{lemma:hplus} iii). The equation for the Cauchy transform follows from Theorem~\ref{th:Cauchy} and \eqref{eq:symbolcurve2}.
\end{proof}

The description of $\tilde\pi_*H(z)=H_+(z)$ in  Corollary \ref{prop:kvadratresultat} can be contrasted with the behavior of the fiberwise maximum $\pi_*H(z)$.
 In each of the three regions determined by $\Delta$, $\pi_*H(z)$ equals the maximum of $H_-$ and $H_+$.  It is also clear that $\pi_*H(z)=H_+(z)$ for large $\vert z\vert$, and hence in the whole unbounded region.
Inside the ovals $H_-(z)$ will be the larger function, and hence $\pi_*H(z)=H_-(z)$. Thus $\pi_*H(z)$ is the piecewise-harmonic function shown in Fig.~\ref{fig: trace}. 
By an argument similar to the proof of Lemma \ref{lemma:hplus}, $H_-(z)$ has no poles when $0<\al<1$, and so $\pi_*H(z)$ is subharmonic in the whole plane. 
On the other hand, if $1<\al<2$,  $H_-(z)$  has poles at $z=\pm 1$ and near these poles $H_-(z)\sim \frac{1-\al}{2-\al}\log\vert z \mp 1\vert$ respectively. So $\pi_*H(z)$ is subharmonic only in the complement of the poles. 

\begin{proof}[Proof of Proposition \ref{prop:maxrelevant}]

 %Note that $\pi_*H(z)\geq\tilde\pi_*H(z)$ and that $\tilde\pi_*H(z)$ can change from one branch of $H(z)$ to another only at the level set in figure \ref{fig: trace} (by Lemma ??, the level set is the complement to $O$). Hence by the discussion preceding the proof, there are only two possibilities. Either $\tilde\pi_*H(z)=\pi_*H(z)$ or, as we claim, $\tilde\pi_*H(z)$ equals to $H_+$ in the complement to $[b_-,b_+]$.

 By Lemma \ref{lemma:situation} for each $z\in \mathcal O$, there is a unique maximally relevant saddle point. Thus we only need to show that this saddle point is not $u_-(z)$.  
By Lemma \ref{lemma:andronov}, a change of the branch that determines the maximally relevant saddle point can only occur when $z$ moves to a different connected component of $\mathcal O$. Consequently, it suffices to analyse the saddle point behavior of $H(z,u^\ast)$ 
for some points $z$ in each  connected component of $\mathcal O$, say, on the real axis. 

The left-right symmetry induced by $H(-z,-u)=H(z,u)$ and 
$u_\pm(-z)=-u_\pm(z)$ 
shows that the level curves through $u_\pm(-z)$ and $u_\pm(z)$ are the mirror images of each other. Hence, by the definition of maximally relevant saddle points,  it is enough to consider a real point $z_1$ inside the right oval, in addition to a real point 
$z_2$ outside both ovals, i.e., in the unbounded component.  

Observe that  for a saddle point $w(z)$ not to be maximally relevant,  it suffices that the two paths $\gamma_i,$ $i=1,2,$ of maximal growth from $w(z)$ have the same pole as their endpoints; see Remark \ref{remark:path}. Note that such an endpoint can only coincide with  $P^+$, $\infty$ or $u_\pm(z)$, all of which belong to the real axis if $z\geq b_+$. 

For  $z$ real,  we have the additional symmetry 
$H(z,\bar u)=H(z,u)$ which implies the following: 

\noindent
 i) each level set of $H(z,u)$ in the complex $u$-plane passing through a saddle point on the real axis is invariant under the complex conjugation;  

\noindent 
  ii)  the gradient field of $H(z,u)$  in the complex $u$-plane is invariant under complex conjugation w.r.t. the variable $u$.  

\smallskip  
  In particular,  ii) implies that if the oriented paths $\gamma_i,\ i=1,2,$ starting at the real saddle point $u_-(z)\in \bR$, do not initially follow  the real axis, they will  have the same endpoint which necessarily belongs to the real axis. This endpoint  has three possibilities; it can either be $P^+$, $\infty$, or  $u_+(z)$. In the first two cases, $u_-(z)$ is not maximally relevant, by Remark \ref{remark:path}. In the latter case $u_-(z)$ will be relevant, but not maximally relevant. This follows from the fact that in order for a maximally relevant saddle point to exist, the two paths of maximal growth from $u_+(z)$ must have different endpoints $P^+$ and $\infty$. Hence $u_+(z)$ is maximally relevant. Thus  we have established that $u_+(z)$ is maximally relevant in all three cases under the assumption that $\gamma_i, \ i=1,2$, do not initially follow  the real axis. Let us finally show that our latter assumption holds, i.e., indeed $\gamma_i, \ i=1,2$, do not initially follow  the real axis.

\smallskip
A simple computation gives  
\begin{equation}
\label{eq:sign}
\frac {\diff H(z,t)}{\diff t}=\frac{(t-u_-(z))(t-u_+(z))}{(2-\al)(t^2-1)(t-z)}.
\end{equation} Now if $z\ge 0$ is large, then $u_+(z)\ge 0$ will also be large while $u_-(z)$ will be close to $0$. Hence, for $t$ close to $u_-(z)$,  the sign of
 $\frac {\diff H(z,t)}{\diff t}$  will be equal to the sign of $-(t-u_-(z))$, which implies that $u_-(z)$ is a local maximum of $H(z,t)$, $t\in \bR$. In particular, the paths of maximal growth cannot start out along the real axis implying that $u_-(z)$ is not maximally relevant. If $z$ is contained in the oval in the right half-plane, we have that 
 $ u_-(z)<u_+(z) $ and we can assume  that for $\al\neq 1$,  $z<1$. A straightforward calculation shows that if $1<\al< d$, we get $z<1< u_-(z)<u_+(z)$, and if $0<\al< 1$, then 
 $ u_-(z)<u_+(z)<z<1$. Analyzing signs in \eqref{eq:sign} we again conclude that $u_-(z)$ is  a local maximum of $H(z,t)$ for real $t$. The result follows.
\end{proof}

A more instructive illustration of why $u_+(z)$ is a maximally relevant saddle point is given by Fig.~\ref{fig:levelsets} -- \ref{fig:levelsets2} which show the level curves passing through both saddle points.

Let us start with the case $0<\al< 1$ presented in Fig. \ref{fig:levelsets}. The dashed line is the level curve of $H(z,u)$ through $u_+(z)$ and we see from its definition and the location of the poles that  $u_+(z)$  is a maximally relevant saddle point. There exists a path of maximal ascent going from $u_+(z)$  to $P^+$ and another one going from $u_+(z)$  first to $u_-(z)$ and then to $\infty$. The non-dashed curve is the level curve of $H(z,u)$  through  $u_-(z)$. We can conclude that this saddle point is non-relevant, either by definition, or since it is impossible to reach $P^+$ by an ascending path from $u_-(z)$. Hence the only maximally relevant saddle point is $u_+(z)$ and
 inside the right oval we get $\tilde\pi_*H(z)=H(z,u_1(z))=H_+(z)$ as was already proven before.

For $1\leq\alpha < 2$, the level curves look differently, see Fig. \ref{fig:levelsets2}.  But again it is possible to use the position of the poles and the definition of maximally relevant saddle points to obtain the same result.

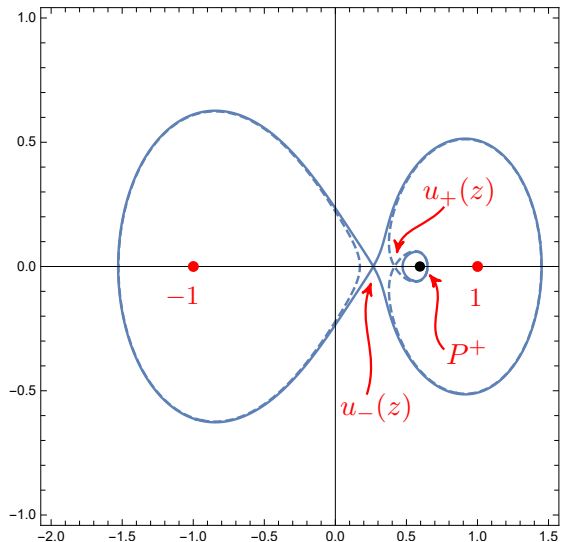
\begin{figure}[H]
\begin{center}
\usetikzlibrary {calc,intersections,through,arrows}
\begin{tikzpicture}[->,>=stealth',auto,node distance=3cm,
  thick,main node/.style={circle,draw,font=\sffamily\Large\bfseries}]
    \node[anchor=south west,inner sep=0] at (0,0) {\includegraphics[scale=.8]{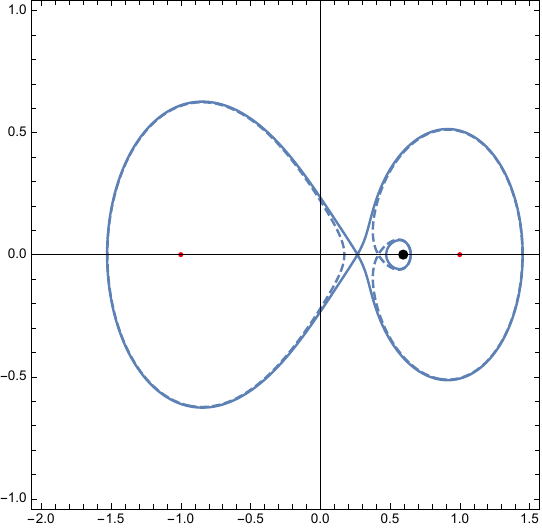}};
    \coordinate[label={[font=\normalsize,red]$-1$}] (N1) at (2.32,3.05);
	\coordinate[label={[font=\normalsize,red]$1$}] (P1) at (6.22,3.05);
	\coordinate[label={[font=\large,red]$u_-(z)$}] (WM) at (4.9,1.5);
	\coordinate[label={[font=\large,red]$u_+(z)$}] (WP) at (6,4.37);
	\coordinate[label={[font=\large,red]$P^+$}] (PP) at (6.1,2.28);
	\coordinate (S1) at (4.77,2);
	\coordinate (G1) at (4.83,3.49);
	\coordinate (S2) at (5.8,4.5);
	\coordinate (G2) at (5.17,3.88);
	\coordinate (S3) at (5.8,2.6);
	\coordinate (G3) at (5.6,3.6);
	\coordinate (RED1) at (2.4646,3.705);
	\coordinate (RED2) at (6.2432,3.705);
	
    \path[every node/.style={font=\sffamily\small},color=red,line width=0.3mm]
	(S1) edge[out=72, in=250] node [left] {} (G1);
	\path[every node/.style={font=\sffamily\small},color=red,line width=0.3mm]
	(S2) edge[out=230, in=80] node [left] {} (G2);
	\path[every node/.style={font=\sffamily\small},color=red,line width=0.3mm]
	(S3) edge[out=135, in=310] node [left] {} (G3);
	
	\tkzDrawPoint[size=3.3,color=red](RED1)
	\tkzDrawPoint[size=3.3,color=red](RED2)
\end{tikzpicture}
\caption{The level curves of $H(z,u)$ through the saddle points, $0<\al< 1$. (Here $\al=0.2$ and $z=0.615$.)}
\label{fig:levelsets}
\end{center}
\end{figure}

Let us now qualitatively describe what happens to $\mu_{\al,z^2-1}$ when $\alpha$ increases from $0$ to $2$. Part of the mass of the two point measures at $\pm 1$ initially moves out of $\pm 1$ to the continuous measure supported on the interval $I_\al:=[b_-,b_+]$. The latter measure then expands with increasing $\al$, until its support becomes the whole interval $[-1,1]$ at $\al=1$. Then when $1<\al<2$ there is no mass left at $\pm1$ and the support of the continuous measure shrinks toward the origin and vanishes when $\al=2$. 
In particular, for all $0<\al<2$, the support of the asymptotic measure is contained in $[-1,1]$, as predicted by the Gauss-Lucas theorem, and for all  $\al\neq 1$, it is (except for the possible point masses) strictly smaller than $[-1,1]$. 

To summarize: we want to find $\tilde\pi_*H(z)$  to obtain the asymptotic root-counting measure. By considering the complement $\mathcal O$ of the non-simple locus $\Delta$,  it suffices to check the situation in the finite number of connected components of $\mathcal O$, see  Proposition \ref{prop:tropicaltraceopen}.  Furthermore, in each such component it is enough to  analyze the behavior  of the paths of maximal ascent at a single point. In principle,  such analysis can be carried out for higher degree polynomials $P$ as well, but is substantially more involved. 
 
\medskip

\begin{figure}
\begin{center}
\includegraphics[scale=.35]{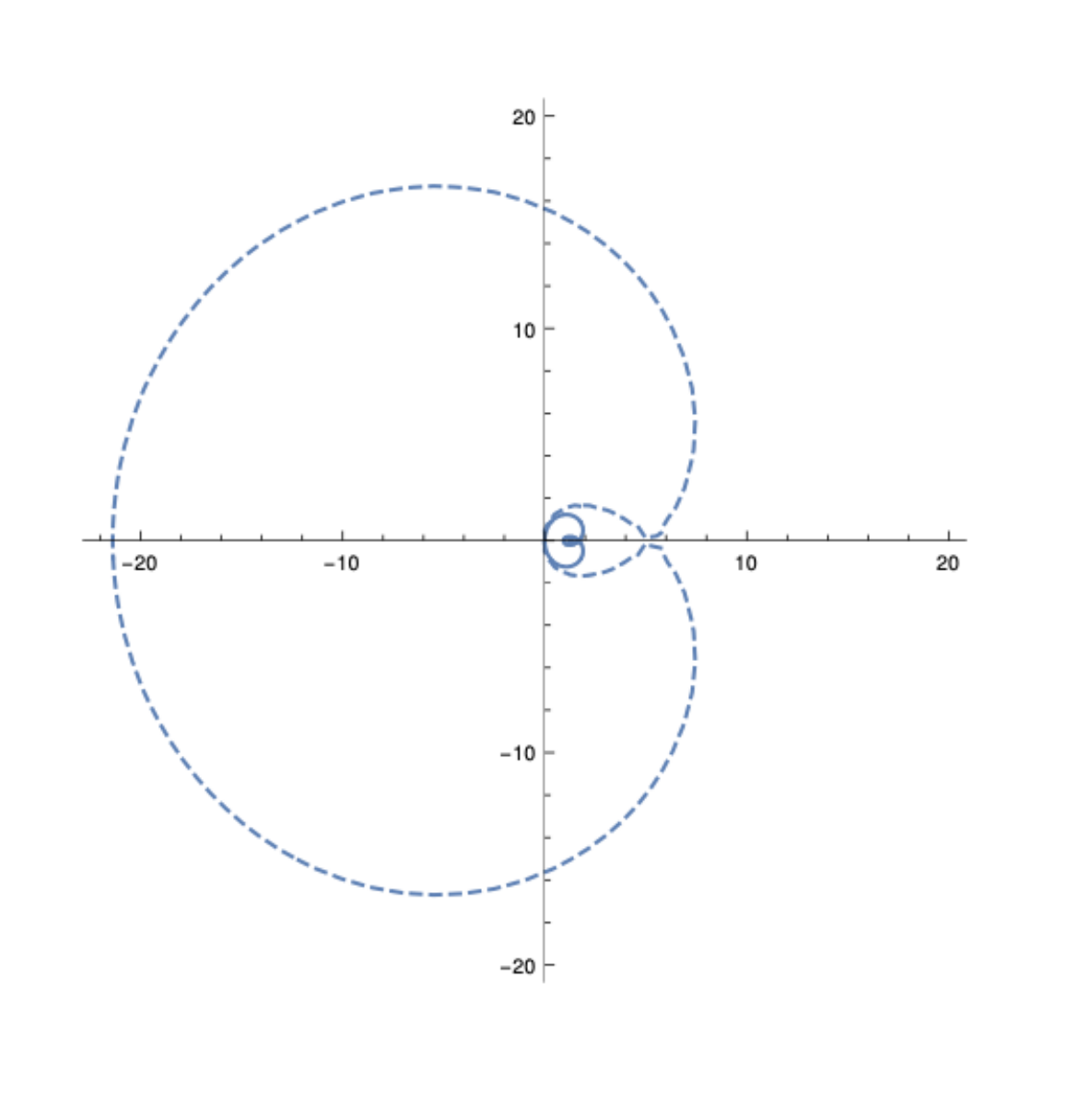}
\end{center}
\caption{The level curves of $H(z,u)$ through the saddle points, $1<\al< 2$. (Here $\al=1.2$ and $z=0.615$.)}
\label{fig:levelsets2}
\end{figure}
%Again we see from the definition and the location of the poles that $w_+(z)$ is a maximally relevant saddle point. The undashed curve is the level set of $H(z,u)$  through  $w_-(z)$, and we see that this saddle point is not relevant. Hence inside the right hand oval $\tilde\pi_*H(z)=H(w_+(z),z)=H_+(z)$, as asserted.

\section {Final remarks and open problems}  \label{sec:final} 

\smallskip
\noindent
{\bf 1.} Practically all the results of the present paper can be generalized to the case where $f$ is a rational function instead of a polynomial. %which we plan to do in the future.  %forthcoming paper \cite{BHS}. 
However poles of a rational function restrict the possibility of deformation of the integration contour used in \S~5 which leads to a more delicate situation  requiring special analysis.

\smallskip
\noindent
{\bf 2.} The set-up of the present paper can also be  randomized and generalized as follows. 
 Let $\xi$ be a probability measure compactly supported in $\bC$. Denote by $P_n=\prod_{i=1}^n(x-\xi_i)$ a random polynomial of degree $n$ whose roots are i.i.d. random variables sampled on $\xi$. Given a sequence $\A=\{\al_n\}$ of non-negative integers, set $Q_n=P_n^{(\al_n)}$ and denote by $\mu_n$ the root-counting measure of $Q_n$. Results from  the recent papers \cite{PR, Ka} motivate the following guess.    
% (See also results of numerical experiments using deterministic sampling in \cite{H}). 
%DEFINE WHAT CONVERGES IN SUPPORTS MEANS!

\begin{conjecture}\label{prop:central}
In the above notation, the following two statements hold:

\medskip\noindent
\rm{(i)} if $\frac{\al_n}{n}\to 0$, then the sequence $\{\mu_n\}$  converges in probability to $\xi;$

\medskip\noindent
\rm{(ii)} if $\frac{\al_n}{n}\to \al,\; 0<\al<1$, then the sequence $\{\mu_n\}$ converges in probability to a measure $\xi_\al$  whose support is contained in the convex hull of the support of $\xi$;
\end{conjecture}

Results of the present paper can be interpreted in the above terms  as follows. We start with a  discrete probability measure $\xi$ assigning the mass $\frac{1}{d}$ to each of the $d$ zeros of $P(z)$. Then we sample this measure uniformly and deterministically $nd$ times, by forming the sequence of polynomials $\{P_n(z):=P^n(z)\}$. Finally, fixing $0<\al<\deg P$,   we differentiate each $P_n(z)$ $[n\alpha]$ times. This procedure creates a sequence of polynomials $\{Q_n(z)\}$ and an associated sequence of root-counting measures $\{\mu_n\}$. The proportion between the number of derivations and the number of sampled points has the limit
$$A:=\frac{\alpha}{d}.$$

Observe now that $\frac{1}{d} \log\vert P(z)\vert$ equals the logarithmic potential $L_\xi(z)$ of the above measure $\xi$. Therefore our presentation of the asymptotic measure in Theorem \ref{prop:primitive} can be interpreted as
\begin{equation}
\label{eq:general}
\lim_{n\to \infty}\mu_n= \frac{2}{\pi}\frac{\partial^2}{\partial z\partial \bar z}\tilde\pi_*\left(\frac{1}{1-A}L_\xi(w)-\frac{A}{1-A} \log \vert w-z\vert\right),
\end{equation}
where $\tilde\pi_*$ denotes the fiberwise maximum of the function $$H(z,w):=\frac{1}{1-A}L_\xi(w)-\frac{A}{1-A} \log \vert w-z\vert$$ considered on the open subset $U_{max}\subset \mathcal D$ of maximally relevant saddle points on the curve 
$$\mathcal D=\{ (z,w)\in \bC_z\times \bC_w\;\vert \;\frac{\partial H(z,w)}{\partial w}=0 \},$$ 
where 
$$2\frac{\partial H(z,w)}{\partial w}=\frac{1}{1-A}\C_\xi(w)-\left (\frac{A}{1-A}\right ) \frac{1}{ w-z}=0\iff \C_\xi(w)=\frac{A}{z-w}.$$
\smallskip
Observe now that  the right-hand side of formula~\eqref{eq:general}  depends on the measure $\xi$ and does not explicitly use the underlying   polynomial $P$.  Hence one can hope that \eqref{eq:general} might  make sense for an arbitrary probability measure $\xi$ where a random polynomial sequence  $\{Q_n(z)\}$ is obtained by independent sampling of roots according to $\xi$. 
At least, it seems plausible that the relation (\ref{eq:general}) holds for a much more general class of probability measures $\xi$ than the  very special  measure originating from a univariate polynomial $P$ which we described above and which implicitly appears in our paper.

\smallskip
\noindent
{\bf 3.} For a given strongly generic polynomial $P$ and $0<\al<\deg P$, the asymptotic root-counting measure $\mu_{\al,P} = \lim_{n\to\infty} \mu_{[\al n], n,P}$ defined in \S~\ref{sec:introduction} is supported inside a certain intriguing domain $\Sh_P\subset \bC$ which we call the \emph{shadow} of a polynomial $P$; see examples in Fig. ~\ref{fig:polShadowExamples} where $\Sh_P$ is shown in red. Further, we say that a polynomial $P$ of degree at least $3$ has roots \emph{in convex position} if  each of them lies on the boundary  of $\mathrm{Conv}_P$, where $\mathrm{Conv}_P$ denotes their convex hull.

\begin{conjecture} \label{conj:main}For any strongly generic polynomial $P$ of degree at least $3$ whose roots are in convex position, but do not form a regular polygon, one has

\noindent
{\rm(i)} $\Upsilon_P\subset \mathrm{Conv}_P$ is a concave domain, i.e., the boundary of $\Upsilon_P$ consists of a finite number of smooth curves $\gamma_k,\,k=1,\dots,\eta$, such that the interior of the line segment connecting any two distinct points on $\gamma_k$ lies entirely in the complement of $\Upsilon_P$.

\noindent
{\rm(ii)} The boundary of $\Upsilon_P$ is contained in the union over $\al\in [0,d]$ of all critical values w.r.t. the variable $z$ of the rational function
\begin{equation}\label{eq:criticalValues}
F_\al(z)=z-\al \frac{P(z)}{P^\prime(z)}.
\end{equation}
 Equivalently, the boundary of $\Upsilon_P$ consists of all values of $u$  for which the family $$\Phi(\al,z,u)=\al P(z)+(u-z)P^\prime(z)$$ has a multiple root w.r.t. $z$ for some fixed $\al\in [0,d]$.
\end{conjecture}
%\textcolor{red}{(*** Borde inte den sista meningen snarare vara ``Equivalently, the boundary of $\Upsilon_P$ is contained in all values of $u$ for which the family [...]''??? ***)}

Curiously, the function $F_\al(z)$ in \eqref{eq:criticalValues} bears a strong resemblance with the iterated expression in the relaxed Newton's method, see e.g. \cite{Su}. Additionally, the family $\Phi(\al,z,u)$ is a natural generalization of the polar derivative $d P(z)+(u-z)P^\prime(z)$ of $P$ with pole $u$, see e.g. \cite{Ma}. (Also compare this family to the polynomial appearing after the relation \eqref{eq:saddle-points1} above). Finally, if the parameter $\al$ runs over the whole real line, the union of all critical values of \eqref{eq:criticalValues} appears to form $d-1$ hyperbola-like branches  in $\mathbb{C}$ that  interact with $\Upsilon_P$.

% \smallskip
% \noindent
% {\bf 3.}  Numerical experiments similar to the one shown in Fig.~\ref{fig:polShadowExamples} motivate the following guess. 
%
% \begin{conjecture} For any polynomial $P$ of degree at least $2$,
%
%
%\noindent
%{\rm(i)} the union of the supports $\Upsilon_P=\cup_{\al\in[0,d]} S_{\al,P}$ is a domain inside the convex hull  of the roots of $P$. The boundary of $\Upsilon_P$ is a subset of the curve formed by the family of branch points of $\Ga_{\al,P}$ depending on $\al$;
%
%
%\noindent
%{\rm(ii)} $\Upsilon_P$ is a concave domain.
%
%\end{conjecture}
%
\smallskip
\noindent
{\bf 4.} Let $P$ be a cubic strongly generic polynomial with non-collinear zeros, and let $h=h(\W)$ denote the left-hand side of equation \eqref{eq:algebraicDiffEq2}. If we solve the quartic equation $\mathrm{Resultant}(h, h')/P = 0$ in $z$ for $\al = 0,1,3$, the twelve solutions that arise are various triangle centers associated with the triangle $T_P$ in $\bC$ whose vertices are the roots of $P$. The distinct among these twelve points are: the center of mass of $T_P$, the roots of $P$ and $P^\prime$, the first and second isodynamic points of $T_P$ (denoted by $\mathcal{I}_1$ and $\mathcal{I}_2$, respectively), and an additional point that we denote $\mathcal{A}$. (Here $\mathcal{A}$ is the point $X(26613)$ in the Encyclopedia of Triangle Centers, see \cite{Ki,HaShSh}.) Numerical experiments indicate that the support of the asymptotic root-counting measure $\mu_{1,P}$ (in the notation of Theorem~\ref{th:Cauchy}) has non-obvious connections to some of these points, and is either a tree with three edges that have a vertex in common, or two disjoint edges; see Fig.~\ref{fig:isodynamicPoints}.
\begin{figure}[htp]
\begin{center}
\includegraphics[width=.48\textwidth]{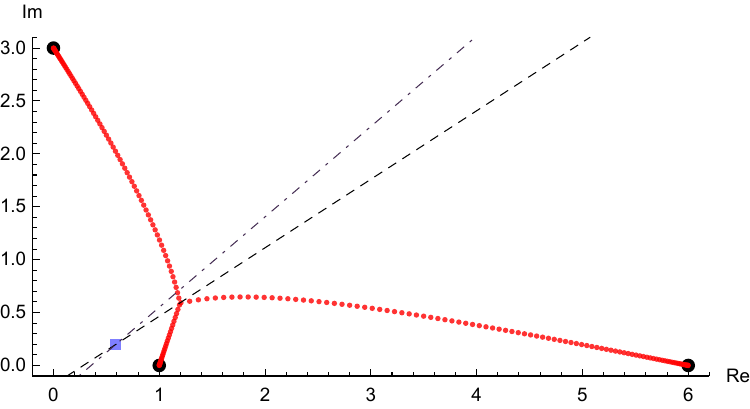}\hfill
\includegraphics[width=.48\textwidth]{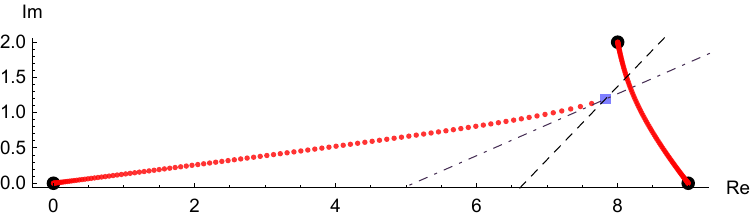}
\end{center}
\caption{The zeros of $\R_{100,100,P}(z)$ (shown by small red dots). Here, the square is the branch point $\mathcal{A}$, and the two lines are $\lline{A \mathcal{I}_1}$ and $\lline{A \mathcal{I}_2}$, respectively. Note that one of the two lines  seems to always pass through the vertex of degree 3 in the tree formed by the support of $\mu_{1,P}$ whenever it exists.}\label{fig:isodynamicPoints}
\end{figure}

More generally, our numerical experiments support the following guess.

\begin{conjecture}\label{conj:forest}
In the notation of Theorem~\ref{th:Cauchy}, for $0<\al<\deg P$,  the support of $\mu_{\al,P}$ is an embedded graph in $\mathbb{C}$ without cycles, i.e., it is a  forest.
\end{conjecture}

%\smallskip 
%\noindent
%{\bf 5.} We conjecture that some of the local behavior of the zeros of $\R_{n,n,P}$ can be characterized in terms of the zeros of the $n$-th Legendre polynomial $L_n$ as follows:

%\begin{conjecture}\label{conj:legendre2}
%Let $P := (z-1)(z+1)\prod_{j=1}^{d}(z-\kappa\alpha_j)$ be a polynomial, where $d\in\mathbb{N},\,\kappa\in\mathbb{R}$, and $\alpha_j\in\mathbb{C}\setminus\{0\},\,j=1,\dotsc,d$. Then, for any $n\in\mathbb{N},\,Z(L_n)\subseteq Z(\R_{n,n,P})$ when $\kappa\to\infty$.
%\end{conjecture}

%Additionally, we sketch an explanation of some of the apparent global behavior of the zeros. Let $P$ be a polynomial with all of its $d_1+d_2+\dotsc+d_r$ zeros clustered into disjoint disks $D_1, D_2, \dotsc, D_r$ centered at the points $z_1, z_2, \dotsc, z_r$, respectively. Furthermore, let $\mu_f$ denote the root-counting measure of the polynomial $f(z)$, and let $U := \mathbb{C}\setminus\{D_1\cup\dotsc\cup D_r\}$. Then we conjecture that $\mathrm{supp}(\mu_\R)\cap U$ with $\R=\R_{[\al n],n,P}$ can be approximated by $\mathrm{supp}(\mu_{\widetilde \R}) \cap U$ with $\widetilde \R=\R_{[\al n],n,W}$ as $n\to\infty$, where $W := (z-z_1)^{d_1}(z-z_2)^{d_2}\dotsm(z-z_r)^{d_r}$. Here, the error in the approximation seems to be inversely proportional to the distances between the disks, in the sense that if the disks (and clusters of zeros) are moved apart while retaining fixed sizes, this leads to a more pointlike concentration of zeros as seen by an observer above $\mathbb{C}$ who maintains a fixed viewing angle between the outermost zeros.

\smallskip
\noindent
{\bf 5.} The Rodrigues descendants $\R_{n,n,P}(z)$ satisfy multiple orthogonal conditions in the following sense.

\begin{lemma}\label{prop:multipleOrthogonal}
Assume that $P(z) = (z-z_1)\dotsm (z-z_d)$ has only simple roots and let $\gamma$ be a path connecting $z_i$ and $z_j$, $1\le i < j\le d$. Then
\begin{equation}\label{eq:multipleOrthogonal}
I := \int_\gamma z^{k} \, \R_{n,n,P}(z)\,\diff z = 0
\end{equation}
where $k=0,1,\dots, n-1$.
\end{lemma}
\begin{proof}
Follows from integration by parts $n$ times.
\end{proof}
Taking $d-1$ homologically non-equivalent different paths $\ga_1, \ga_2, \dots, \ga_{d-1}$ among paths connecting the roots of $P$ 
leads to the multiple orthogonality conditions. For $d>2$, the sequence $\{\R_{n,n,P}(z)\}$ itself does not satisfy any linear recurrence of finite length. On the other hand, fixing a  system $\ga_1, \ga_2, \dots, \ga_{d-1}$ of paths as above, one can introduce the  family of  (type II) multiple orthogonal polynomials indexed by $\mathbf n=(n_1,n_2,\dots, n_{d-1})$ where the polynomial $\R_{\mathbf n}(z)$ has degree $n_1+n_2+\dots+n_{d-1}$ and satisfies the system of orthogonality relations given by 
$$\int_{\gamma_j} z^{k_j} \, \R_{\mathbf n}(z)\,\diff z = 0,\; k_j=0,1,\dots, n_j-1\; \text{and}\; j=1,2,\dots, d-1.$$ (One can check that the above system determines $\R_{\mathbf n}(z)$ up to a scalar factor). Obviously,  $\R_{n,n,P}(z)$ are the special cases of more general polynomials  $\R_{\mathbf n}(z)$ corresponding to $\mathbf n=(n,n,\dots, n)$. The multi-indexed family $\{\R_{\mathbf n}(z)\}$  satisfies a finite recurrence relation of length $d+1$, see \cite{VA}.

\smallskip
\noindent
{\bf 6.} Our final remark concerns Theorem~\ref{th:main}.

\medskip
\begin{conjecture}\label{conj:msin} Under the assumptions of Theorem~\ref{th:main} the signed measure whose existence is proven in this result is unique. \end{conjecture}

%\medskip
%(+++ Equation (\ref{eq:partialDSols}) is probably related to the vanishing of the zeros of the $\frac{R(h,h')}{\al P}$ polynomials in Figure \ref{fig:quadAmoebas} somehow... +++)

%\subsection{Conjecture on number of zeros}

%\medskip
%\noindent
%{\bf 7.} Conjecture on number of zeros. 

%\begin{conjecture}\label{conj:RPowDerivZeros} Let $R = P/Q$, where $\deg{P}\ge 1$ and $\vert Z(P)\cup Z(Q)\vert = \deg{P} + \deg{Q}$. Then the number of zeros of $(R^n)^{(m)}$, counted with multiplicity, is +++asymptotically?+++
%\begin{equation*}\label{eq:RPowDerivZeros}
%\begin{aligned}
% &(\deg{P})n + (\deg{Q}-1)m,\hspace{28pt} if \deg{P} < \deg{Q}, \\
% &(\deg{P})n + (\deg{Q}-1)m,\hspace{28pt} if \deg{P} \ge \deg{Q}\:\:and\:\:0\le m\le (\deg{P}-\deg{Q})n, \\
% &(\deg{Q})(m+n)-(m+1),\hspace{23.8pt} if \deg{P} \ge \deg{Q}\:\:and\:\:m > (\deg{P}-\deg{Q})n.
%\end{aligned}
%\end{equation*}
%\end{conjecture}

%\begin{remark} {\rm Note that the union of all zeros for $m=0,1,\dotsc,(\deg{P}-\deg{Q})n$ in the second case generalizes the polynomial amoebas (?) in Figure \ref{fig:quadAmoebas} naturally to amoebas for rational functions $P/Q$ such that $\deg{P}\ge\deg{Q}$.}+++expand?+++

%Cases 1 and 3 seem to yield zeros that disappear toward $\infty$. This highlights the restriction to our theorem, since the proof uses the boundedness of the zeros in case 2.

%(++++
%What about the independence of $\deg{P}$ in the third case??? Additionally, if case 2 contains cycles, case 3 seems to correspond to a removal of at least one of these cycles. +++)
%\end{remark}

\end{document}